\documentclass[11pt]{article}

\usepackage[latin1]{inputenc}
\usepackage{t1enc}
\usepackage{amsmath, amssymb, amsfonts, amsthm, amsopn,epsfig}
\usepackage[none]{hyphenat}
\usepackage{float}
\usepackage{graphicx}
\usepackage{caption}
\usepackage[toc,page]{appendix}
\usepackage{epstopdf}
\usepackage{etoolbox}
\usepackage[colorlinks=true,backref]{hyperref}
\usepackage{dsfont}
\usepackage{tikz}
\usepackage{array}
\usepackage{rotating} 
\usetikzlibrary{calc,positioning,arrows,shapes,decorations.markings,decorations.pathreplacing,matrix,patterns}
\tikzstyle{vertex}=[circle,draw=black,fill=black,inner sep=0,minimum size=3pt,text=white,font=\footnotesize]
\hypersetup{
	colorlinks=true,
	linkcolor=blue,
	filecolor=magenta,
	urlcolor=cyan,
	citecolor=blue
}
\usepackage{cleveref}
\usepackage[top=21mm, bottom=22mm, left=21mm, right=21mm]{geometry}
\usepackage{comment}

\theoremstyle{plain}
\newtheorem{theorem}{Theorem}[section]

\newtheorem{corollary}[theorem]{Corollary}
\newtheorem{claim}[theorem]{Claim}
\newtheorem{lemma}[theorem]{Lemma}

\newtheorem{question}[theorem]{Question}

\theoremstyle{definition}

\DeclareMathOperator{\disc}{disc}

\DeclareMathOperator{\rank}{rank}
\DeclareMathOperator{\surp}{surp}
\DeclareMathOperator{\trace}{tr}
\DeclareMathOperator{\bw}{bw}
\DeclareMathOperator{\dfc}{dfc}

\newcommand{\cF}{\mathcal{F}}

\newcommand{\bR}{\mathbb{R}}
\newcommand{\bZ}{\mathbb{Z}}

\newcommand{\eps}{\varepsilon}

\DeclareMathOperator{\Cay}{Cay}

\newcommand{\hide}[1]{}

\title{From small eigenvalues to large cuts, and Chowla's cosine problem}
\author{Zhihan Jin\thanks{Department of Mathematics, ETH Z\"urich, Switzerland. Email: {\tt \{zhihan.jin, aleksa.milojevic\}@math.ethz.ch}. Research supported in part by SNSF grant 200021-228014.}
\and Aleksa Milojevi\'c\footnotemark[1]
\and 
Istv\'an Tomon\thanks{Ume\r{a} University, \emph{e-mail}: \texttt{istvantomon@gmail.com}, Research supported in part by the Swedish Research Council grant VR 2023-03375.}
\and
Shengtong Zhang\thanks{Stanford University, \emph{e-mail}: \texttt{stzh1555@stanford.edu}. This work was partially supported by the National Science Foundation under Grant No. DMS-1928930, while the author was in residence at the Simons Laufer Mathematical Sciences Institute in Berkeley, California, during the Spring 2025 semester. This work was also partially supported by NSF Award DMS-2154129.}
}

\date{}

\begin{document}

\maketitle

\begin{abstract}
    \noindent
    We prove that every graph with average degree $d$ and smallest adjacency eigenvalue $|\lambda_n|\leq d^{\gamma}$ contains a clique of size $d^{1-O(\gamma)}$. A simple corollary of this yields  the first polynomial bound for Chowla's cosine problem (1965): for every finite set $A\subseteq \mathbb{Z}_{>0}$, the minimum of the cosine polynomial satisfies $$\min_{x\in [0, 2\pi]}\sum_{a\in A}\cos(ax)\leq -|A|^{1/10-o(1)}.$$
    Another application makes significant progress on the problem of MaxCut in $H$-free graphs initiated by Erd\H{o}s and Lov\'asz in the 1970's. We show that every $m$-edge graph with no clique of size $m^{1/2-\delta}$ has a cut of size at least $m/2+m^{1/2+\eps}$ for some $\eps=\eps(\delta)>0$. 
\end{abstract}

\section{Introduction}

A central theme in spectral graph theory is understanding the interplay between the structural properties of a graph and its spectrum. Here and throughout, the spectrum of the graph refers to the set of eigenvalues of its adjacency matrix. Some of the most prominent results highlighting the connection between eigenvalues and structural properties are the expander mixing lemma \cite{AC88}, which relates eigenvalues to pseudorandomness, and Hoffman's bound \cite{H21}, which uses the smallest eigenvalue to bound the independence number. 

In this paper, we study graphs whose smallest eigenvalue is not very negative.
The prime examples of such graphs are complete graphs, where every eigenvalue is at least $-1$.
In fact, all graphs that are close to disjoint unions of cliques have their smallest eigenvalue small in absolute value. 
Our main technical contribution is a converse to this observation: we prove that even the mild restriction $|\lambda_n|\leq n^{1/4-o(1)}$ on the smallest eigenvalue $\lambda_n$ of an $n$-vertex graph forces the graph to be close to a disjoint union of cliques. 
Moreover, the exponent $1/4$ is best possible.

We further obtain an analogue of this result for sparse graphs. We show that graphs of average degree $d$ and smallest eigenvalue $|\lambda_n|\leq d^{\gamma}$ contain cliques of size $d^{1-O(\gamma)}$. Perhaps surprisingly, this purely graph-theoretic statement has a powerful application to Chowla's cosine problem \cite{C65}, a classical problem from harmonic analysis. It implies that if $A$ is a finite set of positive integers, then the cosine polynomial
$$f(x)=\sum_{a\in A}\cos(ax)$$
attains values as small as $-|A|^{\Omega(1)}$. A more detailed discussion of Chowla's cosine problem is presented in Section~\ref{subsec:intro_chowla}.

Our methods are also applicable in the study of graphs with small maximum cut, leading to substantial progress on a celebrated conjecture of Alon, Bollob\'as, Krivelevich, and Sudakov \cite{ABKS}. This is a central problem in the study of the maximum cut in $H$-free graphs, a topic initiated by Erd\H{o}s and Lov\'asz \cite{erdos} in the 1970's. The conjecture asserts that for any fixed graph $H$, every $H$-free graph $G$ with $m$ edges has a cut of size at least $m/2+m^{3/4+\eps_H}$ for some $\eps_H>0$. 
Over the decades, even the weaker question of whether one can guarantee a cut of size at least $m/2+m^{1/2+\eps}$, for some fixed $\eps>0$, remained wide open.
% It was a long-standing problem whether one can guarantee any $\eps>0$ such that $\eps_H>\eps$. 
We not only resolve this, but also prove a stronger statement: if $G$ contains no clique of size $m^{1/2-\delta}$, then $G$ admits a cut of size at least $$m/2+m^{1/2+\eps}$$ for some $\eps=\eps(\delta)>0$. A more detailed discussion of the maximum cut problem is provided in Section~\ref{subsec:intro_maxcut}.

Finally, our results have further implications about the second eigenvalue. A classical theorem of Alon and Boppana~\cite{alon-boppana} lower bounds the second-largest eigenvalue $\lambda_2$ of $d$-regular $n$-vertex graphs by $$\lambda_2\geq 2\sqrt{d-1}\left(1-\frac{1}{\lfloor D/2\rfloor}\right),$$ where $D$ is the diameter of the graph. This estimate becomes trivial for $D\leq 3$, which can already occur for $d \gg n^{1/3}$. We extend the Alon-Boppana bound to dense graphs: we show that if a regular graph is far from being a Tur\'an graph, then $$\lambda_2\geq n^{1/4-o(1)}$$ and the exponent $1/4$ is best possible. 

Our proofs introduce novel spectral and linear-algebraic techniques based on subspace compressions of matrices and the use of Hadamard products, which may be of independent interest.

\subsection{Chowla's cosine problem}\label{subsec:intro_chowla}

In 1948, in the study of certain Dedekind zeta functions, Ankeny and Chowla came across the following question (see \cite{C52}): is it true that for every $K>0$ and sufficiently large $n>0$, if $a_1, \dots, a_n$ are distinct positive integers, then the minimum of the function $f(x)=\cos(a_1 x)+\cdots+\cos (a_n x)$ is less than $-K$? Soon thereafter, Uchiyama and Uchiyama \cite{UU60} answered this question affirmatively, but with poor quantitative dependencies, by observing the connection to Cohen's work \cite{C60} on Littlewood's $L_1$-problem. This problem asks to show that for each $n$-element set $A\subseteq \mathbb{Z}$, the $L_1$-norm of the Fourier transform of $\mathds{1}_A$ is bounded below by $\Omega(\log n)$, i.e. \[\|\widehat{\mathds{1}_A}\|_1 = \int_{0}^{1}\Big|\sum_{a\in A}e^{2\pi i a x}\Big|\,d\,x=\Omega(\log n).\]
Any lower bound on Littlewood's $L_1$-problem gives a comparable upper bound for Chowla's cosine problem, see \cite{R73} for a detailed derivation. 

In 1965, Chowla \cite{C65} revisited the problem and made a more precise conjecture, today known as \textit{Chowla's cosine problem}: show that for an $n$-element set $A$ of positive integers,
$$\min_{x\in [0, 2\pi]} f(x)=\min_{x\in [0, 2\pi]} \sum_{a\in A}\cos(a x) \leq -\Omega(\sqrt{n}).$$ In case $A$ can be written as $A=B-B$, where $B$ is a Sidon set, one has $\min_{x\in [0, 2\pi]} f(x)=-\Theta(\sqrt{n})$ (see \cite{M19} for a detailed proof), so if the conjecture is true, the bound $-\Omega(\sqrt{n})$ is the best possible. 

The subsequent decades saw a persistent interest in this problem, and the bounds of Uchiyama and Uchiyama were improved by Roth \cite{R73} in 1973, who showed that $\min_{x} f(x)\leq -\Omega(\sqrt{\log n/\log\log n})$. Then the resolution of the Littlewood $L_1$-problem in 1981 by Konyagin \cite{K81} and McGehee, Pigno and Smith \cite{MPS81} improved this to $\min_{x}f(x)\leq -\Omega(\log n)$. It was Bourgain \cite{B86} who first broke this logarithmic barrier, and his method was further refined in 2004 by Ruzsa \cite{R04} to give the previously best known bound $\min_x f(x)\leq -\exp\big(\Omega(\sqrt{\log n})\big)$. Chowla's cosine problem is also highlighted as problem number 81 on Green's 100 problems list \cite{green}. Here, we give the first polynomial bound.

\begin{theorem}\label{thm:MAIN_CHOWLA}
For any finite set $A$ of positive integers, there exists $x\in [0,2\pi]$ such that 
$$\sum_{a\in A}\cos(ax)\leq -|A|^{1/10-o(1)}.$$
\end{theorem}

 We now say a few words about the proof of Theorem~\ref{thm:MAIN_CHOWLA}. The key ingredient of the proof is the following graph-theoretic result.

\begin{theorem}\label{thm:MAIN_SE3}
For every $\gamma\in (0,1/10)$, the following holds for every sufficiently large $d$. Let $G$ be a graph of average degree $d$ and assume that the smallest eigenvalue $\lambda_n$ of $G$ satisfies $|\lambda_n|\leq d^{\gamma}$. Then $G$ contains a clique of size at least $d^{1-4\gamma}$.
\end{theorem}

We embed $A$ into the group $\mathbb{Z}/n\mathbb{Z}$ for a sufficiently large prime $n$ and consider the Cayley graph $G=\Cay(\mathbb{Z}/n\mathbb{Z}, A\cup -A)$. It is well-known that the eigenvalues of Cayley graphs correspond to the Fourier coefficients of the generating set, and thus the smallest eigenvalue $\lambda_n$ satisfies 
$$\lambda_n=\sum_{a\in A\cup -A}\exp(2\pi i ak/n)=2\sum_{a\in A}\cos(2\pi a k/n)$$ for some $k\in \mathbb{Z}/n\mathbb{Z}$. Hence, $\frac{1}{2}\lambda_n\geq \min_x f(x)$, so a lower bound on $\min_x f(x)$ yields an upper bound on $|\lambda_n|$. Then by Theorem \ref{thm:MAIN_SE3}, $G$ contains very large cliques. However, by appealing to the transitive symmetry of the Cayley graph, we show that the existence of large cliques forces large $|\lambda_n|$.

\medskip

We conclude this section by noting that cosine polynomials are the subject of a number of other interesting problems. An old problem of Littlewood \cite{L68} asks to study the minimum number of zeros of $f(x)=\sum_{a\in A} \cos(a x)$ in the interval $[0, 2\pi]$, where $A$ is a set of $n$ positive integers. Although Littlewood conjectured that this number is linear in $n$, Borwein, Erd\'elyi, Ferguson and Lockhart \cite{BEFL08} showed that there are integers $a_1, \dots, a_n$ such that $f(x)$ has at most $n^{5/6+o(1)}$ zeros. This result was later improved to $O((n\log n)^{2/3})$ by Ju\v skevi\v cius and Sahasrabudhe \cite{JS21} and, independently, by Konyagin \cite{K20}. A complementary bound has been proven by Sahasrabudhe \cite{S19} and Erd\'elyi \cite{E16, E20}, who showed that $f(x)$ always has at least $(\log\log\log n)^{1/2-\eps}$ roots, which was later improved to $(\log\log n)^{1-o(1)}$ by Bedert \cite{B24}. Another problem about trigonometric polynomials, asked by Littlewood \cite{L66} and Erd\H{o}s \cite{E57}, concerns the existence of ``flat'' polynomials, i.e. polynomials $f(z)=\sum_{k=0}^n \eps_k z^k$ with coefficients $\eps_k=\pm 1$ such that $|f(z)|=\Theta(\sqrt{n})$ for all $|z|=1$. Writing $z=\cos\theta+i\sin\theta$ shows that this problem is about controlling the size of trigonometric polynomials  $f(z)=\sum_{k=0}^n \eps_k (\cos k\theta+i\sin k\theta)$. The existence of such functions was proved only very recently by Balister, Bollob\'as, Morris, Sahasrabudhe, and Tiba \cite{BBMST24}.

In \Cref{sec: chowla for groups}, we discuss further extensions of Chowla's problem in arbitrary finite groups. 
The proof of Theorem \ref{thm:MAIN_CHOWLA} is presented in Section \ref{sect:chowla}, and the proof of Theorem \ref{thm:MAIN_SE3} is presented in Section \ref{sect:cliques}.

\medskip

\noindent
\textbf{Note added after publication.} Very recently, Bedert \cite{B25} achieved a result similar to our \Cref{thm:MAIN_CHOWLA}, proving that for every $A\subset \mathbb{Z}_{>0}$, one has $\min_x \sum_{a\in A}\cos(ax)\leq -\Omega(|A|^{1/7-o(1)})$. Interestingly, the methods of Bedert are fundamentally different from ours. The proof in \cite{B25} is Fourier-analytic, compared to our spectral and graph-theoretic approach.

\subsection{Maximum Cut}\label{subsec:intro_maxcut}

Given a graph $G$, a \emph{cut} in $G$ is a partition $(U,V)$ of the vertex set together with all the edges having exactly one endpoint in both parts. The \emph{size} of the cut is the number of its edges. The \emph{maximum cut} (or \emph{MaxCut}) of $G$, denoted by $\mbox{mc}(G)$, is the largest possible size of a cut. The MaxCut is among the most extensively studied graph parameters, lying at the intersection of theoretical computer science \cite{EMS,GW95,KKMO}, extremal combinatorics \cite{AlonMaxCut,ABKS,Edwards1,erdos} and probabilistic graph theory \cite{CGHS,CLMS,DMS}. In theoretical computer science, one is usually interested in efficient approximations of the maximum cut, while in extremal combinatorics the emphasis is on establishing sharp bounds in terms of various graph parameters, such as the number of vertices or edges. 

A simple probabilistic argument shows that every graph with $m$ edges has a cut of size at least $m/2$. Indeed, if one chooses a cut uniformly at random, its expected size is exactly $m/2$. The constant $1/2$ is best possible in general, and therefore it is often more natural to measure the \emph{surplus} of a graph $G$, defined as $\surp(G)=\mbox{mc}(G)-m/2$. A fundamental result of Edwards \cite{Edwards1,Edwards2} asserts that any graph $G$ with $m$ edges satisfies $\mbox{mc}(G)\geq \frac{m}{2}+\frac{\sqrt{8m+1}-1}{8}$ or, equivalently, that $\surp(G)\geq \frac{\sqrt{8m+1}-1}{8}$, which is sharp when $G$ is a clique on an odd number of vertices. 

In general, if $G$ is a disjoint union of constantly many cliques, then its maximum cut is $m/2+O(\sqrt{m})$. This naturally raises the question: can one improve this bound if $G$ is far from being a disjoint union of cliques? One way to ensure that a graph is far from a disjoint union of cliques is to forbid some fixed graph $H$ as a subgraph. The study of the MaxCut, and in particular the surplus, in $H$-free graphs was initiated by Erd\H{o}s and Lov\'asz in the 1970's \cite{erdos}. A landmark result in the area is due to Alon \cite{AlonMaxCut}, who proved that if a graph $G$ has $m$ edges and contains no triangles, then $\surp(G)=\Omega(m^{4/5})$, and this bound is tight. Two natural generalizations of this result are to consider graphs without short cycles, and graphs that avoid complete graphs $K_r$. 

The surplus in graphs without short cycles has been studied extensively \cite{ABKS,AKS05,BJS,GJS}, with tight bounds obtained in \cite{BJS,GJS}. 
On the other hand, determining the minimum surplus in $K_r$-free graphs is much more difficult. 
Alon, Bollob\'as, Krivelevich, and Sudakov \cite{ABKS} proved that for every $r$, there exists $\eps_r>0$ such that every $K_r$-free graph has surplus at least $m^{1/2+\eps_r}$. 
This was improved by Carlson, Kolla, Li, Mani, Sudakov, and Trevisan \cite{CKLMST}, and further strengthened by Glock, Janzer, and Sudakov \cite{GJS}, who proved $\surp(G)=\Omega_r\big(m^{\frac{1}{2}+\frac{3}{4r-2}}\big)$. Nevertheless, these bounds seem far from optimal: Alon, Bollob\'as, Krivelevich and Sudakov conjectured in \cite{ABKS} that the true bound should be $\surp(G)\geq m^{3/4+\eps_r}$ for some $\eps_r>0$. This conjecture remains open. For many years, it was a tantalizing open problem to find any absolute constant $\eps>0$, independent of $r$, such that every $K_r$-free graph has surplus $\Omega_r(m^{1/2+\eps})$. Glock, Janzer and Sudakov \cite{GJS} write \textit{``Arguably, the main open problem is to decide whether there exists a positive absolute constant $\eps$ such that any $K_r$-free graph with $m$ edges has surplus $\Omega_r(m^{1/2+\eps})$.''} Our next main result not only resolves this problem, but also shows that one can guarantee such a surplus by forbidding extremely large cliques as well.

\begin{theorem}\label{thm:MAIN_MC1}
For every $\delta>0$ there exists $\eps>0$ such that the following holds for every sufficiently large $m$. Let $G$ be a graph with $m$ edges such that $G$ contains no clique of size $m^{1/2-\delta}$. Then $G$ has a cut of size at least $\frac{m}{2}+m^{1/2+\eps}$. 
\end{theorem}

In the very extreme case, Balla, Hambardzumyan, and Tomon \cite{BHT} recently showed that graphs with clique number $o(\sqrt{m})$ already have surplus $\omega(m^{1/2})$. Despite the similarity between this result and the previous theorem, there is no implication between the two due to the hidden dependencies. The methods achieving these results are also very different, despite both being algebraic in nature. 

There is a close relationship between the MaxCut of a graph $G$ and its smallest eigenvalue. It is well known that $\surp(G)\leq |\lambda_n|n$ (see e.g. Claim \ref{claim:surplus_and_lambdan} for a short proof). %For many algebraically defined graph families, we also have $\surp(G)=\Theta(|\lambda_n| n)$, but in general these quantities can be far apart. 
A good way to think about the surplus is as a robust version of the smallest eigenvalue: in many natural cases $\surp(G)=\Theta(|\lambda_n|n)$, but $\surp(G)$ is much less sensitive to local modifications and harder to study.

The proof of Theorem \ref{thm:MAIN_MC1} is presented in Section \ref{sect:cliques}. 
In \Cref{sect:MaxCut_stability}, we discuss a variant of Theorem \ref{thm:MAIN_MC1} which shows that graphs that are far from the disjoint union of cliques also have large surplus.

\subsection{Smallest eigenvalue}

A central topic of spectral graph theory is understanding the structure of graphs whose adjacency matrix has large (i.e. not very negative) smallest eigenvalue. Let $G$ be an $n$-vertex graph and let $\lambda_n$ denote the smallest eigenvalue of its adjacency matrix. Probably the best-known theorem in spectral graph theory involving the smallest eigenvalue is the celebrated Hoffman bound (see e.g. \cite{H21}), which states that $\lambda_n$ controls the independence number of the graph. In particular, if $G$ is an $n$-vertex $d$-regular graph, then $\alpha(G)\leq \frac{n|\lambda_n|}{|\lambda_n|+d}$. The Expander Mixing Lemma is of equal importance, stating that the maximum of $|\lambda_n|$ and the second-largest eigenvalue $\lambda_2$ determines the expansion and mixing properties of the graph \cite{AC88}. Moreover, as we mentioned before, $\lambda_n$ controls the maximum cut.

 A simple consequence of the Cauchy's interlacing theorem is that if $G$ is non-empty, then $\lambda_n\leq -1$ with equality if and only if $G$ is the disjoint union of cliques. In the 1970's, Cameron, Goethals, Seidel, and Shult \cite{CGSS} gave a complete characterization of graphs satisfying $|\lambda_n|\leq 2$, which are exactly generalized line graphs and some sporadic examples with at most 36 vertices. Koolen, Yang, and Yang \cite{KYY} obtained a partial characterization in the case $|\lambda_n|\leq 3$ using integral lattices. 

Beyond these specific values, much less is known. Hoffman \cite{H73} studied structural properties of graphs with $\lambda_n\geq -\lambda$, for some fixed constant $\lambda$, and his work was extended by Kim, Koolen, and Yang \cite{KKY}, who proved the following structure theorem for regular graphs satisfying $|\lambda_n|\leq \lambda$: one can find dense induced subgraphs $Q_1,\dots,Q_c$ in $G$ such that each vertex lies in at most $\lambda$ of $Q_1,\dots,Q_c$, and almost all edges are covered by the union of $Q_1,\dots,Q_c$. However, the proof of this is based on certain forbidden subgraph characterizations and Ramsey-theoretic arguments, and the results are no longer meaningful if $\lambda$ grows faster than polylogarithmically in $n$. 

For highly structured graphs, such as strongly regular graphs (SRGs), it is known \cite{N79} that if $|\lambda_n|$ is at most a small polynomial of the average degree, then the graph belongs to one of two special families (see also \cite{KLMP}). However, these results rely on the highly structured nature of SRGs. We refer the interested reader to the survey of Koolen, Cao, and Yang \cite{KCY} for a general overview of the topic.

Many of these results show that the property of having small $|\lambda_n|$ and the existence of large trivial substructures, such as cliques, are interconnected. However, such results were previously only known when $|\lambda_n|$ is bounded by a constant, or growing very slowly with $n$. We prove that this phenomenon already starts to appear when $|\lambda_n|<n^{1/4-o(1)}$, and we show that graphs with smallest eigenvalue below this threshold are close to a trivial structure: the disjoint union of cliques. The exponent $1/4$ is also sharp, a celebrated construction of de Caen~\cite{deCaen} related to equiangular lines provides a graph with smallest eigenvalue $|\lambda_n|=\Theta(n^{1/4})$ which is far from the disjoint union of cliques. We say that an $n$-vertex graph $G$ is \emph{$\delta$-close} to some family of graphs $\cF$ if one can change at most $\delta n^2$ edges to non-edges and vice versa to transform $G$ into a member of $\cF$. 

\begin{theorem}\label{thm:MAIN_SE1}
Let $\gamma\in (0, 1/4)$, $\delta>0$, and let $n$ be sufficiently large. If $G$ is an $n$-vertex graph with $|\lambda_n|\leq n^{\gamma}$, then $G$ is $\delta$-close to the vertex-disjoint union of cliques. 
\end{theorem}

\noindent
The previous theorem only guarantees $o(1)$-closeness in the case $|\lambda_n|\le n^{1/4-o(1)}$. Note that one cannot expect a substantially stronger result than $o(1)$-closeness, as illustrated by the following example. Suppose $G$ is the disjoint union of $n^{4\varepsilon}$ copies of the graph constructed by de Caen~\cite{deCaen}, each of size $n^{1-4\varepsilon}$. Then $G$ satisfies $|\lambda_n|\le n^{1/4-\varepsilon}$ and, as mentioned above, one must add or remove $\Omega(n^{4\varepsilon}(n^{1-4\varepsilon})^2)$ edges to make $G$ a disjoint union of cliques. Therefore, one cannot hope to prove that $G$ is closer than $n^{-4\varepsilon}$ to a union of cliques. By imposing a slightly stronger bound on $|\lambda_n|$, we can indeed establish polynomial proximity to a disjoint union of cliques.

\begin{theorem}\label{thm:MAIN_SE2}
    For every $\gamma\in (0, 1/6)$, there exists $\alpha>0$ such that for every sufficiently large $n$ we have the following. 
    If $G$ is an $n$-vertex graph with $|\lambda_n|\leq n^{\gamma}$, then $G$ is $n^{-\alpha}$-close to the vertex-disjoint union of cliques.
\end{theorem}

While these results provide good structural understanding of somewhat dense graphs with small $|\lambda_n|$, they are no longer meaningful for sparse graphs. 
In fact, it is impossible to formulate any reasonable extension of the previous theorem even for moderately sparse graphs, as the following example show. 
All line graphs satisfy $|\lambda_n|\leq 2$, but the line graph of the complete graph $K_t$ has $\Theta(t^2)$ vertices and $m=\Theta(t^3)$ edges and it is not possible to add or remove $o(m)$ edges to get a disjoint union of cliques.

Despite this, we recall that \Cref{thm:MAIN_SE3} shows that large cliques, of size $d/|\lambda_n|^4$, emerge in graphs of any sparsity, whenever $|\lambda_n|\leq d^{\gamma}$ for $\gamma\in (0, 1/10)$. This extends the result of Yang and Koolen \cite{YK}, who showed that if $d$ is exponentially large compared to $|\lambda_n|$, then $G$ must contain a clique of size $d/|\lambda_n|^3$. This suggests that such graphs might be close to the blow-up of much smaller graphs, and shows that trivial structures start to appear at any sparsity, assuming $|\lambda_n|$ is sufficiently small. 
We prove \Cref{thm:MAIN_SE1} in \Cref{sect:dens2} and \Cref{thm:MAIN_SE2} in \Cref{sect:stab}.

\subsection{Alon--Boppana theorem} \label{sec: intro alon--boppana}

The Alon--Boppana theorem \cite{alon-boppana} is a cornerstone result of spectral graph theory. It states that if $G$ is an $n$-vertex $d$-regular graph, then the second-largest eigenvalue $\lambda_2$ of the adjacency matrix is at least $$\lambda_2\geq 2\sqrt{d-1}-o_n(1).$$  In its precise formulation, the Alon--Boppana theorem states that if $D$ is the diameter of $G$, then $\lambda_2\geq 2\sqrt{d-1}-\frac{2\sqrt{d-1}}{\lfloor D/2\rfloor}.$
In particular, if $D\rightarrow \infty$, which is satisfied in the case $d=n^{o(1)}$, one gets the former lower bound. 
For fixed $d$, families of graphs satisfying $\max\{|\lambda_n|,\lambda_2\}\leq 2\sqrt{d-1}$ are called Ramanujan graphs, and their existence is known for many different values of $d$ \cite{ramanujan}. 
A breakthrough of Friedman \cite{Friedman} shows that random $d$-regular graphs are close to being Ramanujan. Since the spectral gap $d-\lambda_2$ controls the expansion properties of graphs, the Ramanujan graphs are optimal expanders. For this reason, such graphs are of great interest in the design of resilient networks, with countless further applications in theoretical computer science and extremal combinatorics.

In the case where the diameter $D$ is at most three, which can already happen if $d\approx n^{1/3}$, the Alon--Boppana bound is no longer meaningful. Also, one cannot hope for the bound $\lambda_2=\Omega(\sqrt{d})$ to hold unconditionally; for example the complete bipartite graph has $\lambda_2=0$. Recently, a number of authors \cite{Balla21,BRST,Ihringer,RST} studied the second eigenvalue in the case of denser graphs, and uncovered some highly unexpected behavior of its extremal value. In particular, \cite{Balla21,RST} (see \cite{BRST} for a short note) proved that 
$$\lambda_2=\begin{cases}\Omega(d^{1/2}) &\mbox{ if } d\leq n^{2/3},\\
                        \Omega(n/d) &\mbox{ if } d\in [n^{2/3},n^{3/4}],\\
                        \Omega(d^{1/3}) &\mbox{ if } d\in [n^{3/4},(1/2-\eps)n].
            \end{cases}$$
Moreover, these bounds are (essentially) sharp in the first two regimes, and also in case $d=\Omega(n)$ \cite{DHJP}. As we observed earlier, if $d=n/2$, we might have $\lambda_2=0$ by the complete bipartite graph. In general, when $d=(1-1/r)n$ for some positive integer $r$, the \emph{Tur\'an graph} $T_r(n)$ (the complete $r$-partite graph with parts of size $n/r$) is $d$-regular and satisfies $\lambda_2=0$.

However, what happens when $d$ is not of the form $(1-1/r)n$ or $G$ is far from a Tur\'an graph? The methods of \cite{RST} and related papers no longer apply when $d>n/2$, and there are no obvious further obstructions for having large second eigenvalue. In \cite{RST}, it was conjectured that the answer to the second question is $\Omega(n^{1/4})$, which is sharp by the equiangular lines construction of de Caen \cite{deCaen}. Considering complements, Theorem \ref{thm:MAIN_SE1} immediately implies an almost complete solution of this conjecture. If $G$ is a regular graph with second eigenvalue $\lambda_2$, then the complement of $G$ has smallest eigenvalue $-\lambda_2-1$.

\begin{theorem}\label{thm:MAIN_AB}
Let $\gamma\in (0, 1/4)$, $\delta>0$, and let $n$ be sufficiently large with respect to $\gamma,\delta$. If $G$ is an $n$-vertex $d$-regular graph with $\lambda_2\leq n^{\gamma}$, then $G$ is $\delta$-close to a Tur\'an graph. Thus, if $\lambda_2\leq n^{\gamma}$, then $$\frac{d}{n}\in \left\{1-\frac{1}{r}:r\in \mathbb{Z}^+\right\}+[-\delta,\delta].$$ 
\end{theorem}

\section{Proof overview and organization}

First, we outline the proof of \Cref{thm:MAIN_SE1}, which states that if a graph $G$ has smallest eigenvalue $|\lambda_n|\leq n^{\gamma}$ for some $\gamma\in (0, 1/4)$, then $G$ is $\delta$-close to a disjoint union of cliques. Let $A$ be the adjacency matrix of $G$. 
In order to exploit the fact that $A$ is a $0/1$-matrix, we study the identity $A=A\circ A$, where $\circ$ denotes the entry-wise or Hadamard product (so that $(A\circ B)_{ij}=A_{ij}B_{ij}$, see Section \ref{sect:prelim} for formal definitions). Writing $A=\sum_{i=1}^n \lambda_i v_iv_i^T$ for the spectral decomposition, we get that 
\begin{equation}\label{eq: hadamard decomposition}
    \sum_{i=1}^n \lambda_i v_iv_i^T=\sum_{i=1}^n\sum_{j=1}^n \lambda_i\lambda_j (v_i\circ v_j)(v_i\circ v_j)^T.
\end{equation}
But how to use this identity? An instructive case is when $G$ is a Cayley graph of a finite abelian group $(\Gamma,+)$. In this case, the eigenvalues can be indexed by the group elements, and (\ref{eq: hadamard decomposition}) reduces to a clean convolution relation: $\lambda_a = \frac{1}{n}\sum_{b+c=a} \lambda_b\lambda_c$ for all $a \in \Gamma$. 
This identity also follows from special properties of the characters of the group, which are also the eigenvectors of $G$ (see Section 1.4.9 of Brouwer and Haemers \cite{BH12} for further details). 
If $\lambda_n$ is not very negative, we can almost ignore the negative terms in the sum. Hence, roughly speaking, this convolution relation shows that large eigenvalues reinforce each other, i.e. if $\lambda_b,\lambda_c \ge T$, then $\lambda_{b+c} \gtrsim T^2/n$. 
This motivates the definition $S_T=\sum_{\lambda_i\ge T}\lambda_i$, the spectral weight above threshold $T$. 
Summing over all $\lambda_b\ge T$ and $\lambda_c \ge T$, the above observation gives that $$S_{T^2/n} = \sum_{\lambda_a \ge T^2/n} \lambda_a\gtrsim \frac{1}{n}\sum_{\lambda_b,\lambda_c \ge T}\lambda_b\lambda_c = \frac{1}{n}S_T^2 .$$
This heuristic can be converted into a formal argument and generalized to arbitrary graphs, yielding the following curious recursive inequality on the sum of large eigenvalues: for all $T \ge 2|\lambda_n|\sqrt{n}$,
\begin{equation}\label{eq:recursion}
    4nS_{\frac{T^2}{2n}}\geq S_T^2.
\end{equation}
To derive (\ref{eq:recursion}) in general graphs, we use the notion of subspace-compression of matrices. We compress both sides of (\ref{eq: hadamard decomposition}) onto the subspace $W$ spanned by the vectors $v_i\circ v_j$ where $\lambda_i,\lambda_j \ge T$; see \Cref{sect:main_se} for a detailed argument. We then use the recursive inequality (\ref{eq:recursion}) to show that the contribution of small eigenvalues in the quadratic sum of all eigenvalues is negligible; we show this in \Cref{sect:recursion}. But this means that $A$ can be well-approximated in Frobenius norm by a low-rank positive semidefinite matrix. However, this is only possible if $G$ is close to a disjoint union of cliques, which we prove in \Cref{sect:dens2}.

\medskip

\noindent
Now we discuss the proof of Theorem \ref{thm:MAIN_SE2}, which states that if a graph $G$ has $|\lambda_n|\leq n^{\gamma}$ for some $\gamma\in (0, 1/6)$, then $G$ is $n^{-\alpha}$-close to a disjoint union of cliques. The bottleneck in the previous argument is its last step, where we show that if $A$ is well-approximated by a low-rank matrix, then $G$ must be close to a union of cliques. Our argument requires the rank of the approximation to be constant, which we cannot achieve if the graph $G$ is sparse. To overcome this, we first show that either $G$ is already $n^{-\alpha}$-sparse (in which case $G$ is $n^{-\alpha}$-close to the empty graph), or $G$ contains a very large clique.  We then repeatedly pull out large cliques, which gives enough structure to easily conclude the desired result. In order to find large cliques, we use a density-increment strategy, which is divided into three phases. We use $\eps$ to denote a small positive constant depending only on $\gamma$ and $\alpha$.
 \begin{description}
    \item[Phase 1.] Using (\ref{eq:recursion}), we show that $G$ contains an unusually high number of triangles. We count triangles by the cubic sum of eigenvalues, and we argue that this sum is large because most of the mass of the quadratic sum of eigenvalues is concentrated on the few largest eigenvalues. Having many triangles means that we can find a vertex whose neighbourhood is much denser than $G$. We repeat this process until we find an induced subgraph $G_1\subseteq G$ on $n^{1-\eps}$ vertices of positive constant density. This phase of the argument requires $\gamma<1/6$. The details are given in Section \ref{sect:dens1}.

    \item[Phase 2.] Due to Cauchy's interlacing theorem, $G_1$ also lacks very negative eigenvalues. Hence, Theorem \ref{thm:MAIN_SE1} applies to $G_1$, implying that $G_1$ is $o(1)$-close to a disjoint union of cliques. Therefore, using that $G_1$ has positive constant edge density, we show that $G_1$ contains a linear-sized induced subgraph $G_2\subseteq G_1$ of edge density $1-o(1)$. This step is explained in Section \ref{sect:dens2}.

    \item[Phase 3.] For very dense graphs, we employ a new method, inspired by the work of R\"aty, Sudakov and Tomon \cite{RST}. We prove that if $G_3$ is a somewhat regular induced subgraph of $G_2$, then the complement of $G_3$ must have average degree $O(|\lambda_n|^2)$, assuming $|\lambda_n|\ll |V(G_3)|^{1/4}$. Thus by Tur\'an's theorem, $G_3$ contains a clique of size $\Omega(|V(G_3)|/|\lambda_n|^2)$. In order to prove this, we analyse the triple Hadamard product of certain positive semidefinite shifts of the adjacency matrix. This can be found in Section \ref{sect:dens3}.
 
\end{description}
\noindent
We put together all of these ingredients in Section~\ref{sect:stab} to provide the proof of Theorem~\ref{thm:MAIN_SE3}.

Next, we discuss \Cref{thm:MAIN_SE3}, which states that any graph with average degree $d$ and smallest eigenvalue $|\lambda_n|\leq d^{\gamma}$ contains a clique of size $d^{1-O(\gamma)}$. Note that the methods discussed above only apply to somewhat dense graphs, whose average degree is $n^{1-\alpha}$ for some small $\alpha$. Therefore, as the first step in proving \Cref{thm:MAIN_SE3}, we introduce another densification method, allowing us to move to density at least $1/|\lambda_n|$, which will be sufficient assuming $\lambda_n$ is small with respect to the average degree. Then, we apply the previous three densification steps to conclude the proof. More precisely, we prove the following.

\begin{description}
    \item[Phase 0.] If $G$ has average degree $d$, then we show that $G$ contains a subgraph on $d$ vertices of edge density $\Omega(1/|\lambda_n|)$. This follows by picking a vertex $x$ with a set of $d$ neighbours $S$, and then analyzing the inequality $v^TAv\geq \lambda_n \|v\|_2^2$ for an appropriately chosen $v$ with support $\{x\}\cup S$. This can be found in \Cref{sect:cliques}.
\end{description}

In order to prove our results concerning graphs with small maximum cut, that is, Theorem \ref{thm:MAIN_MC1}, we follow the same steps. In Section \ref{sect:MaxCut}, we present a toolkit that gives various lower bounds on MaxCut based on the negative eigenvalues of the graph. With the help of these, instead of having a bound on $|\lambda_n|$, we can bound the linear, quadratic, and cubic sum of the negative eigenvalues. This allows us to transfer most of the machinery developed for graphs with bounded smallest eigenvalue to graphs with bounded MaxCut, but with the cost of incurring some losses quantitatively.

\section{Chowla's cosine problem}\label{sect:chowla}

In this section, we give a short proof of \Cref{thm:MAIN_CHOWLA}, assuming Theorem~\ref{thm:MAIN_SE3}. We begin the section by recalling some standard notation. Let $\Gamma$ be a finite group, $A\subset \Gamma$ be a symmetric subset (i.e. a set satisfying $A=A^{-1}$), and let $G=\Cay(\Gamma, A)$. Recall that $\Cay(\Gamma,A)$ is the \emph{Cayley graph} on $\Gamma$ generated by $A$, that is, the graph on vertex set $\Gamma$ in which $x,y\in \Gamma$ are joined by an edge if $xy^{-1}\in A$. 
%In case $G$ is abelian, we use $+$ to denote the group operation. 
If $\Gamma$ is abelian, it is well known that the eigenvalues of $G$ are the values of the discrete Fourier transform $\widehat{\mathds{1}_{A}}$. 
In the special case $\Gamma=\mathbb{Z}/n\mathbb{Z}$, this gives that the eigenvalues of the Cayley graph are
$$\sum_{a\in A}e^{\frac{2 \pi i}{n}\cdot a\xi} = \sum_{a \in A} \cos\left(\frac{2\pi a\xi}{n} \right)$$
for $\xi \in \mathbb{Z}/n\mathbb{Z}$. We restate Theorem \ref{thm:MAIN_CHOWLA} for the reader's convenience.

\begin{theorem}
For any finite set $A$ of positive integers, there exists $x\in [0,2\pi]$ such that 
$$\sum_{a\in A}\cos(ax)\leq -\Omega(|A|^{1/10-o(1)}).$$
\end{theorem}

\begin{proof}
    Let $n>4\max A$ be a prime, and let $G=\Cay(\mathbb{Z}/n\mathbb{Z}, A\cup -A)$.
    Then $G$ is an $n$-vertex $d$-regular graph with $d=2|A|$. 
    Every $\xi\in \mathbb{Z}/n\mathbb{Z}$ corresponds to an eigenvalue of $G$ given by
    $$\lambda_\xi=\sum_{a\in A\cup -A}e^{2\pi i a \xi /n}= 2\sum_{a \in A} \cos\left(\frac{2\pi a\xi}{n} \right).$$
    Hence, if we denote by $\lambda_n$ the smallest eigenvalue of $G$, then there exists $x=\frac{2\pi \xi}{n}$ such that $$\sum_{a\in A}\cos(ax)=\frac{1}{2}\lambda_n.$$
    Let $\gamma=1/10-\eps$ for any fixed $\eps>0$. Our aim is to show that $|\lambda_n|\geq d^{\gamma}$ for $d$ sufficiently large. 
    Assume to the contrary that $|\lambda_n|< d^{\gamma}$, then by Theorem \ref{thm:MAIN_SE3}, $G$ contains a clique $S$ of size $|S|\geq d^{1-4\gamma}=d^{3/5+4\eps}$. We now argue that this is impossible with the help of two auxiliary claims. 
    
    \begin{claim}\label{claim:clique_int}
        There exists a non-zero $t\in \mathbb{Z}/n\mathbb{Z}$ such that $|(t+S)\cap S|\geq |S|(|S|-1)/d$.
    \end{claim}
    
    \begin{proof}
        As $S$ is a clique in $G$, we have $S-S\subset A\cup -A\cup\{0\}$. 
        By averaging, there exists some $t\in A\cup -A$ such that $s'-s=t$ holds for at least $\frac{|S|(|S|-1)}{2|A|}=\frac{|S|(|S|-1)}{d}$ pairs $(s,s')\in S\times S$. 
        Hence, for at least $|S|(|S|-1)/d$ values of $s\in S$ we have $s+t\in S$, and therefore $|(t+S)\cap S|\geq |S|(|S|-1)/d$.
    \end{proof}
    
    In the second auxiliary claim, we identify a simple forbidden induced subgraph of $G$. For a positive integer $k$, let $H_k$ be the graph that is formed by a clique of size $2k$, and an additional vertex connected to exactly half of the vertices of the clique. We show that the smallest eigenvalue of $H_k$ is $-\Omega(\sqrt{k})$. We remark that $H_k$ and its relatives have been studied in connection to the smallest eigenvalue problem for a long time, see e.g. \cite{H73}.
    
    \begin{claim}\label{claim:H_k}
        The smallest eigenvalue $\mu$ of $H_k$ satisfies $\mu<-\sqrt{k/2}$. 
    \end{claim}
    
    \begin{proof}
        Let $V(H_k)=X\cup Y\cup \{x_0\}$, where $X\cup Y$ is a clique of size $2k$, and $X$ is the neighbourhood of $x_0$. Then we have $|X|=|Y|=k$. Let $B$ be the adjacency matrix of $H_k$ and let $v\in \mathbb{R}^{V(H_k)}$ be the vector defined as 
        $$v(x)=\begin{cases} \frac{1}{\sqrt{2}}  &\mbox{if } x=x_0\\
        -\frac{1}{2\sqrt{k}} &\mbox{if } x\in X\\
        \frac{1}{2\sqrt{k}}  &\mbox{if } y\in Y.\end{cases}$$
        Then $\|v\|_2=1$ and thus \[\mu\leq v^TBv=2\sum_{xy\in E(G)}v(x)v(y)=-\sqrt{\frac{{k}}{2}}-\frac{1}{2}<-\sqrt{\frac{k}{2}}.\hfill \qedhere\]
    \end{proof}
    
    By Claim \ref{claim:H_k} and Cauchy's interlacing theorem (cf. Section~\ref{sect:prelim}), $G$ does not contain $H_k$ as an induced subgraph for $k=2d^{2\gamma}=2d^{1/5-2\eps}$.
    As $G[S]$ is a clique with $|S|> 2k$, each vertex of $G$ sends either at most $k$ edges to $S$, or at least $|S|-k$ edges. We prove that every vertex in $G$ must send at least $|S|-k$ edges to $S$. This easily leads to a contradiction for $n$ sufficiently large: this implies that there are at least $(n-|S|)(|S|-k)\geq \frac{n}{2}\cdot \frac{|S|}{2} > d|S|$ edges with an endpoint in $S$, contradicting that $G$ is $d$-regular.
    
    \begin{claim}\label{claim:chowla_2}
        Every $v\in V(G)$ sends at least $|S|-k$ edges to $S$.
    \end{claim}
    
    \begin{proof}
        Let $t\in \mathbb{Z}/n\mathbb{Z}$ be a non-zero element such that $|(t+S)\cap S|\geq \frac{|S|(|S|-1)}{d}$, whose existence is guaranteed by Claim \ref{claim:clique_int}. We prove by induction on $\ell$ that every vertex of $\ell t+S$ sends at least $|S|-k$ edges to $S$. As every vertex $v\in V(G)$ is contained in some $\ell t+S$, this finishes the proof. The base case $\ell=0$ is trivial, so let $\ell\geq 1$. By our induction hypothesis and translation invariance, every vertex $v\in \ell t+S$ sends at least $|S|-k$ edges to $t+S$. But then $v$ sends at least 
        $$|S\cap (t+S)|-k=\frac{|S|^2}{2d}-k\geq \frac{1}{2}d^{1/5+8\eps}-2d^{1/5-2\eps}>2d^{1/5-2\eps}=k$$ edges to $S\cap (t+S)$, and in particular, more than $k$ edges to $S$. Therefore, as $G$ contains no induced copy of $H_k$, $v$ must send at least $|S|-k$ edges to $S$, and we are done. 
    \end{proof}
\end{proof}

\section{Preliminaries}\label{sect:prelim}

We recall some basic facts and standard notation from linear algebra and graph theory. The \emph{edge density} of an $n$-vertex graph $G$ is $m/\binom{n}{2}$, where $m=e(G)$ is the number of edges. Given a subset $U$ of the vertices, $G[U]$ denotes the subgraph of $G$ induced on vertex set $U$. Also, if $V\subset V(G)$ is disjoint from $U$, then $G[U,V]$ is the bipartite subgraph of $V(G)$ induced between $U$ and $V$. The \emph{complement} of $G$ is denoted by $\overline{G}$. The \emph{maximum degree} of $G$ is denoted by $\Delta(G)$, and the average degree by $d(G)$. We will often identify the set of vertices of $G$ with $[n]=\{1,2,\dots,n\}$.

The \emph{MaxCut} of $G$, denoted by $\mbox{mc}(G)$, is the maximum size of a cut, where a \emph{cut} is a partition $(U,V)$ of the vertices into two parts, with all the edges having exactly one endpoint in both parts. The size of a cut is the number of its edges. The \emph{surplus} of $G$ is defined as $\surp(G)=\mbox{mc}(G)-m/2$, where $m$ is the number of edges of $G$. Note that $\surp(G)$ is always nonnegative. A useful property of the surplus is that if $G_0$ is an induced subgraph of $G$, then $\surp(G_0)\leq \surp(G)$, see e.g. \cite{GJS}.

Given an $n\times n$ real symmetric matrix $M$, we denote by $\lambda_1(M)\geq\dots\geq \lambda_n(M)$ the eigenvalues of $M$ with multiplicity. 
If $G$ is an $n$-vertex graph whose adjacency matrix is $A$, then we denote by $\lambda_i=\lambda_i(A)$ the eigenvalues of $A$, sometimes also calling them the eigenvalues of $G$.
We also denote by $v_1, \dots, v_n$ a corresponding orthonormal basis of eigenvectors (all vectors in this paper will be column vectors by default).
By the Perron--Frobenius theorem, we may take $v_1$ to be a vector with non-negative entries, which we call the \emph{principal eigenvector} of $A$. Furthermore, the corresponding eigenvalue satisfies $d(G)\leq \lambda_1\leq \Delta(G)$. See the survey \cite{eigenvalue_survey} as a general reference on the principal eigenvector.

An important and useful fact about spectra of graphs is \textit{Cauchy's interlacing theorem}.
In the case of graphs, it states that if $G$ is an $n$-vertex graph with eigenvalues $\lambda_1\geq \dots\geq \lambda_n$ and $G'\subseteq  G$ is an induced subgraph on $n-1$ vertices with eigenvalues $\mu_1\geq  \dots \geq \mu_{n-1}$, then we have \[\lambda_1\geq \mu_1\geq \lambda_2\geq \mu_2\geq \dots\geq \mu_{n-1}\geq \lambda_n.\]
For a proof of this result, see e.g. \cite{fisk04}. Crucially, this implies that if $G$ is a graph with smallest eigenvalue $\lambda_n$ and $G'\subseteq G$ is an induced subgraph of $G$ with smallest eigenvalue $\mu_k$, then $\mu_k\geq \lambda_n$.

Given two $n\times n$ matrices $A$ and $B$, their scalar product is defined as $$\langle A,B\rangle=\trace(AB^T)=\sum_{1\leq i,j\leq n} A_{i,j}B_{i,j}.$$
The \emph{Frobenius-norm} of an $n\times n$ matrix $A$ is
$$\|A\|_F^2=\langle A,A\rangle=\sum_{i,j=1}^{n}A_{i,j}^2.$$
If $A$ is symmetric with eigenvalues  $\lambda_1,\dots,\lambda_n$, then we also have $$\|A\|_F^2= \langle A,A\rangle = \trace(A^2)=\sum_{i=1}^n\lambda_i^2.$$

The \emph{Hadamard product} (also known as entry-wise product) of $A$ and $B$ is the $n\times n$ matrix $A\circ B$ defined as $(A\circ B)_{i,j}=A_{i,j}B_{i,j}$. We denote the $k$-term Hadamard product $A\circ\dots\circ A$ by $A^{\circ k}$. A useful feature of the Hadamard product, which is a key component of our arguments, is that it preserves positive semidefiniteness.
\begin{theorem}[Schur product theorem] \label{theorem: schur}
If $A$ and $B$ are positive semidefinite matrices, then $A\circ B$ is also positive semidefinite.
\end{theorem}
We also exploit the simple observation that if $A$ is an adjacency matrix, then $A=A\circ A$. Another useful identity involving the Hadamard product is that if $x,y,u,v$ are vectors, then $$(xy^T)\circ(uv^T)=(x\circ u) (y\circ v)^T.$$
Here, we use the Hadamard product for vectors: for $u,v \in \mathbb{R}^n$, their Hadamard product vector $u\circ v \in \mathbb{R}^{n}$ is defined by $(u\circ v)(i):=u(i)v(i)$ for all $i \in [n]$.

Throughout our proofs, we omit the use of floors and ceilings whenever they are not crucial. 

\section{Spectral lower bounds for the surplus}\label{sect:MaxCut}

In this section, we present bounds on the MaxCut of a graph in terms of its spectrum. These inequalities are crucial for transferring our results for the smallest eigenvalue to the MaxCut setting.

\begin{claim}\label{claim:surplus_and_lambdan}
For an $n$-vertex graph $G$ with smallest eigenvalue $\lambda_n$, we have $\surp(G)\leq |\lambda_n|n/4$. 
\end{claim}
\begin{proof}
Let $A$ be the adjacency matrix of $G$. We can assign a vector with entries $\pm 1$ to each cut $V(G)=X\cup Y$, by setting $x_u=1$ if $u\in X$ and $x_u=-1$ otherwise. Then, the surplus of this cut equals $\frac{1}{2}\big(e(X, Y)-e(X)-e(Y)\big)=-\frac{1}{2}\sum_{\{u, v\}\subseteq V(G)} x_uA_{uv}x_v=-\frac{1}{4}\sum_{u, v\in V(G)} x_uA_{uv}x_v$. 
Hence, we have 
\[\surp(G)=\frac{1}{4}\max_{x\in \{-1,1\}^n}-x^T Ax=\frac{1}{4}\max_{x\in [-1,1]^n}-x^T Ax.\]
Note that $-x^TAx\leq |\lambda_n|\|x\|_2^2$ for every vector $x\in \bR^n$, so $\surp(G)\leq \frac{1}{4} |\lambda_n| \sqrt{n}^2=|\lambda_n|n/4$.
\end{proof}

The key ingredient of the above proof is the relation $\surp(G)=\frac{1}{4}\max_{x\in [-1,1]^n}-x^T Ax$.
This can also be written as $\surp(G)=\frac{1}{4}\max_{x\in [-1, 1]^n}\langle -A, xx^T\rangle$, where we observe that $xx^T$ is a positive-semidefinite matrix with diagonal entries bounded by $1$. Based on this, we define the semidefinite relaxation of the surplus as follows. Given an $n$-vertex graph $G$ with adjacency matrix $A$, define $$\surp^*(G)=\max_{X} -\langle A,X\rangle,$$ where the maximum is taken over all $n\times n$ positive semidefinite matrices $X$ such that $X_{i,i}\leq 1$ for every $i\in [n]$. The following inequality between $\surp(G)$ and $\surp^*(G)$ can be found in \cite{RT}, and it is a simple application of the graph Grothendieck inequality of Charikar and Wirth \cite{CW04}. 

\begin{lemma}[\cite{RT}]\label{claim:groth}
For every graph $G$, we have $\surp^*(G)\geq \surp(G)\geq \Omega\Big(\frac{\surp^*(G)}{\log n}\Big)$.
\end{lemma}

The semidefinite relaxation $\surp^*(G)$ allows us to obtain lower bounds on the surplus using the negative eigenvalues of a graph $G$.  Parts of the following lemma and similar bounds can be also found in \cite{RST,RT}. Given a graph $G$, let 
$$\Delta^*(G):=\min\{\Delta(G),\Delta(\overline{G})\}.$$

\begin{lemma}\label{lemma:surp_star}
There exists an absolute constant $c>0$ such that the following holds. Let $G$ be a graph on $n$ vertices with eigenvalues $\lambda_i=\lambda_i(G)$, and let $\Delta^*=\Delta^*(G)$. Then
\begin{itemize}
    \item[(i)] $\surp^*(G)\geq \displaystyle\sum_{\lambda_i<0}|\lambda_i|$
    \item[(ii)] $\surp^*(G)\geq\frac{c}{\sqrt{\Delta^*+1}}\displaystyle\sum_{\lambda_i<0}\lambda_i^2$
    \item[(iii)]  $\surp^*(G)\geq \frac{c}{\Delta^*+1}\displaystyle\sum_{\lambda_i<0}|\lambda_i|^3$.
\end{itemize}
\end{lemma}

Before we prove Lemma~\ref{lemma:surp_star}, we briefly discuss  two preliminary results. First, we show that the entries of eigenvectors corresponding to large eigenvalues are smoothly distributed. Then, we show that the entries of the principal eigenvector are especially well behaved.

\begin{lemma}\label{lemma:smooth_eigenvectors}
Let $G$ be an $n$-vertex graph, and let $\lambda$ be an eigenvalue with normalized eigenvector $v$. Then
$$\|v\|_{\infty}\leq \frac{\sqrt{n}}{|\lambda|}.$$
\end{lemma}

\begin{proof}
    For every $b\in [n]$, we have $\lambda v(b)=\sum_{b\sim i}v(i)$, where we use $x\sim y$ to denote that $x$ is connected to $y$ by an edge in $G$. 
    By the inequality between the arithmetic and quadratic mean,  
    $$\frac{1}{n}\left|\sum_{b\sim i}v(i)\right|\leq \frac{1}{n}\sum_{i=1}^n |v(i)|\leq \sqrt{\frac{\sum_i v(i)^2}{n}}=\frac{1}{\sqrt{n}},$$ where we used that $\sum_{i=1}^nv_1(i)^2=1$. 
    Hence, $|\lambda| |v(b)|\leq \sqrt{n}$, i.e. $|v(b)| \le \sqrt{n}/|\lambda|$.
\end{proof}

\begin{lemma}\label{lemma:max_entry}
    Let $G$ be a graph on $n > 10$ vertices, whose complement has edge density $p\leq 1/10$ and maximum degree $\Bar{\Delta}=\Delta(\overline{G})$. If $v_1$ is the principal eigenvector of $G$, then for each $i\in [n]$ we have \[\frac{1-3\Bar{\Delta}/n}{\sqrt{n}} \leq v_1(i)\leq \frac{1+2p+2/n}{\sqrt{n}}.\]
\end{lemma}
\begin{proof}
    Let $d=d(G)=(1-p)(n-1)$ be the average degree of $G$, and recall that $\lambda_1\geq d$. 
    By \Cref{lemma:smooth_eigenvectors},
    $$v_1(i)\leq \frac{\sqrt{n}}{\lambda_1}\leq \frac{\sqrt{n}}{d}=\frac{\sqrt{n}}{(1-p)(n-1)}\leq \frac{1+2p+2/n}{\sqrt{n}}.$$
    In the last inequality, we used that $p<1/10$ and $n > 10$.
    To prove the lower bound, we may assume that $\Bar{\Delta} \ge 1$.
    Observe that 
    $$1=\sum_{i=1}^n v_1(k)^2\leq \|v_1\|_{\infty} \sum_{k=1}^n v_1(k)\le \frac{\sqrt{n}}{|\lambda_1|}\sum_{k=1}^n v_1(k),$$
    which implies that $\sum_{k=1}^n v_1(k)\geq \frac{\lambda_1}{\sqrt{n}}$. 
    Then, using the identity $Av_1=\lambda_1 v_1$,
    $$  \lambda_1v_1(i)
        =\sum_{k\sim i}v_1(k)
        \geq \sum_{k=1}^nv_1(k)- (\Bar{\Delta}+1) \|v_1\|_{\infty}
        \geq\frac{\lambda_1}{\sqrt{n}}-\frac{(\Bar{\Delta}+1)\sqrt{n}}{\lambda_1}
        = \frac{\lambda_1}{\sqrt{n}}\Big(1-\frac{(\Bar{\Delta}+1)n}{\lambda_1^2}\Big)\!\!
        \geq \frac{\lambda_1}{\sqrt{n}}\Big(1-\frac{3\Bar{\Delta}}{n}\Big),$$
    where we used $\lambda_1^2\geq d^2\geq (9/10)^2(n-1)^2\ge (2/3)n^2$ in the end. 
    Cancelling $\lambda_1$ gives $v_1(i)\geq \frac{1-3\Bar{\Delta}/n}{\sqrt{n}}$.
\end{proof}

\begin{proof}[Proof of \Cref{lemma:surp_star}.]
We begin by showing the inequalities \textit{(i)} and \textit{(iii)}, which we then combine to derive \textit{(ii)}. 
Let $v_1,\dots,v_n$ be an orthonormal basis of eigenvectors corresponding to the eigenvalues $\lambda_1,\dots,\lambda_n$; so $A = \sum_{i=1}^n \lambda_i v_i v_i^T$.
The inequalities \textit{(i)} and \textit{(iii)} will be shown by plugging in the appropriate test matrix $X$ in the formula $\surp^*(G)= \max_X -\langle A, X\rangle$. Observe that, if we choose $X=\sum_{i=1}^n\alpha_i v_iv_i^T$ for some real numbers $\alpha_1,\dots,\alpha_n$, then
$$\langle A,X\rangle=\sum_{i=1}^n\sum_{j=1}^n\alpha_i\lambda_j \langle v_iv_i^T,v_jv_j^T\rangle=\sum_{i=1}^n\sum_{j=1}^n\alpha_i\lambda_j \langle v_i,v_j\rangle^2=\sum_{i=1}^n\alpha_i\lambda_i.$$

\textit{(i)} Let $X=\sum_{\lambda_i<0} v_iv_i^T$. Then $X$ is positive semidefinite, and as $v_1,\dots,v_n$ is an orthonormal basis,  $$X_{j,j}=\sum_{\lambda_i<0}v_i(j)^2\leq\sum_{i=1}^n v_i(j)^2=\sum_{i=1}^n \langle v_i, \mathbf{e}_j\rangle^2= \|\mathbf{e}_j\|^2=1.$$ Therefore, \[\surp^*(G)\geq -\langle A,X\rangle=\sum_{\lambda_i<0}|\lambda_{i}|.\]

\textit{(iii)} Let $\beta=\frac{1}{120(\Delta^*+1)}$, and $X=\beta\sum_{\lambda_i<0} \lambda_i^2v_iv_i^T$. Then $X$ is positive semidefinite. It is enough to prove that the diagonal entries of $X$ are bounded by $1$, as then $\surp^*(G)\geq -\langle A,X\rangle =\beta\sum_{\lambda_i<0}|\lambda_{i}|^3$. 

First, consider the (easier) case $\Delta^*=\Delta(G)$.
Observe that $$\beta A^2-X=\beta\sum_{\lambda_i>0}\lambda_i^2v_iv_i^T$$ is positive semidefinite, so the diagonal entries of $X$ are dominated by those of $\beta A^2$, which are simply the degrees in $G$.
So, $X_{i,i} \le \beta(A^2)_{i,i}\leq \beta \Delta\leq 1$, as claimed.

Next, consider the case $\Delta^*=\Delta(\overline{G})$. 
We may assume that the edge density of $\overline{G}$ is less than $1/10$, otherwise $\Delta^*=\Omega(n)$, and the previous case implies $\surp^*(G)=\Omega\big(\frac{1}{n}\sum_{\lambda_i<0}|\lambda_i|^3\big)$. 
To show that  $X_{i,i}\leq 1$, we analyse the matrix $B=A-\lambda_1v_1v_1^T$.
Since $$\beta B^2-X=\beta\sum_{i\neq 1,\lambda_i>0}\lambda_i^2v_iv_i^T,$$ we have that $\beta B^2-X$ is positive semidefinite.
This means $X_{i,i} \le (\beta B^2)_{i,i}$ for every $i\in [n]$. Therefore, it is enough to show that $(B^2)_{i,i}\leq 1/\beta=120(\Delta^*+1)$.

To this end, we first bound the entries of $B$. We denote by $p$ the density of $\overline{G}$, and observe that $p\leq \frac{\Delta^* n/2}{\binom{n}{2}}= \frac{\Delta^*}{n-1}$. 
Then, \Cref{lemma:max_entry} implies that for any $i, j\in [n]$ we have 
\[ 1-\frac{7(\Delta^* + 1)}{n}
    \leq (n-1)(1-p)\left(\frac{1-3\Delta^*/n}{\sqrt{n}}\right)^2
    \leq \lambda_1 v_1(i)v_1(j)\leq n\left(\frac{1+2p+2/n}{\sqrt{n}}\right)^2
    \leq 1+\frac{7(\Delta^*+1)}{n}.\]
Therefore, for every $i, j\in [n]$, if $ij\in E(G)$ and $A_{i, j}=1$, then $|B_{i, j}|\leq \frac{7(\Delta^*+1)}{n}$. 
Otherwise, we have $|B_{i,j}|\leq 1+7(\Delta^*+1)/n\leq 8$. 
From this, we have 
$$(B^2)_{i,i}=\sum_{j=1}^n (B_{i,j})^2\leq 64(\Delta^*+1)+n\frac{49(\Delta^*+1)^2}{n^2}\leq 120(\Delta^*+1).$$
\textit{(ii)} We show that \textit{(i)} and \textit{(iii)} can be combined to give the desired lower bound on $\surp^*(G)$. Namely, we have
\[\surp^*(G)^2\geq \beta\left(\sum_{\lambda_i<0}|\lambda_i|^3\right)\left(\sum_{\lambda_i<0}|\lambda_i|\right)\geq \beta\left(\sum_{\lambda_i<0}\lambda_i^2\right)^2.\]
Note that the first inequality is the combination of \textit{(i)} and \textit{(iii)}, while the second one is simply the Cauchy--Schwartz inequality applied to the sequences $(|\lambda_i|^3)_{\lambda_i<0}$ and $(|\lambda_i|)_{\lambda_i<0}$. Taking square roots then proves \textit{(ii)}.
\end{proof}

Finally, we remark two simple, but important properties of $\surp^*(\cdot)$, that will be used repeatedly.

\begin{claim}
If $G'$ is an induced subgraph of $G$, then $\surp^*(G')\leq \surp^*(G)$.
\end{claim}

\begin{proof}
    Write $A'$ for the adjacency matrix of $G'$, and let $X'\in \mathbb{R}^{V(G')\times V(G')}$ be a matrix such that $X'$ is positive semidefinite, $X'_{i,i}\leq 1$ for every $i\in V(G')$, and $\surp^*(G')=-\langle A',X'\rangle$. 
    Then, write $A$ for the adjacency matrix of $G$ and let $X\in \mathbb{R}^{V(G)\times V(G)}$ be the matrix that agrees with $X'$ on every entry $(x,y)\in V(G')\times V(G')$, and zero everywhere else. Then \[\surp^*(G)\geq -\langle A,X\rangle=-\langle A',X'\rangle=\surp^*(G').\hfill\qedhere\]
\end{proof}

\begin{claim}\label{claim: surp* lambdan}
If $G$ is an $n$-vertex graph with smallest eigenvalue $\lambda_n$, then $\surp^*(G)\leq |\lambda_n|n$.
\end{claim}

\begin{proof}
Let $X\in \mathbb{R}^{n\times n}$ be a positive semidefinite matrix such that $X_{i,i}\leq 1$ for every $i\in [n]$. Let $A=\sum_{i=1}^n\lambda_iv_iv_i^T$ be the spectral decomposition of $A$, then 
$$ -\langle A,X\rangle 
    =-\sum_{i=1}^n \lambda_i\langle v_iv_i^T,X\rangle
    \leq \sum_{i=1}^n |\lambda_n| \langle v_iv_i^T,X\rangle=|\lambda_n|\langle I,X\rangle\leq |\lambda_n|n.$$
In the first inequality, we used that $\langle v_iv_i^T,X\rangle=v_i^T X v_i \ge 0$ as $X$ is positive semidefinite.
\end{proof}

\section{Main lemmas}\label{sect:main_lemmas}

An important component of the proofs of our main results is the notion and properties of the \emph{subspace compression} of matrices. This is a special instance of the compression of linear operators, see the book of Halmos \cite{Halmos} as a general reference.

\medskip

\noindent
\textbf{$W$-compression and $W$-trace.} Let $W<\mathbb{R}^n$ be a subspace. 
We denote by $\Pi_{W}$ the orthogonal projection matrix onto $W$.  
It is easy to check that $\Pi_W$ is symmetric.
Given an $n\times n$ symmetric matrix $M$, the \emph{$W$-compression} of $M$ is the symmetric matrix $$M_W:=\Pi_{W}M\Pi_{W}.$$ Furthermore, the $W$-trace of $M$ is
$$\trace_W(M):=\trace(M_W).$$
Clearly, $\trace_W$ is a linear functional. Observe that if $M=uu^T$, then $M_W=(\Pi_Wu)(\Pi_Wu)^T$ and thus 
$$\trace_W(u u^T)=\|\Pi_Wu\|_2^2.$$ Finally, given an orthonormal basis $w_1,\dots,w_d$ of $W$, the $W$-trace can be calculated as 
$$\trace_W(M)=\sum_{i=1}^d w_i^TMw_i.$$
From this equality, it also follows that $\trace_W(I)=\dim(W)$. 
We present an upper bound on the $W$-trace that will be used later.

\begin{lemma}\label{lemma:W-trace_bound}
$|\trace_W(M)|\leq \dim(W)^{1/2}\|M\|_F.$
\end{lemma}

\begin{proof}
Let $M=\sum_{i=1}^n\mu_i v_iv_i^T$ be the spectral decomposition of $M$. Then 
\begin{align*}
|\trace_W(M)|&=\left|\sum_{i=1}^n\mu_i \trace_W(v_iv_i^T)\right|=\left|\sum_{i=1}^n\mu_i \|\Pi_W v_i\|_2^2\right|\leq \sum_{i=1}^n|\mu_i|\cdot \|\Pi_W v_i\|_2\\
&\leq \left(\sum_{i=1}^n \mu_i^2\right)^{1/2}\cdot \left(\sum_{i=1}^n \|\Pi_W v_i\|_2^2\right)^{1/2}= \|M\|_F\dim(W)^{1/2}.
\end{align*}
Here, the first inequality uses that $\|\Pi_W v_i\|\leq 1$ for every $i\in [n]$, and the second inequality is due to the Cauchy--Schwartz inequality.
\end{proof}

\bigskip

The importance of the $W$-compression and $W$-trace is that it allows us to analyse the contribution of the large eigenvalues of a matrix, by choosing an appropriate subspace $W$. Given a graph $G$ with adjacency matrix $A$, eigenvalues $\lambda_1\geq \dots\geq \lambda_n$ and a positive real number $T$, we write 
$$S_T(G)=\sum_{i: \lambda_i\geq T}\lambda_i.$$
If the graph $G$ is clear from the context, we simply write $S_T$ instead of $S_T(G)$. Furthermore, let $N_T=N_T(G)$ denote the number of eigenvalues at least $T$. We will use repeatedly  that 
$$N_T\leq \frac{S_T}{T}.$$
The next lemma gives a simple upper bound on the trace of the $W$-compression of $A$.

\begin{lemma}\label{lemma:trace_compression}
Let $G$ be an $n$-vertex graph with adjacency matrix $A$ and let $W<\mathbb{R}^n$. Then for every $K>0$,
$$\trace_W(A)\leq S_K+K\dim(W).$$
\end{lemma}
\begin{proof}
We have
\[\trace_W(A)=\sum_{i=1}^n \lambda_i \|\Pi_W v_i\|_2^2\leq \sum_{\lambda_i\geq K}\lambda_i+K\sum_{i=1}^n \|\Pi_W v_i\|_2^2=S_K+K\dim(W).\hfill\qedhere\]
\end{proof}

\subsection{Main lemma -- smallest eigenvalue version}\label{sect:main_se}

The following lemma is the heart of our argument. It shows a curious recursive relation between the sums of the largest eigenvalues, under the assumption that there are no very negative eigenvalues in the graph. Later, in \Cref{sect:recursion}, we show how to use this relation to conclude that the quadratic sum of all but the large eigenvalues are negligible. This is, in turn, equivalent to saying that the adjacency matrix is well approximated in the Frobenius norm by the part of the spectral decomposition corresponding to the large eigenvalues, which we will discuss in more detail in \Cref{sect:dens2}.

\begin{lemma}\label{lemma:main_smallest_eigenalue}
Let $G$ be an $n$-vertex graph. If $T\geq 2|\lambda_n|\sqrt{n}$, then
\begin{equation}\label{eqn:S_T bound}
    4nS_{\frac{T^2}{2n}}\geq S_T^2.
\end{equation}
\end{lemma}

\noindent
The proof this lemma is prepared by a technical result, which will be used later as well.

\begin{claim}\label{claim:hadamard_sum}
Let $v_1,\dots,v_n$ be an orthonormal basis of eigenvectors of $A$ corresponding to the eigenvalues $\lambda_1\geq \dots\geq \lambda_n$. Then:
    $$\sum_{\lambda_i,\lambda_j\geq T}\lambda_i\lambda_j\|v_i\circ v_j\|_2^2\geq\frac{S_T^2}{n}.$$
\end{claim}

\begin{proof} 
We can write
\begin{equation} \nonumber
    \sum_{\lambda_i,\lambda_j\geq T}\lambda_i\lambda_j\|v_i\circ v_j\|_2^2
    = \sum_{\lambda_i,\lambda_j\geq T}\lambda_i\lambda_j\sum_{k=1}^n v_i(k)^2 v_j(k)^2
    = \sum_{k=1}^n\left(\sum_{\lambda_i\geq T} \lambda_i v_{i}(k)^2\right)^2\geq \frac{1}{n}\left(\sum_{\lambda_i\geq T}\sum_{k=1}^n\lambda_i v_i(k)^2\right)^2,
\end{equation}
where the inequality comes from the inequality between the quadratic and arithmetic mean. Since the eigenvectors are of unit length, we have $\sum_{k=1}^n v_i(k)^2=1$ for all $i$. Hence, the right-hand-side is equal to $\frac{1}{n}S_T^2$, finishing the proof.
\end{proof}

\begin{proof}[Proof of \Cref{lemma:main_smallest_eigenalue}]
The main idea is to analyse the identity $A=A\circ A$ using the spectral decomposition of $A$. In fact, (\ref{eqn:S_T bound}) can be deduced by considering the $W$-traces of both sides of this identity for an appropriately chosen $W$. 
More precisely, in order to isolate the contribution of large eigenvalues, we define $W$ to be the subspace generated by those vectors $v_i\circ v_j$, where $\lambda_i$ and $\lambda_j$ are both at least $T$, i.e. 
$$W=\langle v_i\circ v_j : \lambda_i,\lambda_j\geq T\rangle.$$

Then, \Cref{lemma:trace_compression} applied with $K=\frac{T^2}{2n}$ implies that $\trace_W(A)\leq S_{\frac{T^2}{2n}}+\frac{T^2}{2n}\dim(W)$. On the other hand, we can lower-bound $\trace_W(A\circ A)$ as follows.

\begin{claim}
\[\trace_W(A\circ A)\geq \frac{S_T^2}{n}-\lambda_n^2\dim(W).\]
\end{claim}
\begin{proof}
We can rewrite $A\circ A$ as
\begin{align*}
A\circ A&=(A+|\lambda_n|I)\circ (A+|\lambda_n|I)-\lambda_n^2I=\left(\sum_{i=1}^n (\lambda_i+|\lambda_n|)v_iv_i^T\right)^{\circ 2}-\lambda_n^2 I\nonumber\\
&=\sum_{i,j}(\lambda_i+|\lambda_n|)(\lambda_j+|\lambda_n|)(v_i\circ v_j)(v_i\circ v_j)^T-\lambda_n^2 I.
\end{align*}
Thus, the $W$-trace of $A\circ A$ can be computed as follows
\begin{align*}
    \trace_W(A\circ A)
    &=\sum_{i,j}(\lambda_i+|\lambda_n|)(\lambda_j+|\lambda_n|)\trace_W\big((v_i\circ v_j)(v_i\circ v_j)^T\big)-\lambda_n^2 \trace_W(I)\nonumber\\
    &=\sum_{i,j}(\lambda_i+|\lambda_n|)(\lambda_j+|\lambda_n|)\big\|\Pi_W v_i\circ v_j\big\|_2^2 - \lambda_n^2\dim(W).
\end{align*}
Note that if $\lambda_i,\lambda_j\geq T$, then $v_i\circ v_j\in W$, so $\Pi_W v_i\circ v_j=v_i\circ v_j$ and 
$$(\lambda_i+|\lambda_n|)(\lambda_j+|\lambda_n|)\big\|\Pi_W v_i\circ v_j \big\|_2^2\geq \lambda_i\lambda_j \|v_i\circ v_j\|_2^2.$$ 
Also, each term in the sum is nonnegative, so if $\lambda_i<T$ or $\lambda_j<T$, we simply lower bound the contribution of $(\lambda_i+|\lambda_n|)(\lambda_j+|\lambda_n|)\|\Pi_W v_i\circ v_j\|_2^2$ by 0. 
Finally, using \Cref{claim:hadamard_sum}, we get \[\trace_W(A\circ A)\geq \sum_{\lambda_i,\lambda_j\geq T}\lambda_i\lambda_j\|v_i\circ v_j\|_2^2-\lambda_n^2 \dim(W)\geq  \frac{S_T^2}{n}-\lambda_n^2\dim(W). \hfill\qedhere\]
\end{proof}
We now complete the proof of inequality (\ref{eqn:S_T bound}). We have 
\[S_{\frac{T^2}{2n}}+\frac{T^2}{2n}\dim(W)\geq \trace_W(A)= \trace_W(A\circ A)\geq \frac{S_T^2}{n}-\lambda_n^2\dim(W).\]
Note that $W$ has at most $N_T^2$ generators, so $\dim(W)\leq N_T^2=\frac{S_T^2}{T^2}.$ Finally, $|\lambda_n|^2\leq T^2/4n$ and thus the previous inequality implies
\[S_{\frac{T^2}{2n}}\geq \frac{S_T^2}{n}-\Big(\frac{T^2}{2n}+\lambda_n^2\Big)\dim(W)\geq \frac{S_T^2}{n} -\frac{3T^2}{4n} \cdot\frac{S_T^2}{T^2}=\frac{S_T^2}{4n}. \hfill\qedhere\]
\end{proof}

\subsection{Main lemma -- MaxCut version}

In this section, we present a variant of the previous lemma for graphs with small MaxCut. 
We employ a similar strategy as in the proof of \Cref{lemma:main_smallest_eigenalue}, but, instead of writing $A=(A+|\lambda_n|I)^{\circ 2}-|\lambda_n| I$, we write $A=(A+E)\circ (A+E)-2A\circ E-E\circ E$, where $E$ is the contribution of the negative eigenvalues. Then, most of the proof comes down to showing that if the surplus is small, then $\trace_W(A\circ E)$ and $\trace_W(E\circ E)$ are also small for an appropriately chosen $W$.
For technical reasons, we prove it in terms of $\surp^*(G)$ rather than $\surp(G)$.

\begin{lemma}\label{lemma:main_MaxCut}
For every $\gamma \in (0,1/6)$, there exists a constant $C>0$ such that the following holds.
Let $G$ be an $n$-vertex graph such that $\surp^*(G)\leq n^{1+\gamma}$. 
% Let $T\geq Cn^{1-\frac{1}{24}+\frac{\gamma}{4}}$ for some sufficiently large absolute constant $C>0$.
Then, for every $T\geq Cn^{1-\frac{1}{24}+\frac{\gamma}{4}}$,
$$4nS_{\frac{T^2}{2n}}\geq S_T^2.$$
\end{lemma}
\begin{proof}
Let $Q=\surp^*(G)$. 
As in the proof of \Cref{lemma:main_smallest_eigenalue}, we analyse the identity $A=A\circ A$. Let $E$ be the ``negative part'' of $A$, that is, $$E:=\sum_{\lambda_i<0}|\lambda_i|v_iv_i^T.$$ 
Then we can rewrite $A=A\circ A$ as
\begin{equation}\label{equ:mainlemma1}
A=(A+E)\circ (A+E)-2A\circ E-E\circ E.
\end{equation}
The proof revolves around choosing an appropriate subspace $W$ and bounding the $W$-traces of both sides. The terms $2A\circ E$ and $E\circ E$ constitute as error terms, for which we show that their contribution to the $W$-trace is negligible.

Let $W_0<\mathbb{R}^n$ be the subspace generated by those vectors $v_i\circ v_j$, where $\lambda_i$ and $\lambda_j$ are at least $T$, i.e.
$$W_0=\langle v_i\circ v_j : \lambda_i,\lambda_j\geq T\rangle.$$
The subspace $W_0$ is almost what we want. However, when bounding the error term $\trace_{W_0}(E\circ E)$, the large entries of $E$ have a non-negligible contribution. In order to overcome this, we introduce a cut-off 
$$\beta:=\frac{Q^{1/4}n^{7/8}}{T}>1.$$
Let $J\subset[n]$ be the set of indices $i$ such that $E_{i,i}>\beta$. Note that as $E$ is positive semidefinite, we have $\max_{i,j}|E_{i,j}|=\max_{i,i}E_{i,i}$, so $|E_{i,j}|\leq \beta$ for every $i,j\notin J$. Moreover, $|J|$ is small.

\begin{claim}
$|J|\leq Q/\beta.$
\end{claim}
\begin{proof}
    By Lemma \ref{lemma:surp_star} (i), we have 
        $\sum_{i=1}^n E_{i,i}=\trace(E)=\sum_{\lambda_i<0}|\lambda_i|\leq Q.$
    Hence, the sum $\sum_{i=1}^n E_{i,i}$ contains at most $Q/\beta$ terms larger than $\beta$.
\end{proof}

Let $Y<\mathbb{R}^n$ be the subspace of vectors that vanish on $J$, that is, 
$$Y:=\{y\in\mathbb{R}^n:\forall i\in J, y(i)=0\}.$$ Finally, define $$W:=\Pi_Y(W_0).$$
Note that $$\dim(W)\leq \dim(W_0)\leq N_T^2\leq \frac{S_T^2}{T^2}.$$
Consider the trace of the $W$-compressions of both sides of (\ref{equ:mainlemma1}). 
Let $K=\frac{T^2}{2n}$. 
Then \Cref{lemma:trace_compression} implies
$$\trace_W(A)\leq S_K+K\dim(W)\leq S_{K}+K\frac{S_T^2}{T^2}= S_{\frac{T^2}{2n}}+\frac{S_T^2}{2n}.$$
On the other hand, the term $\trace_W((A+E)\circ (A+E))$ can be lower bounded as follows.
\begin{claim}
    \[\trace_W((A+E)\circ (A+E))\geq \frac{S_T^2}{n}-\frac{Q  S_T^2 n^2}{\beta T^4}.\]    
\end{claim}
\begin{proof}
Since we have 
\begin{equation*}
(A+E)\circ (A+E)=  \bigg(\sum_{\lambda_i>0}\lambda_iv_iv_i^T\bigg)\circ \bigg(\sum_{\lambda_i>0}\lambda_iv_iv_i^T\bigg) =\sum_{\lambda_i,\lambda_j>0}\lambda_i\lambda_j (v_i\circ v_j)(v_i\circ v_j)^T,
\end{equation*}
we can write write
$$\trace_W((A+E)\circ (A+E))=\sum_{\lambda_i,\lambda_j>0} \lambda_i\lambda_j \|\Pi_W v_i\circ v_j\|_2^2\geq \sum_{\lambda_i,\lambda_j\geq T} \lambda_i\lambda_j \|\Pi_W v_i\circ v_j\|_2^2.$$
Here, by definition, we have $v_i\circ v_j\in W_0$, so $\Pi_W (v_i\circ v_j)=\Pi_Y (v_i\circ v_j)$. Thus,
$$\|\Pi_W (v_i\circ v_j)\|_2^2=\|\Pi_Y (v_i\circ v_j)\|_2^2=\|v_i\circ v_j\|_2^2-\sum_{k\in J} (v_i(k)v_j(k))^2.$$
By \Cref{lemma:smooth_eigenvectors}, the entries of $v_i$ and $v_j$ are bounded as $|v_i(k)|,|v_j(k)|\leq \frac{\sqrt{n}}{T}$, so we get
$$\big\|\Pi_W v_i\circ v_j \big\|_2^2
    =\|v_i\circ v_j\|_2^2-\sum_{k\in J} (v_i(k)v_j(k))^2
    \geq \|v_i\circ v_j\|_2^2-\frac{|J|n^2}{T^4}.$$
With this bound, we get 
\begin{align*}
    \trace_W((A+E)\circ (A+E))&\geq \sum_{\lambda_i,\lambda_j\geq T} \lambda_i\lambda_j\left(\|v_i\circ v_j\|_2^2-\frac{|J|n^2}{T^4}\right)\\
    &\geq \bigg(\,\sum_{\lambda_i,\lambda_j\geq T}\lambda_i\lambda_j\|v_i\circ v_j\|_2^2\,\bigg)-S_T^2\frac{|J|n^2}{T^4}
    \geq  \frac{S_T^2}{n}-\frac{Q  S_T^2 n^2}{\beta T^4}.
\end{align*}
Here, the second inequality follows by writing $\sum_{\lambda_i,\lambda_j\geq T}\lambda_i\lambda_j=S_T^2$, and the last inequality follows by Claim \ref{claim:hadamard_sum}, and writing  $|J|\leq Q/\beta$.
\end{proof}

Finally, we bound $\trace_W(E\circ A)$ and $\trace_W(E\circ E)$. First, we have 
$$\|E\circ A\|_F\leq \|E\|_F=\bigg(\sum_{\lambda_i<0}\lambda_i^2\bigg)^{1/2}=O(n^{1/4}Q^{1/2}),$$
where the last equality follows from \Cref{lemma:surp_star} (ii). 
By \Cref{lemma:W-trace_bound}, we have
$$\trace_W(A\circ E)=\dim(W)^{1/2} \|E\circ A\|_F\leq O\big(\dim(W)^{1/2}n^{1/4}Q^{1/2}\big).$$
Now consider $\trace_W(E\circ E)$.
Let $E'=E_Y$ be the $Y$-compression of $E$.
Then $E'_{i,j}=E_{i,j}$ if $i,j\not\in J$, and $E'_{i,j}=0$ otherwise. 
Recall that $|E_{i,j}| \le \beta$ for every $i,j \notin J$. We acquire
\begin{equation} \nonumber
    \|E'\circ E'\|_F
    =\bigg(\,\sum_{i,j\not\in J} E_{i,j}^4\,\bigg)^{1/2}
    \leq \beta \bigg(\,\sum_{i,j\not\in J}E_{i,j}^2\,\bigg)^{1/2}
    \leq \beta \|E\|_F=O(\beta n^{1/4}Q^{1/2}).
\end{equation}
From this, using \Cref{lemma:W-trace_bound} and that $W=\Pi_{Y} W_0$,
$$\trace_W(E\circ E)=\trace_W(E'\circ E')
    \leq \dim(W)^{1/2}\|E'\circ E'\|_F
    \leq O(\dim(W)^{1/2}\beta n^{1/4}Q^{1/2}).$$
Hence, the total contribution from the error terms can be bounded as
\begin{align*}
    2\trace_W(A\circ E)+\trace_W(E\circ E)
    &\leq O\big(\dim(W)^{1/2}\beta n^{1/4}Q^{1/2} \big)\\
    &\leq  O\big(\dim(W)\beta n^{1/4}Q^{1/2} \big)
    \leq O\!\left(\frac{S_T^2\beta n^{1/4}Q^{1/2}}{T^2}\right).
\end{align*}
In  the second inequality, we upper bounded $\dim(W)^{1/2}$ by $\dim(W)$, which is quite wasteful, but it simplifies upcoming calculations. Putting everything together, we proved that 
\begin{align*}
\trace_W(A\circ A)&=\trace_W((A+E)\circ (A+E))-2\trace_W(A\circ E)-\trace_W(E\circ E)\\
&\geq \frac{S_T^2}{n}-\frac{QS_T^2n^2}{\beta T^4}-O\left(\frac{S_T^2\beta n^{1/4}Q^{1/2}}{T^2}\right)=S_T^2\left(\frac{1}{n}-\frac{Qn^2}{\beta T^4}-O\left(\frac{\beta n^{1/4}Q^{1/2}}{T^2}\right)\right).
\end{align*}
The parameter $\beta$ was chosen such that the two negative terms have the same order of magnitude. After substituting $\beta=\frac{Q^{1/4}n^{7/8}}{T}$, we get
$$\trace_W(A\circ A) \ge S_T^2\left(\frac{1}{n}-O\left(\frac{Q^{3/4}n^{9/8}}{T^3}\right)\right)\geq \frac{3S_T^2}{4}.$$
Here, the last inequality holds by our assumptions that $Q\leq n^{1+\gamma}$ and $T\geq Cn^{1-\frac{1}{24}+\frac{\gamma}{4}}$. 
Now, comparing the left-hand-side and right-hand-side of (\ref{equ:mainlemma1}), we conclude the desired inequality by 
\begin{align}
S_{\frac{T^2}{2n}} + \frac{S_T^2}{2n} &\ge \trace_W(A) = \trace_W(A\circ A)
\geq \frac{3S_T^2}{4n}. \qedhere
\end{align}
\end{proof}

\subsection{Recursion}\label{sect:recursion}

In this section, we show how to use the inequality $S_T^2\leq 4nS_{\frac{T^2}{2n}}$ to gain insight into the global structure of the spectrum. This section can be summed up in the motto ``\textit{large eigenvalues carry all the mass in the second moment}''.

\begin{lemma}\label{lemma:recursion}
Let $n$ be an integer, $0<\gamma<q<1$ be fixed parameters, and let $\lambda_1\geq\dots\geq \lambda_n$ be the eigenvalues of an $n$-vertex graph $G$. Assume that $\sum_{\lambda_i>0}\lambda_i\leq n^{1+\gamma}$ and $S_T^2\leq 4nS_{\frac{T^2}{2n}}$, for every $T\geq 2n^{1-q}$. Then for every $\kappa\in [0,1]$,
\begin{equation}\label{eqn:second_moment_bound}
    \sum_{0\leq \lambda_i\leq \kappa n}\lambda_i^2\leq 50\kappa^{1-\gamma/q}n^{2}.
\end{equation}
\end{lemma}
\begin{proof}
First, observe that if $\kappa\leq 2n^{-q}$, the statement is easy to prove. 
Indeed, using that $0 < \gamma < q < 1$, 
$$\sum_{0\leq \lambda_i
    \leq \kappa n}\lambda_i^2
    \leq \kappa n\sum_{\lambda_i>0}\lambda_i
    \leq \kappa n^{2+\gamma}
    \leq \kappa n^2\cdot (2/\kappa)^{\gamma/q}
    \leq 2^{\gamma/q}\kappa^{1-\gamma/q}n^2
    \leq 50\kappa^{1-\gamma/q}n^2.$$

In the rest of the proof, we assume $\kappa>2n^{-q}$. 
Define an increasing sequence $\kappa_{-1}, \kappa_0, \kappa_1, \dots$ by $\kappa_{-1}=2n^{-2q}$ and $\kappa_{i}=\sqrt{2\kappa_{i-1}}$ for $i\geq 0$. Then $\kappa_i=2n^{-q/2^i}$ for $i\geq -1$, explaining our reason to start the indexing from $-1$. We also highlight that $\kappa>\kappa_{-1}$.

First, we show inductively that $S_{\kappa_i n}\leq 8\kappa_i^{-\gamma/2q}n$. 
For $i=-1$ this is straightforward, since $$S_{\kappa_{-1}n}\leq \sum_{\lambda_j>0}\lambda_j\leq n^{1+\gamma}\leq 8\kappa_{-1}^{-\gamma/2q}n.$$
For $i\geq 0$, we have \[S_{\kappa_i n}\leq \sqrt{4nS_{\kappa_i^2n/2}}=\sqrt{4nS_{\kappa_{i-1}n}}\leq \sqrt{4n\cdot 8\kappa_{i-1}^{-\gamma/2q}n}=\sqrt{32(\kappa_i^2/2)^{-\gamma/2q}n^2}\leq 8\kappa_i^{-\gamma/2q}n,\]
where in the first inequality we use that $\kappa_i n\geq 2n^{1-q}$ for $i\geq 0$, so $S_T^2\leq 4nS_{\frac{T^2}{2n}}$ holds for $T=\kappa_i n$,  and in the second inequality we use the induction hypothesis for $S_{\kappa_{i-1}n}$.

Next, we prove (\ref{eqn:second_moment_bound}). 
We may assume $\kappa^{1-\gamma/q}\leq 1/2$, since the statement is otherwise trivial. 
Let $p$ be the largest integer such that $\kappa_p\leq \kappa$, for which we have $\kappa < \sqrt{2\kappa_p}$, i.e. $\kappa_p> \kappa^2/2$. 
Then, we can decompose the sum $\sum_{0\leq \lambda_i\leq \kappa n}\lambda_i^2$ as follows
\begin{align*}
    \sum_{0\leq \lambda_i\leq \kappa n}\lambda_i^2
    &=\underbrace{\sum_{\kappa_p n< \lambda_i\leq \kappa n}\lambda_i^2}_{\Sigma_p}
        +\underbrace{\sum_{\kappa_{p-1}n< \lambda_i\leq \kappa_p n}\lambda_i^2}_{\Sigma_{p-1}}
        +\cdots+
        \underbrace{\sum_{0< \lambda_i\leq \kappa_{-1} n}\lambda_i^2}_{\Sigma_{-2}}.
\end{align*}
We can bound $\Sigma_p$ as  $$\Sigma_p\leq \kappa n\cdot S_{\kappa_p n}\leq \kappa n\cdot 8\kappa_p^{-\gamma/2q}n\leq 16\kappa^{1-\gamma/q} n^2,$$ where in the last step we used that $\kappa_p> \kappa^2/2$. 
Similarly, for any $-1\leq \ell\leq p-1$, we have $$\Sigma_{\ell-1}\leq \kappa_{\ell}n\cdot S_{\kappa_{\ell-1}n}\leq \kappa_{\ell}n\cdot 8\kappa_{\ell-1}^{-\gamma/2q}n\leq 16\kappa_{\ell}^{1-\gamma/q}n^2,$$ where we again used that $\kappa_{\ell-1}=\kappa_{\ell}^2/2$. 
Finally, using that $\kappa_{-1}\le 2n^{-q}$, we have
$$\Sigma_{-2}\leq \kappa_{-1}n\sum_{\lambda_i>0}\lambda_i\leq \kappa_{-1}n^{2+\gamma}\leq 2\kappa_{-1}^{1-\gamma/q} n^{2}.$$
Combining all of this, we obtain
\[\sum_{0\leq \lambda_i\leq \kappa n}\lambda_i^2\leq 16\kappa^{1-\gamma/q}n^2+\sum_{\ell\leq p}16\kappa_{\ell}^{1-\gamma/q}n^2+2\kappa_{-1}^{1-\gamma/q}n^2.\]
We note that $\kappa_{\ell}^{1-\gamma/q}\leq 2^{-(p-\ell)}\kappa^{1-\gamma/q}$, which follows by simple induction and our assumption that $\kappa^{1-\gamma/q}<1/2$. 
Thus, the right-hand-side of the above inequality is less than $50\kappa^{1-\gamma/q}n$.
\end{proof}

\section{Densification --- Phase 1}\label{sect:dens1}

\Cref{lemma:recursion}, combined with earlier results, has a number of powerful consequences.
One of them asserts that a graph with large smallest eigenvalue (or small surplus) contains a large subgraph of positive density.
We prove this via the following density increment argument.
We apply \Cref{lemma:recursion} to show that the cubic sum of eigenvalues is large, which in turn coincides with six times the number of triangles.
But if a graph has too many triangles, then some vertex has a very dense neighbourhood.
So we can repeatedly pass to this neighbourhood to eventually obtain a large subgraph with constant density.
The main step of this argument is presented in the next lemma.

\begin{lemma}\label{lemma:dens1_triangle_counting}
Let $\gamma\in (0,1/6)$, $C>2$, and let $G$ be a $n$-vertex graph with  edge density $p>n^{-1/2}$, $\Delta(G)\leq Cpn$, and smallest eigenvalue $\lambda_n$ satisfying $|\lambda_n|\leq n^\gamma$. Then $G$ has a subgraph on at least $pn$ vertices of edge density at least $c_0 p^{2\gamma/(1-4\gamma)}$ for some $c_0=c_0(\gamma,C)>0$.
\end{lemma}
\begin{proof}
    We may and will assume $n$ is sufficiently large.
    Let $m=p\binom{n}{2}$ denote the number of edges. Let $\lambda_1\geq\dots\geq\lambda_{n}$ be the eigenvalues of $G$, then $$\sum_{\lambda_i>0}\lambda_i=\sum_{\lambda_i<0}|\lambda_i|\leq n|\lambda_n|\leq n^{1+\gamma}.$$
    By Lemma \ref{lemma:main_smallest_eigenalue}, we also have
    $$4nS_{\frac{T^2}{2n}}\geq S_T^2$$
    for every $T\geq 2n^{1/2+\gamma}$. 
    So applying Lemma \ref{lemma:recursion} with $q=1/2-\gamma>\gamma$ and the sequence of positive eigenvalues, we get that for every $\kappa\in [0,1]$,
    $$\sum_{0<\lambda_i\leq \kappa n}\lambda_i^2\leq c\kappa^{1-\gamma/q}n^{2}$$
    for an absolute constant $c>1$. 
    Write $u:=\gamma/q<1/2$ and set $\kappa:=(8c)^{-1/(1-u)}p^{1/(1-u)}$.
    Then $\kappa\leq 1$ and the right-hand-side equals $pn^2/8$. On the other hand, we have $$\sum_{\lambda_i<0} \lambda_i^2\leq n\lambda_n^2\leq n^{1+2\gamma}<pn^2/8,$$ so
    $$\sum_{\lambda_i\leq \kappa n} \lambda_i^2\leq pn^2/4.$$
    But $\sum_{i=1}^n\lambda_i^2=\|A\|_F^2=2m$, so we conclude that
    $$\sum_{\lambda_i>\kappa n}\lambda_i^2\geq 2m-\sum_{\lambda_i\leq H}\lambda_i^2\geq pn^2/2.$$
    Writing $N$ for the number of triangles, we have
    $$6N=\sum_{i=1}^n\lambda_i^3
        \geq \kappa n\sum_{\lambda_i> \kappa n}\lambda_i^2-\sum_{\lambda_i<0}|\lambda_i|^3
        \geq \kappa pn^3/2-n|\lambda_n|^3\geq c'p^{\frac{2-u}{1-u}}n^{3}-n^{1+3\gamma}$$
    for some constant $c'>0$ depending only on $\gamma$. 
    Here, using that $\gamma<1/6$, $u<1/2$ and $p>n^{-1/2}$, we have $n^{1+3\gamma}< \frac{c'}{2}p^{\frac{2-u}{1-u}}n^{3}$. 
    Hence, we get
    $$N\geq \frac{c'}{12}p^{\frac{2-u}{1-u}}n^{3}.$$
    Counting triangles by vertices, we observe that there is a vertex $v\in V(G)$ whose neighbourhood $X$ contains at least $\frac{3N}{n}$ edges.
    In addition, $|X|\leq \Delta(G)\leq Cpn$. 
    Now, let $X'$ be an arbitrary superset of $X$ containing $\max(pn,|X|)$ elements.
    Then, the edge density of $G[X']$ is 
    $$\frac{3N/n}{\binom{|X'|}{2}}
        > \frac{6N/n}{|X'|^2} 
        \ge \frac{c'p^{\frac{2-u}{1-u}}n^{2}/2}{C^2p^2n^2}
        = \frac{c'}{2C^2}p^{\frac{u}{1-u}}.$$
    As $\frac{u}{1-u}=\frac{\gamma}{q-\gamma}=\frac{2\gamma}{1-4\gamma}$, this finishes the proof.
\end{proof}

In the next lemma, we show how to handle the case when $G$ has some vertices of too large degrees.

\begin{lemma}\label{lemma:dens1_high_degrees}
Let $C>2$ and $G$ be an $n$-vertex graph of average degree $d$ with $\surp(G)\leq \frac{dn}{100}$. Then either 
\begin{itemize}
    \item[(i)] $G$ contains a subgraph on $n/C$ vertices of average degree at least $Cd/5$.
    \item[(ii)]$G$ contains a subgraph on at least $n/2$ vertices with average degree at least $d/4$, and maximum degree at most $Cd$.
\end{itemize}
\end{lemma}

\begin{proof}
    Let $X\subset V(G)$ be the set of vertices of degree more than $Cd$, then $|X|\leq n/C$. 
    Let $Y=V(G)\setminus X$. 
    The maximum degree of $G[Y]$ is at most $Cd$, so if $G[Y]$ has average degree at least $d/4$, then (ii) holds.
    Otherwise, $G[Y]$ has at most $nd/8$ edges. 
    Moreover, $e(G[X,Y])\leq \frac{e(G)}{2} +\surp(G)\le \frac{dn}{4}+\frac{dn}{100} = 0.26dn$.
    Hence, $G[X]$ contains at least $dn/2-dn/8-0.26dn>dn/10$ edges. 
    Let $X_0$ be any superset of $X$ of size $n/C$.
    Then, $X_0$ has average degree at least $Cd/5$, so $G[X_0]$ satisfies (i).
\end{proof}

As mentioned before, we will repeatedly apply \Cref{lemma:dens1_triangle_counting,lemma:dens1_high_degrees} to pass to denser neighbourhoods and eventually obtain a large subgraph with constant density.
This is our main result for the smallest eigenvalue in this section.
However, to make the constants work nicely, we adopt a slightly different proof by considering a subgraph that maximises a carefully chosen potential function.
If its density is $o(1)$, \Cref{lemma:dens1_triangle_counting,lemma:dens1_high_degrees} guarantee a subgraph whose potential function is even higher, which is impossible.
\begin{lemma}\label{lemma:dens1_main}
    Let $\gamma,\eps,\rho > 0$ satisfy $\rho < \frac{1}{2}$, $\eps + 6\gamma < 1$ and $\frac{\rho}{\eps} + \frac{2\gamma}{1-\eps-4\gamma} < 1$. 
    Then there exist $c_1=c_1(\gamma,\eps,\rho)>0$ and $n_0=n_0(\gamma,\eps,\rho)$ such that the following holds for every $n > n_0$.
    Let $G$ be an $n$-vertex graph with edge density $p$ and smallest eigenvalue $\lambda_n$ such that $p>n^{-\rho}$ and $|\lambda_n|\leq n^\gamma$.
    Then $G$ has a subgraph on at least $n^{1-\eps}$ vertices with edge density at least $c_1$. 
\end{lemma}
\begin{proof}
    We may and will assume that $n$ is sufficiently large. 
    For a graph $H$, we write $v(H)$ for the number of its vertices, $d(H)$ for its average degree, and $p(H)$ for its density.

    Let $\gamma' \in (\gamma,\frac{1-\eps}{6})$ be any constant such that $\frac{\rho}{\eps} + \frac{2\gamma'}{1-\eps-4\gamma'} < 1$; then $|\lambda_n| \le n^{\gamma'}/2$.
    Let $H$ be an induced subgraph of $G$ that maximizes the function $$v(H)^{\rho/\eps}\cdot p(H).$$ 
    We show that $H$ is the desired subgraph. 
    Due to the maximality,  $v(H)^{\rho/\eps}\cdot p(H) \ge v(G)^{\rho/\eps}\cdot p(G)>n^{\rho/\eps-\rho}$.
    Since $p(H) \le 1$, we get $v(H) \geq  n^{1-\eps}$, as desired. 
    In what follows, we lower bound the density of $H$.
    Recall that Cauchy's interlacing theorem ensures that the smallest eigenvalue of $H$ is at least $-n^{\gamma'}/2$.
    
    First, we show that $p(H) > v(H)^{-\rho}$.
    Indeed, if $p(H) \le v(H)^{-\rho}$, then $$v(H)^{\rho/\eps-\rho}\ge v(H)^{\rho/\eps}\cdot p(H) \ge v(G)^{\rho/\eps}\cdot p(G)>n^{\rho/\eps-\rho}.$$
    This contradicts $v(H) \le n$. 
    Therefore, $p(H) > v(H)^{-\rho}$. 
    Next, we prove that $d(H) > n^{3\gamma'}$. Recall that $\rho<\frac{1}{2},\gamma'<\frac{1-\eps}{6}$ and $v(H) > n^{1-\eps}$ is sufficiently large.
    We have $$
        d(H) = \big(v(H)-1 \big)p(H) > \big(v(H)-1 \big)v(H)^{-\rho}>v(H)^{1/2}> n^{(1-\eps)/2}>n^{3\gamma'}.
    $$
    Now apply \Cref{lemma:dens1_high_degrees} to $H$ with $C=5$.
    The requirement $\surp(H) \le \frac{1}{100}d(H) v(H)$ trivially holds as by \Cref{claim:surplus_and_lambdan}, we have $\surp(H) \le n^{\gamma'}\cdot v(H)/4< \frac{1}{100}d(H)v(H)$.
    By \Cref{lemma:dens1_high_degrees}, either (i) $H$ contains an induced subgraph $H_0$ such that $v(H_0) = v(H)/5$ and $d(H_0) \ge d(H)$, or (ii) $H$ contains an induced subgraph $H_1$ such that $v(H_1) \ge v(H)/2, d(H_1) \ge d(H)/4$ and $\Delta(H_1) \le 5d(H)$.
    
    Assume that (i) holds. Using that $v(H)>n^{1-\eps}$ is sufficiently large, we have $$p(H_0) = \frac{d(H_0)}{v(H_0)-1} \ge \frac{d(H)}{v(H)/5-1}=(1-o(1))5p(H).$$
    Therefore, $$v(H_0)^{\rho/\eps}\cdot p(H_0)
        \ge 5^{-\rho/\eps}v(H)^{\rho/\eps}\cdot (1-o(1))5p(H)
        = (1-o(1)){5^{1-\rho/\eps}}\cdot v(H)^{\rho/\eps}\cdot p(H).$$
    As $\rho/\eps < 1$, this contradicts the maximality of $H$. 
   
    Hence, (ii) must hold. 
    Then, we apply \Cref{lemma:dens1_triangle_counting} to $H_1$ with $C=20$.
    To be able to apply this lemma, we first verify the requirements on $p(H_1)$, $\Delta(H_1)$ and the smallest eigenvalue. Since $\rho< 1/2$ and $v(H_1) \ge v(H) / 2$, we have $p(H) > v(H)^{-\rho} > 8\cdot v(H_1)^{-1/2}$. 
    Thus, $$p(H_1) = \frac{d(H_1)}{v(H_1)-1} \ge \frac{d(H)}{4v(H)}\ge \frac{p(H)}{8}>\frac{1}{8}v(H)^{-\rho}>v(H_1)^{-1/2}.$$
    Furthermore, $\Delta(H_1) \le 5d(H) \le 20d(H_1)$.
    In addition, the smallest eigenvalue of $H_1$ has absolute value at most $n^{\gamma'}/2$. 
    As $v(H_1) \ge v(H)/2 > n^{1-\eps}/2$ and $\gamma'/(1-\eps) < 1$, this is at most $\big(2v(H_1)\big)^{\gamma'/(1-\eps)}/2 < v(H_1)^{\gamma'/(1-\eps)}$.
    Here, $\frac{\gamma'}{1-\eps} <1/6$ as $\gamma' < \frac{1-\eps}{6}$.   
    Therefore, we can apply \Cref{lemma:dens1_triangle_counting} to get an induced subgraph $H_2$ of $H_1$ such that $v(H_2) \ge p(H_1)v(H_1)$ and $p(H_2) \ge c_0 \cdot p(H_1)^{\frac{2\gamma'/(1-\eps)}{1-4\gamma'/(1-\eps)}}=c_0p(H_1)^{\frac{2\gamma'}{1-\eps-4\gamma'}}$, where $c_0=c_0(\gamma',\eps) > 0$.
    Hence, $$
        v(H_2)^{\rho/\eps}\cdot p(H_2)
        \ge \big[(p(H_1)v(H_1) \big]^{\rho/\eps}\cdot \Big[c_0 \cdot p(H_1)^{\frac{2\gamma'}{1-\eps-4\gamma'}}\Big]
        = c_0\cdot v(H_1)^{\rho/\eps}\cdot p(H_1)^{\frac{\rho}{\eps}+\frac{2\gamma'}{1-\eps-4\gamma'}}.
    $$
    Recall that $v(H_1) \ge v(H)/2$ and $p(H_1) \ge p(H)/5$. Hence,
   $$
        v(H_2)^{\rho/\eps}\cdot p(H_2)
        \ge c_0'\cdot v(H)^{\rho/\eps}\cdot p(H)^{\frac{\rho}{\eps}+\frac{2\gamma'}{1-\eps-4\gamma'}}    
    $$
    for some  $c_0'=c_0'(\gamma',\eps,\rho)>0$.  Since $H$ maximizes the function $v(H)^{\rho/\eps}\cdot p(H)$, we have 
    $$        v(H)^{\rho/\eps}\cdot p(H)
        \ge v(H_2)^{\rho/\eps}\cdot p(H_2)
        \ge c_0'\cdot v(H)^{\rho/\eps}\cdot p(H)^{\frac{\rho}{\eps}+\frac{2\gamma'}{1-\eps-4\gamma'}}.$$
    Using the condition $1>\frac{\rho}{\eps}+\frac{2\gamma'}{1-\eps-4\gamma'}$, this implies that $p(H) \ge (c_0')^{1/\big(1-\frac{\rho}{\eps}-\frac{2\gamma'}{1-\eps-4\gamma'}\big)}$.
    This completes the proof by taking $c_1=(c_0')^{1/\big(1-\frac{\rho}{\eps}-\frac{2\gamma'}{1-\eps-4\gamma'}\big)}$.
\end{proof}

Next, we prove a counterpart of \Cref{lemma:dens1_main} for the surplus.
For this purpose, we require a surplus version of \Cref{lemma:dens1_triangle_counting} as follows.
As the proofs are more or less the same, with only some parameters changed, we only highlight the key differences.
\begin{lemma}\label{lemma:dens1_triangle_count_surplus}
    Let $\gamma\in (0,1/60)$, $C>2$, and let $G$ be a $n$-vertex graph with  edge density $p>n^{-1/3}$, $\Delta(G)\leq Cpn$, and $\surp(G)\leq n^{1+\gamma}$. 
    Then $G$ has a subgraph on at least $pn$ vertices of edge density at least $c_0 p^{4/5}$ for some $c_0=c_0(\gamma,C)>0$.
\end{lemma}
\begin{proof}
    Assume $n$ is sufficiently large.
    By \cref{claim:groth}, $\surp^*(G) = O(\surp(G)\log n) < n^{1+\gamma'}$ for some constant $\gamma' \in (\gamma,1/60)$.
    Then, \Cref{lemma:surp_star} (i) implies
    $$\sum_{\lambda_i>0}\lambda_i= \sum_{0<\lambda_i}|\lambda_i|\leq \surp^*(G)\leq n^{1+\gamma'}.$$
    By \Cref{lemma:main_MaxCut}, we also have
    $$4nS_{\frac{T^2}{2n}}\geq S_T^2$$
    for every $T>C_0n^{1-\frac{1}{24}+\frac{\gamma'}{4}}$.
    So we can apply \Cref{lemma:recursion} with $q=3/80>\gamma'$ and the sequence of positive eigenvalues to get that for every $\kappa \leq 1$,
    $$\sum_{0<\lambda_i\leq \kappa n}\lambda_i\leq c\kappa^{1-\gamma'/q}n^{2}.$$
    Furthermore, by \Cref{lemma:surp_star} (ii), we have
    $$\sum_{\lambda_i<0}\lambda_i^2\leq O\big(n^{1/2}\surp^*(G) \big)=O(n^{3/2+\gamma'}).$$
    Write $u:=\gamma'/q<4/9$, and set $\kappa:=(8c)^{-1/(1-u)}p^{1/(1-u)}$.
    We get $\sum_{\lambda_i<\kappa n}\lambda_i^2\leq pn^{2}/4$.
    Thus, if $N$ is the number of triangles, then
    $$6N=\sum_{i=1}^n\lambda_i^3\geq \kappa n\sum_{\lambda_i\geq H}\lambda_i^2-\sum_{\lambda_i<0}|\lambda_i|^3\geq \kappa pn^3/2-O\big(n\surp^*(G) \big)\geq c'p^{\frac{2-u}{1-u}}n^{3}-n^{2+\gamma'}.$$
    Here, we used \Cref{lemma:surp_star} (iii) in the second inequality.
    The rest of the proof is identical to that of \Cref{lemma:dens1_triangle_counting}.
\end{proof}

Now, we state the counterpart of \Cref{lemma:dens1_main} for surplus -- the other main result of this section.
The proof of \Cref{lemma:dens1_main_surplus} follows almost identically to that of \Cref{lemma:dens1_main}, except that we use \Cref{lemma:dens1_triangle_count_surplus} in place of \Cref{lemma:dens1_triangle_counting}.
\begin{lemma}\label{lemma:dens1_main_surplus}
    Let $\gamma\in (0,1/60)$, $\eps\in (0,1/2)$.
    Then there exist $\rho=\rho(\gamma,\eps)>0$ and $c_1=c_1(\gamma,\eps)>0$ such that the following holds for every $n>n_0(\gamma,\eps)$. 
    Let $G$ be an $n$-vertex graph with edge density $p>n^{-\rho}$ and $\surp(G)\leq n^{1+\gamma}$.
    Then $G$ has a subgraph on at least $n^{1-\eps}$ vertices with edge density at least~$c_1$.
\end{lemma}

% \begin{proof}
% The proof of this is almost identical to the proof of Lemma \ref{lemma:dens1_main}, but we use Lemma \ref{lemma:dens1_triangle_count_surplus} instead of Lemma \ref{lemma:dens1_triangle_counting}. We omit further details.
% \end{proof}

\section{Densification --- Phase 2}\label{sect:dens2}

In this section, we prove that graphs of positive constant density and large smallest eigenvalue (or small surplus) are $o(1)$-close to the disjoint union of cliques. In particular, this implies that such graphs must contain subgraphs of density $1-o(1)$. 

In the first step of the proof, we use our main lemmas, Lemma \ref{lemma:main_smallest_eigenalue} and \ref{lemma:main_MaxCut}, to show that the adjacency matrix of a graph is close to a constant-rank matrix, in Frobenius norm. 

In the second step, we will use this approximation to show that $G$ admits an \emph{ultra-strong regularity partition}. Ultra-strong regularity partition is a partition of $V(G)$ where almost all pairs of parts have very high or very low density of edges between them. These are closely related to Szemer\'edi's regularity lemma, but they provide substantially stronger quantitative bounds. Ultra-strong regularity lemmas first appeared in relation to graphs of bounded VC-dimension; see the seminal work of Lov\'asz and Szegedy~\cite{LSz}. Our approach to finding this regularity partition is morally similar to the spectral approach of Frieze and Kannan~\cite{FK} (see also~\cite{T12}), but the good understanding of the spectrum of $G$ coming from Section~\ref{sect:main_lemmas} allows us to obtain a much stronger quantitative result.

Finally, in the last step of the proof, we analyse the regularity partition obtained from the previous step and show that it contains very few induced paths of length 2, i.e. cherries. This shows that the whole graph is close to the union of cliques. We now give the details. 

\begin{lemma}\label{lemma:dens2_finding_low_rank_approximation}
Let $\gamma\in (0, 1/4)$, $\eps>0$, and let $n$ be sufficiently large. If $G$ is an $n$-vertex graph with adjacency matrix $A$ and with $|\lambda_n|\leq n^{\gamma}$, then there is a matrix $B$ of rank $O_{\gamma,\eps}(1)$ such that $\|A-B\|_F^2 \leq \eps n^2$.
\end{lemma}
\begin{proof}  
    Let $A=\sum_{i=1}^n \lambda_i v_iv_i^T$ be the spectral decomposition of $A$.
    We have 
    $$\sum_{\lambda_i>0}\lambda_i=\sum_{\lambda_i<0}|\lambda_i|\leq |\lambda_n|n\leq n^{1+\gamma}.$$
    Also, \Cref{lemma:main_smallest_eigenalue} implies $4nS_{\frac{T^2}{2n}}\geq S_T^2$ for every $T\geq 2n^{1/2+\gamma}\geq 2|\lambda_n|\sqrt{n}$. Hence, we can apply Lemma~\ref{lemma:recursion} to the sequence of positive eigenvalues with $q:=1/4>\gamma$ to conclude that for every $\kappa \in (0,1)$, we have
    $$\sum_{0<\lambda_i<\kappa n}\lambda_i^2\leq O(\kappa^{1-4\gamma}n^{2}).$$
    Furthermore, $\sum_{\lambda_i<0}\lambda_i^2\leq n|\lambda_n|^2<n^{3/2}$, so 
    $$\sum_{\lambda_i<\kappa n}\lambda_i^2\leq O(\kappa^{1-4\gamma}n^{2})+n^{3/2}.$$
    Hence, we can choose $\kappa$ (depending only on $\gamma$ and $\eps$) such that $\sum_{\lambda_i<\kappa n}\lambda_i^2\leq \eps n^2.$ Having chosen $\kappa$, set $B=\sum_{\lambda_i\geq \kappa n}\lambda_i v_iv_i^T$. 
    It satisfies that
    $$\|A-B\|_{F}^2=\sum_{\lambda_i<\kappa n}\lambda_i^2\leq \eps n^2.$$
    Furthermore, the rank of $B$ is at most $\kappa^{-2}$ as $n^2 \ge \|A\|_F^2 \ge \sum_{\lambda_i \ge \kappa n} \lambda_i^2 \ge \rank(B)\cdot (\kappa n)^2$.
    % $r:=N_H\leq \frac{n^2}{(\kappa n)^2}=\kappa^{-2}$, where in the first inequality we used that  $n^2\geq \sum_{\lambda_i\geq \kappa n}\lambda_i^2\geq (\kappa n)^2 |\{i:\lambda_i>\kappa n\}|$.
\end{proof}

Given a graph $G$, $\delta\in (0,1)$, and two disjoint sets $X,Y\subset V(G)$, the pair $(X,Y)$ is \emph{$\delta$-empty} if there are at most $\delta |X||Y|$ edges between $X$ and $Y$. Also, $(X,Y)$ is \emph{$\delta$-full} if there are at least $(1-\delta) |X||Y|$ edges between $X$ and $Y$. Then $(X,Y)$ is \emph{$\delta$-homogeneous} if it is either $\delta$-empty or $\delta$-full. 
A \emph{$\delta$-regular partition} of $G$ is an equipartition (a partition where all parts share the same size) $V_1,\dots,V_K$ of the vertex set such that  all but at most $\delta K^2$ of the pairs $(V_i,V_j)$ for $1\leq i<j\leq K$ are $\delta$-homogeneous. 

\begin{lemma}\label{lemma:dens2_low_rank_approx}
For every $\delta\in(0,1)$, there exists $\eps>0$ such that the following holds for every positive integer $r$, and every $n$ that is sufficiently large with respect to $\delta,r$. Let $G$ be a graph with adjacency matrix  $A$. Assume that there exists an $n\times n$ symmetric matrix $B$ of rank $r$ such that $\|A-B\|_F^2\leq \eps n^2$. Then $G$ has a $\delta$-regular partition into $K$ parts, where  $1/\delta<K<O_{r,\delta}(1)$.
\end{lemma}

\begin{proof}
We show that $\eps=\delta^2/100$ suffices. Let $B=\sum_{i=1}^r\mu_i w_iw_i^T$ be the spectral decomposition of $B$. Then 
$$\left(\sum_{i=1}^r\mu_i^2\right)^{1/2}=\|B\|_F\leq \|A\|_F+\|B-A\|_F<2n,$$
which shows that $|\mu_i|\leq 2n$ for all $i\in [r]$. Next, we group the coordinates of the vectors $w_1,\dots,w_r$ with respect to how close they are, which then we use to form a partition of $B$ into submatrices that are close to constant matrices. 

Pick $\beta:=10^{-3}\delta^{1/2}r^{-3/2}$.
For $i \in [r]$ and $\ell\in\mathbb{Z}$, let 
$$X_{i,\ell}=\left\{j\in [n]: \frac{\beta}{\sqrt{n}}\ell\leq w_i(j)<\frac{\beta}{\sqrt{n}}(\ell+1)\right\}.$$
That is, for fixed $i\in [n]$, the sets $X_{i,\ell}$ form a partition of the coordinates of $w_i$ into chunks that are close to constant. Next, we show that most coordinates of $w_i$ are covered by $O_{r,\delta}(1)$ of these sets. Set $h:=10^4r^2/\delta$. As $\sum_{j=1}^nw_i(j)^2=1$, the number of $j\in [n]$ not contained in $\bigcup_{\ell=-h}^hX_{i,\ell}$ is at most $n/(h^2\beta^2)< \frac{\delta n}{8r}$. 

Let $I=\{-h,\dots,h\}^r$. For every $\overline{\ell}\in I$, let $X_{\overline{\ell}}=\bigcap_{i\in [r]}X_{i,\overline{\ell}(i)}.$ Then 
\begin{equation}\label{equ:sparse_entries}
    \bigcup_{\overline{\ell}\in I}X_{\overline{\ell}}
    \ge n - r \cdot \frac{\delta n}{8r}
    \geq n\Big(1-\frac{\delta}{8}\Big).
\end{equation}
Thus, the sets $X_{\overline{\ell}}$ form a disjoint covering of all but at most $\delta n/8$ of the indices. 
Next, our goal is to show that if $\overline{\ell}_1,\overline{\ell}_2\in I$, then the submatrix of $B$ induced on $X_{\overline{\ell}_1}\times X_{\overline{\ell}_2}$ is close to a constant matrix. 
We refer to the rectangles $X_{\overline{\ell}_1}\times X_{\overline{\ell}_2}$ as \emph{blocks}.  
Let $$\eta=\eta_{\overline{\ell}_1,\overline{\ell}_2}=\sum_{i=1}^r \mu_i\cdot \frac{\beta^2}{n}\cdot\overline{\ell}_1(i)\overline{\ell}_2(i).$$
Note that for every $(j_1,j_2)\in X_{\overline{\ell}_1}\times X_{\overline{\ell}_2}$, we have $$\left|w_i(j_1)w_i(j_2)-\frac{\beta^2}{n}\overline{\ell_1}(i)\overline{\ell_2}(i)\right| \leq \frac{\beta^2}{n}\cdot 4h,$$
which we get from the general inequality $|ab-cd|\leq |a||b-d|+|d||a-c|$. 
Using that $|\mu_i| \le 2n$,  we have
\begin{equation}\nonumber %\label{equ:non_homogeneous}
    |B_{j_1,j_2}-\eta|\leq \sum_{i=1}^r |\mu_i|\cdot \left|w_i(j_1)w_i(j_2)-\frac{\beta^2}{n}\overline{\ell_1}(i)\overline{\ell_2}(i)\right|\leq \sum_{i=1}^r |\mu_i|\cdot \frac{\beta^2}{n}\cdot 4h\leq 8r\beta^2h <\frac{1}{3}.
\end{equation}
Furthermore, observe that if $X\subset X_{\overline{\ell}_1}$ and $Y\subset X_{\overline{\ell}_2}$ are such that $(X,Y)$ is not $\delta$-homogeneous, then $$\big\|A[X\times Y]-B[X\times Y]\big\|_F^2\geq \frac{\delta}{36}|X|| Y|.$$
Indeed, if $\eta_{\overline{\ell}_1,\overline{\ell}_2}<1/2$, then $A_{j_1,j_2}-B_{j_1,j_2}\geq 1/6$ for every $A_{j_1,j_2}=1$, otherwise $|A_{j_1,j_2}-B_{j_1,j_2}|\geq 1/6$ for every $A_{j_1,j_2}=0$.

Now let $K=|I|/(8\delta)$, and define an  equipartition $V_1,\dots,V_K$  of $V(G)$ as follows. To avoid certain technicalities coming from rounding, we assume that $K$ divides $n$. Let $V^*$ be the set of elements not covered by any of the $X_{\overline{\ell}}$ for $\overline{\ell}\in I$. 
For each $\overline{\ell}\in I$, partition $X_{\overline{\ell}}$ arbitrarily into sets, each of size $n/K$, with at most one exceptional set whose size is less than $n/K$.  Move the elements of the exceptional set to $V^*$. Then finally partition $V^*$ into sets of size $n/K$. Let $V_1,\dots,V_K$ be the collection of all the sets in these partitions. Each $X_{\overline{\ell}}$ contributes at most $n/K$ elements to $V^*$, so in the end we have $|V^*|\leq \delta n/8+|I|\cdot (n/K)\leq \delta n/4$. 
Therefore, at most $\delta K/4$ sets $V_i$ are contained in $V^*$.  We show that $V_1,\dots,V_K$ is a $\delta$-regular partition.  

Assume that $(V_i,V_j)$ is not $\delta$-homogeneous. There are at most $\delta K^2/2$ such pairs where either $V_i\subset V^*$ or $V_j\subset V^*$. On the other hand, if $V_i,V_j\not\in V^*$, then $\|A[V_i\times V_j]-B[V_i\times V_j]\|_F^2\geq \frac{\delta}{36}|V_i||V_j|$. As $\|A-B\|_F^2\leq \eps n^2$, this means that the number of such pairs is at most $36\eps/\delta K^2\leq \delta K^2/2$. Hence, the total number of pairs that are not $\delta$-homogeneous is at most $\delta K^2$, as desired.
\end{proof}

An important feature of Lemma \ref{lemma:dens2_low_rank_approx} that $\eps$ only depends on $\delta$, and not on $r$. 
To continue from this point, we observe that if $X,Y,Z$ are sets of linear sizes such that $(X,Y)$ and $(Y,Z)$ are $\delta$-full, then $(X,Z)$ cannot be $\delta$-empty, assuming $\surp(G)$ is small.

\begin{lemma}\label{lemma:dens2_3parts}
    Let $G$ be a graph on $n$ vertices.
    Let $X,Y,Z\subset V(G)$ be disjoint sets such that $|X|=|Y|=|Z|$ and $(X,Y)$ and $(Y,Z)$ are $\delta$-full, but $(X,Z)$ is $\delta$-empty. Then $\surp(G)\geq (1/4-3\delta)|X|^2$.
\end{lemma}
\begin{proof}
    Let $G'=G[X\cup Y\cup Z]$, and consider the cut $(Y,X\cup Z)$ in $G'$. This cut has at least $|X|^2(2-2\delta)$ edges. 
    On the other hand, $e(G')\leq \frac{3}{2}|X|^2+2|X|^2+\delta|X|^2\leq (\frac{7}{2}+\delta)|X|^2$. Therefore, 
    \[\surp(G)\geq \surp(G')\ge e\big(G[Y,X\cup Z] \big)-\frac{e(G')}{2}\geq  |X|^2(2-2\delta)-\left(\frac{7}{4}+\frac{\delta}{2}\right)|X|^2\ge|X|^2\left(\frac{1}{4}-3\delta\right).\hfill \qedhere\]
\end{proof}

A graph is the disjoint union of cliques if and only if it does not contain an induced \emph{cherry}, that is, the path of length $2$. Therefore, by the \emph{induced graph removal lemma} \cite{AFKSz}, being close to the disjoint union of cliques is equivalent to having few cherries. For the special case of cherries, one does not need the full power of this lemma, and a simple proof of the following quantitatively stronger bound is given by Alon and Shapira \cite{AS06}.

\begin{lemma}\label{lemma:dens2_removal}
Let $G$ be an $n$-vertex graph containing at most $\eps n^3$ cherries. Then $G$ is $\eps^{c}$-close to the disjoint union of cliques for some absolute constant $c>0$.

Furthermore, if $G$ is $\delta$-close to the union of cliques, then $G$ contains at most $3\delta n^3$ cherries.
\end{lemma}

\begin{proof}
The first part follows from Alon and Shapira \cite{AS06}, so we only prove the second part. Let $\widetilde{G}$ be the disjoint union of cliques that is $\delta$-close to $G$. Then each cherry of $G$ contains at least one edge or non-edge from $\widetilde{G}\Delta G$, so we are  done.
\end{proof}

Now we are ready to prove Theorem \ref{thm:MAIN_SE1}, which we restate here for convenience.

\begin{theorem}
Let $\gamma\in (0, 1/4)$, $\delta>0$, and let $n$ be sufficiently large. If $G$ is an $n$-vertex graph with $|\lambda_n|\leq n^{\gamma}$, then $G$ is $\delta$-close to the vertex-disjoint union of cliques. 
\end{theorem}
\begin{proof}
    Let $\delta_0>0$ be a sufficiently small constant, depending only on $\delta$. Let $\eps$ be the constant required so that Lemma~\ref{lemma:dens2_low_rank_approx} would hold with the parameter $\delta_0$. 
    By Lemma~\ref{lemma:dens2_finding_low_rank_approximation}, there is a matrix $B$ of rank $r=O_{\gamma,\eps}(1)$ such that $\|A-B\|_F^2\leq \eps n^2$. Hence, we can apply Lemma \ref{lemma:dens2_low_rank_approx} to conclude that there is a $\delta_0$-regular partition $V_1,\dots,V_K$ for some $K$ with $1/\delta_0< K <O_{r,\delta_0}(1)=O_{\gamma,\delta}(1)$.

    In order to finish the proof, we count cherries. Let $x,y,z$ be the vertices of a cherry with $xy,yz\in E(G), xz\not\in E(G)$, and let $x\in V_i, y\in V_j, z\in V_{k}$. We put this cherry into one of the following categories:
    \begin{itemize}
        \item[(i)] $i,j,k$ are not all distinct,
        \item[(ii)] $(V_i,V_j)$ or $(V_j,V_k)$ or $(V_i,V_k)$ is not $\delta_0$-homogeneous,
        \item[(iii)] $(V_i,V_j)$ or $(V_j,V_k)$ is $\delta_0$-empty,
        \item[(iv)]  $(V_i,V_k)$ is $\delta_0$-full.
    \end{itemize}
    By \Cref{lemma:dens2_3parts}, we cannot have that $(V_i,V_j)$ and $(V_j,V_k)$ are $\delta_0$-full, but $(V_i,V_k)$ is $\delta_0$-empty. Therefore, each cherry belongs to one of the four categories. We observe that the number of cherries belonging to each category is at most $O(\delta_0 n^3)$. 
    Indeed, for (i), there are $O(K^2)$ choices for the set $\{i,j,k\}$, and then there are at most $(n/K)^3$ choices for $x,y,z$, so in total $O(K^2 (n/K)^3)=O(n^3/K)=O(\delta_0 n^3)$.
    For (ii), we use the fact that at most $\delta_0 K^2$ pairs $(V_i,V_j)$ are not $\delta_0$-homogeneous to derive that the number of choices for $(V_i,V_j,V_k)$ is $O(\delta_0 K^3)$. So the number of cherries belonging to (ii) is $O(\delta_0 K^3(n/K)^3)=O(\delta_0 n^3)$.
    For (iii) and (iv), we observe that if we fixed $(V_i,V_j,V_k)$, then there are at most $\delta_0 (n/K)^3$ choices for $x,y,z$. Indeed, if say $(V_i,V_j)$ is $\delta_0$-empty, the pair $(x,y)$ can be chosen from only the $\delta_0 (n/K)^2$ edges between $V_i$ and $V_j$. So the number of cherries belonging to (iii) or (iv) is $O(\delta_0n^3)$.

    In conclusion, the number of cherries in $G$ is $O(\delta_0 n^3)$. But then by Lemma \ref{lemma:dens2_removal}, $G$ is $O(\delta_0)^c$-close to a disjoint union of cliques for some absolute constant $c>0$. We are done by setting $\delta_0>0$ sufficiently small with respect to $\delta$.
\end{proof}

The following immediate corollary of this lemma will be used later.

\begin{corollary}\label{cor:dens2_main}
Let $\gamma\in (0,1/4)$, $p>0$ and $\delta>0$, then the following holds for every sufficiently large $n$. Let $G$ be an $n$-vertex graph of edge density $p$ such that $|\lambda_n|\leq n^{\gamma}$. Then $G$ contains a subgraph on at least $pn/2$ vertices of edge density at least $1-\delta$.
\end{corollary}

\begin{proof}
    Let $\delta_0=\delta p^2/16$. 
    By \Cref{thm:MAIN_SE1}, $G$ is $\delta_0$-close to some graph $H$ that is the disjoint union of cliques. Let $C_1,\dots,C_k$ be the vertex sets of the cliques forming $H$, then 
    $$e(H)=\sum_{i=1}^k\binom{|C_i|}{2}\leq \sum_{i=1}^k\frac{|C_i|^2}{2}\leq \frac{n}{2}\cdot \max_{i\in [k]} |C_i|.$$ 
    As $e(H)\geq e(G)-\delta_0 n^2\geq pn^2/4$, this shows that at least one of the $C_i$'s has size at least $pn/2$.
    Without loss of generality, say $|C_1|\geq pn/2$. 
    Then $G[C_1]$ has at least $\binom{|C_1|}{2}-\delta_0n^2$ edges, so $G[C_1]$ has edge density at least $1-\delta_0n^2/\binom{|C_1|}{2}\ge 1-\delta_0n^2/(|C_1|^2/4)\ge 1-\delta_0n^2/(p^2n^2/16)=1-\delta$, as desired.
\end{proof}

Next, we present the MaxCut version of the previous lemma, whose proof is almost identical. We only highlight the key differences.

\begin{theorem}\label{thm:dens2_main_MaxCut}
Let $\gamma\in (0,1/30)$, $\delta>0$, then the following holds for every sufficiently large $n$. Let $G$ be an $n$-vertex graph such that $\surp(G)\leq n^{1+\gamma}$. Then $G$ is $\delta$-close to a disjoint union of cliques.
\end{theorem}

\begin{proof}
    Note that \cref{claim:groth} implies $\surp^*(G) = O(\surp(G)\log n) \le n^{1+\gamma'}$  for some constant $\gamma' \in (\gamma,1/30)$.
    One of the key differences compared to the proof of \Cref{thm:MAIN_SE1} is that we use Lemma \ref{lemma:main_MaxCut} to have $4nS_{\frac{T^2}{2n}}\geq S_T^2$ satisfied for every $T\geq n^{1-\frac{1}{24}+\frac{\gamma'}{4}}$. 
    Then, setting $q=1/30$, we know that $\gamma' < q$ and $1-q >1-\frac{1}{24}+\frac{\gamma'}{4}$.
    So we can apply \Cref{lemma:recursion}. Another difference is that we bound $\sum_{0<\lambda_i}\lambda_i^2$ using \Cref{lemma:surp_star} (ii), which gives
    $\sum_{0<\lambda_i}\lambda_i^2\leq O(\sqrt{n}\surp^*(G))\leq O(n^{3/2+\gamma'})=o(n^2)$. 
\end{proof}

Finally, we deduce the immediate corollary of this theorem about finding dense subgraphs. The proof of this is identical to the proof of Corollary \ref{cor:dens2_main}, so we omit it.

\begin{corollary}\label{cor:dens2_main_MaxCut}
Let $\gamma\in (0,1/30)$, $p>0$ and $\delta>0$, then the following holds for every sufficiently large $n$. Let $G$ be an $n$-vertex graph of edge density $p$ such that $\surp(G)\leq n^{1+\gamma}$. Then $G$ contains a subgraph on at least $pn/2$ vertices of edge density at least $1-\delta$.
\end{corollary}

\section{Densification --- Phase 3}\label{sect:dens3}

In this section, we show that every $n$-vertex graph of density at least $1-10^{-6}$ and $|\lambda_n|= \widetilde{O}(n^{1/4})$ contains a large induced subgraph with density $1-\widetilde{O}(|\lambda_n|^2/n)$. Here and later, the $\widetilde{O}(.)$ and $\widetilde{\Omega}(.)$ notations hide factors that grow at most poly-logarithmically in $n$. 
Furthermore, we present the analogous result for the surplus. In particular, we prove the following theorem.

\begin{theorem} \label{theorem:dens3_main}
For sufficiently large $n$, any $n$-vertex graph $G$ of density at least $1-10^{-6}$ satisfies the following.
    \begin{itemize}
        \item[(a)] If the smallest eigenvalue $\lambda_n$ satisfies $|\lambda_n| \le n^{1/4}/(\log n)^4$, then $G$ contains an induced subgraph on $\Omega(n/\log n)$ vertices whose density is at least $1-\widetilde{O}\big(|\lambda_n|^2/n\big)$.
        \item[(b)] If $\surp(G) \leq n^{6/5}/(\log n)^{6}$, then $G$ contains an induced subgraph on $\Omega(n/\log n)$ vertices whose density is at least $1- \widetilde{O}\big(\surp(G)^2/n^3\big)$.
    \end{itemize}
\end{theorem}

We note that both $1-\widetilde{O}(|\lambda_n|^2/n)$ and $1-\widetilde{O}(\surp(G)^2/n^3)$ are tight up to a poly-logarithmic factor.
Indeed, with high probability, the Erd\H{o}s--R\'enyi graph $G(n,1-p)$ has smallest eigenvalue $-\Theta((pn)^{1/2})$ and surplus $\Theta(n(pn)^{1/2})$ while no induced subgraph has density much larger than $1-p$.
Hence, \Cref{theorem:dens3_main} shows that under moderate conditions on the smallest eigenvalue or the surplus, any dense graph must contain a large induced subgraph of density \textit{very} close to $1$.

We then present the intuition behind this result, in particular focusing on \textit{(a)}. 
To simplify the setup, assume that $G$ is $(n-1-d)$-regular with $d\leq n/10^6$ and has smallest eigenvalue $\lambda_n$. 
Then its complement $\overline{G}$ is $d$-regular with second eigenvalue $|\lambda_n|-1$. 
As discussed in \Cref{sec: intro alon--boppana}, the work of Balla \cite{Balla21}, and R\"aty, Sudakov, Tomon \cite{BRST,RST} asserts that the second eigenvalue of a $d$-regular graph is at least 
\[
    |\lambda_n(G)|-1
    =\lambda_2(\overline{G})
    =\Omega\Big(\max\big\{d^{1/3},\min\{d^{1/2},n/d\}\big\}\Big)
    =   \begin{cases}\Omega(d^{1/2}) &\mbox{ if } d\leq n^{2/3},\\
            \Omega(n/d) &\mbox{ if } d\in [n^{2/3},n^{3/4}],\\
            \Omega(d^{1/3}) &\mbox{ if } d\in [n^{3/4},(1/2-\eps)n].
        \end{cases}
\]
Interestingly, this lower bound $f(d)=\Omega(\max\{d^{1/3},\min\{d^{1/2},n/d\}\})$ is not monotone in $d$. 
However, if $$|\lambda_n(G)|\ll n^{1/4}\approx \min_{n^{2/3}\leq d\leq (1/2-\eps)n} f(d),$$ then we must have $d\leq n^{2/3}$, showing that $|\lambda_n(G)|\geq \Omega(d^{1/2})$, or equivalently, $d\leq O(|\lambda_n|^2)$. 
Hence, $G$ has density $1-\frac{d}{n}\geq 1-O(\frac{|\lambda_n|^2}{n})$, as required. 

Unfortunately, there are several difficulties to deal with graphs that are not regular, which requires significant new ideas.
First, we extend the main results of \cite{RST} on eigenvalues and surplus to graphs that are somewhat regular; this is presented in \Cref{subsec: hadamard bounds,subsec: unified bounds via eigenvalues}. To do this, we employ a novel trick which uses triple Hadamard products.
% Then, we show in \Cref{subsec: regularisation} that any dense $G$ contains a large induced subgraph of $G$ whose complement is somewhat regular.
This allows us to complete the proof in \Cref{subsec:concluding dens3}.
% We prepare the proof of Theorem \ref{theorem: both bounds for balanced graphs} in three steps.

\subsection{Finding balanced subgraphs} \label{subsec: regularisation}

We begin by passing to an induced subgraph of $G$ whose complement is somewhat regular. 
We say a graph $G$ is \emph{$C$-balanced} if $\Delta(G)\leq Cd(G)$.

We note that the problem of finding large $C$-balanced (or $C$-almost-regular) induced subgraphs in general graphs was considered by Alon, Krivelevich and Sudakov in \cite{AKS08}. Lemma~\ref{lemma:dens4_finding_balanced} is similar in spirit to their results, but it controls the density of the resulting graph explicitly, which will be useful later.

\begin{lemma}\label{lemma:dens4_finding_balanced}
    Let $G$ be an $n$-vertex graph of edge density $p\leq 1/5$. Then, $G$ has an induced subgraph $G'$ on $\Omega(n/\log n)$ vertices such that the density of $G'$ is $p'\leq p$, and $G'$ is $C$-balanced with $C=4\log_2 1/p'$.
\end{lemma}

\begin{proof}
    We may assume that $n$ is sufficiently large as otherwise, we can simply take $G'$ to be a single vertex. Let $G_0=G$ and define the sequence of induced subgraphs $G_0\supset G_1\supset...$ as follows. If the graph $G_i$ has $n_i$ vertices and density $p_i$, and it contains an induced subgraph $G_{i+1}$ on at least $\big(1-\frac{1}{\log_2 1/p_i}\big)n_i$ vertices and density $p_{i+1}=p(G_{i+1})< p_i/2$, then pick this induced subgraph to be the next element of the sequence (if there are several such subgraphs $G_{i+1}$, choose one arbitrarily). 
    If there is no such subgraph, terminate the procedure.
    
    Suppose that subgraphs $G_0\supset G_1\supset \dots \supset G_k$ have been defined in this way. Then for each $i\leq k$, we have $p_i\leq p/2^i$, and therefore the process terminates in at most $k\leq 2\log_2 n$ steps (since we cannot have $p_k\leq p/2^k\leq 1/n^2$ unless $p_k=0$, in which case the process terminates). In addition, $\log_2 1/p_i\ge i+\log_2 1/p$, so $n_{i+1} \ge \big(1-\frac{1}{i+\log_2 1/p}\big) n_i$, which leads to the conclusion 
    \[ 
        n_k\geq n\prod_{i=0}^{k-1} \left(1-\frac{1}{i+\log_2 1/p}\right)  
        = \frac{\log_2 1/p-1}{k-1+\log_2 1/p} \cdot n \ge \frac{n}{k+1} = \Omega\left(\frac{n}{\log n}\right).
    \]
    
    Finally, to define $G'$, remove from $G_k$ all vertices of degree at least $(n_k-1)p_k\log_2 1/p_k$; there are at most $\frac{n_k}{\log_2 1/p_k}$ such vertices.
    Hence, the number of vertices in $G'$ is $n' \ge \big(1-\frac{1}{\log_2 1/p_k}\big)n_k$ and the density of $G'$ is $p' \le p_k$.
    Using that $p_k \le p \le 1/5$, we know that $n' \ge (n_k+1)/2$.
    Moreover, since the process terminated at $G_k$, no induced subgraph of $G_k$ on at least $\big(1-\frac{1}{\log_2 1/p_k}\big)n_k$ vertices can have density less than $p_k/2$, so $p_k/2\le p' \le p_k$.
    Since the function $x\log_2\frac{1}{x}$ is increasing in $(0,1/e)$, we have $\frac{1}{2}p_k\log_2\frac{2}{p_k}\le p'\log_2\frac{1}{p'}$, thereby $p_k\log_2\frac{1}{p_k}\le 2p'\log_2\frac{1}{p'}$.
    Then, the maximum degree of $G'$ satisfies $$
        \Delta(G') \le (n_k-1)p_k\log_2 (1/p_k)
        \le 2(n'-1)\cdot 2p'\log_2 (1/p')
        = d(G')\cdot 4\log_2(1/p').$$
    This completes the proof as $n' \ge (n_k+1)/2=\Omega(n/\log n)$.
\end{proof}

We also use the fact that the density of a $C$-balanced graph is robust under deleting few vertices.

\begin{claim}\label{lemma: balanced graphs are robust}
    Let $G$ be an $n$-vertex graph with density $p$ that is $C$-balanced for some $C \ge 1$.
    Then every induced subgraph on at least $(1-1/4C)n$ vertices has density at least $p/2$.
\end{claim}
\begin{proof}
    Let $U\subseteq V(G)$ be any subset of size at least $(1-1/4C)n$.
    Since $G$ is $C$-balanced, the number of edges with an endpoint outside $U$ is at most $\frac{n}{4C}\cdot\Delta(G) \le \frac{n}{4C}\cdot Cp(n-1)=n(n-1)p/4$.
    So $e(G[U]) \ge n(n-1)p/4$ and the density of $G[U]$ is at least $\frac{n(n-1)p/4}{|U|(|U|-1)/2} \ge p/2$.
\end{proof}

\subsection{The smallest eigenvalue and surplus of dense balanced graphs}\label{subsec: hadamard bounds}

In this section, we prove that balanced graphs with a sufficiently high density $1-p$ must satisfy $|\lambda_n|=\Omega((pn)^{1/3})$ and $\surp^*(G)=\Omega((pn)^{1/4})$.
Both bounds are effective as long as $p$ is not too small.
Then, in the next subsection, we show complementary bounds for small $p$.

First, we consider the smallest eigenvalue.

\begin{lemma}\label{lemma: degree bounds for balanced graphs}
   Let $G$ be an $n$-vertex graph with edge density $1-p$, whose complement $\overline{G}$ is $C$-balanced. If $C^2p \le 1/100$, then 
     $$|\lambda_n| = \Omega\big((pn)^{1/3}\big).$$
\end{lemma}
\begin{proof}
    We may assume $p \ge 1/n$ as otherwise $|\lambda_n| \ge 1 \ge (pn)^{1/3}$. Let $A$ be the adjacency matrix of $G$ with eigenvalues $\lambda_1\geq \dots\geq \lambda_n$ and corresponding orthonormal basis of eigenvectors $v_1,\dots,v_n$. 
    Set $B=A-\lambda_1v_1v_1^T$. 
    The key idea of the proof is to consider the following triple Hadamard product:
    $$D= \big(B+|\lambda_n|I \big)^{\circ 3}=B^{\circ 3}+3|\lambda_n|B\circ B\circ I+3|\lambda_n|^2B\circ I+|\lambda_n|^3 I.$$
    As $B+|\lambda_n|I=|\lambda_n|v_1v_1^T+\sum_{i=2}^n (\lambda_i+|\lambda_n|)v_iv_i^T$, we have that $B+|\lambda_n|I$ is positive semidefinite. 
    Therefore, $D$ is also positive semidefinite by the Schur product theorem (\Cref{theorem: schur}). 

    Write $\bar{\Delta}$ for the maximum degree $\overline{G}$; then $\bar{\Delta} \le Cpn$.
    By \Cref{lemma:max_entry}, for every vertex $i\in V(G)$, 
    \begin{equation} \label{eq: scale of v1}
        0\le \frac{1-3Cp}{\sqrt{n}} \le \frac{1-3\bar{\Delta}/n}{\sqrt{n}} \le v_1(i) \le \frac{1+2p+2/n}{\sqrt{n}} \le \frac{1+4p}{\sqrt{n}}
    \end{equation}
    We now evaluate the terms of $\mathds{1}^T D \mathds{1}$, where $\mathds{1}$ is the all-ones vector.
    
    For the main term $\mathds{1}^T B^{\circ 3} \mathds{1}$, note that $$
    B_{i,j}=\begin{cases}
        1-\lambda_1v_1(i)v_1(j) &\mbox{if }ij\in E(G),\\
        -\lambda_1v_1(i)v_1(j) &\mbox{if }ij\notin E(G).
    \end{cases}$$
    Using that $\lambda_1 \ge d(G)= (1-p)(n-1), C \ge 1,p \ge 1/n$ and $Cp \le C^2p \le 1/100$, we have for every $i,j\in V(G)$ that  $$
        1-\lambda_1v_1(i)v_1(j)
        \le 1 - (1-p)(n-1) \cdot \frac{(1-3Cp)^2}{n}
        \le p + 6Cp + \frac{1}{n}
        \le 8Cp
        \le 1/2.
    $$
    This further shows $\lambda_1v_1(i)v_1(j) \ge \frac{1}{2}$.
    Hence, using that $C^3p^2\le (C^2p)^2 \le 10^{-4}$, we get
    \begin{equation}    \nonumber
      \begin{aligned}
        \mathds{1}^T B^{\circ 3}\mathds{1}
        &= \sum_{i \sim j} (1-\lambda_1v_1(i)v_1(j))^3 - \sum_{i\not\sim j}(\lambda_1v_1(i)v_1(j))^3 \\
        &\le n^2\cdot (8Cp)^3 - \big(pn(n-1)+n\big)\cdot\frac{1}{8}
        \le 512C^3p^3n^2 - \frac{pn^2}{8}
        \le -\frac{pn^2}{16}.
      \end{aligned}
    \end{equation}
    For the other terms, we observe that $B_{i,i}=-\lambda_1v_1(i)^2 \in [-2,0]$ using (\ref{eq: scale of v1}).
    Hence, 
    \begin{equation} \nonumber
        \mathds{1}^T \big(3|\lambda_n|B\circ B\circ I+3|\lambda_n|^2B\circ I+|\lambda_n|^3I \big) \mathds{1}  = O\big(n|\lambda_n|^3\big).        
    \end{equation}
    In conclusion, we showed that $$
        0 \le \mathds{1}^T D \mathds{1}
        \le -\frac{pn^2}{16} + O\big(n|\lambda_n|^3\big).
    $$
    This gives $|\lambda_n| = \Omega\big((pn)^{1/3}\big)$, as desired.
\end{proof}

Now we consider the surplus.

\begin{lemma}
   Let $G$ be an $n$-vertex graph with edge density $1-p$, whose complement $\overline{G}$ is $C$-balanced. If $C^2p \le 1/100$, then 
        $$\surp^*(G) = \Omega\big(C^{-3/4}n(pn)^{1/4} \big).$$
\end{lemma}
\begin{proof}
    Set $S:= \surp^*(G)$ and assume for contradiction that $S= o\big(C^{-3/4}n(pn)^{1/4}\big)$.
    Since $p \le C^2p \le 1/100$, there are $\Omega(n)$ non-isolated vertices in $G$.
    Then, a classical result of Erd\H{o}s, Gy\'arf\'as, and Kohayakawa \cite{EGyK} asserts $\surp(G) = \Omega(n)$, so $\surp^*(G) \ge \surp(G) = \Omega(n)$.
    Hence, $S= o\big(C^{-3/4}n(pn)^{1/4}\big)$ implies $pn = \omega(1)$ and $C = o\big((pn)^{1/3})\big)$.
    
    Let $A$ be the adjacency matrix of $G$ with eigenvalues $\lambda_1\geq \dots\geq \lambda_n$ and corresponding orthonormal basis of eigenvectors $v_1,\dots,v_n$.
    Define $B=A-\lambda_1v_1v_1^T$ and $E=\sum_{\lambda_i<0}|\lambda_i| v_iv_i^T$. Then, the matrices $E$ and $B+E=\sum_{\lambda_i>0,i\neq 1}\lambda_i v_iv_i^T$ are positive semidefinite. Consider 
    $$D=(B+E)^{\circ 3}=B^{\circ 3}+3B\circ B\circ E+3B\circ E\circ E+E^{\circ 3}.$$
    As $B+E$ is positive semidefinite, so is $D$ by the Schur product theorem (\Cref{theorem: schur}). 
    Next, we identify a set of well-behaved vertices $U$, and carefully evaluate the product 
    \begin{equation} \label{eq: expansion of cut--hadamard}
        0 \leq \mathds{1}_U^T \,D\, \mathds{1}_U
        = \mathds{1}_U^T\,B^{\circ 3}\,\mathds{1}_U+ 3\cdot \mathds{1}_U^T\big(B\circ B\circ E\big)\mathds{1}_U + 3\cdot \mathds{1}_U^T\big(B\circ E\circ E\big)\mathds{1}_U + \mathds{1}_U^T \,E^{\circ 3}\, \mathds{1}_U.
    \end{equation}

    Let $\bar{d}, \bar{\Delta}$ be the average degree and the maximum degree of $\overline{G}$, respectively, so we have $\bar{\Delta} \le C\bar{d}$.
    Let $U$ be the set of vertices $i \in [n]$ such that $E_{i,i} \le \frac{4CS}{n}$.
    Recall from \Cref{lemma:surp_star} (i) that $$ \trace(E)= \sum_{\lambda_i < 0} |\lambda_i| \le \surp^*(G)=S.$$
    So at most $n/4C$ vertices $i \in [n]$ have $E_{i,i} > \frac{4CS}{n}$. This means $|U| \ge (1-\frac{1}{4C})n$.
    By \Cref{lemma: balanced graphs are robust}, the density of $\overline{G}[U]$ is at least $p/2$.
    Moreover, \Cref{lemma:surp_star} (ii) and our assumption that $S=o\big(n\bar{d}^{\,1/4}C^{-3/4}\big)$ imply 
    \begin{equation} \label{eq: maxcut F-norm upper bound}
        \|E\|_F^2 = \sum_{\lambda_i < 0} \lambda_i^2
        = O\Big({\bar{\Delta}}^{1/2}\surp^*(G)\Big) 
        = O\big(C^{1/2}\bar{d}^{\,1/2} S\big)
        = o\big(n\bar{d}^{\,3/4}C^{-1/4}\big).
    \end{equation}
    
    In order to bound (\ref{eq: expansion of cut--hadamard}), we start with the main term $\mathds{1}_U^T \,B^{\circ 3}\, \mathds{1}_U$.
    Repeating the same analysis as in the proof of \Cref{lemma: degree bounds for balanced graphs}, for all $i,j \in [n]$, we have 
    $$\frac{1-3Cp}{\sqrt{n}} \le v_1(i) \le \frac{1+4p}{\sqrt{n}}\;\;\text{ and }\;\;
        1 - \lambda_1v_1(i)v_1(j) \le 8Cp.$$
    Using that $p \le C^2p \le 1/100$, for every  $ij \in E(G)$, we have 
    $$ -1 \le -9p \le 1-(1+4p)^2\le B_{i,j} = 1-\lambda_1v_1(i)v_1(j) \le 8Cp \le 1,$$
    while for any $ij \not\in E(G)$, using that $p \le Cp \le C^2p \le 1/100$, we have 
    $$-2 \le -(1+4p)^2\le B_{i,j} = -\lambda_1v_1(i)v_1(j) \le 8Cp - 1 \le -\frac{1}{2}.$$
    Since $C^3p^2 \le (C^2p)^2 \le 10^{-4}$, 
    \begin{equation}    \nonumber
      \begin{aligned}
        \mathds{1}^T_U B^{\circ 3}\mathds{1}_U
        &= \sum_{i,j\in U,i \sim j} (1-\lambda_1v_1(i)v_1(j))^3 - \sum_{i,j\in U,i\not\sim j}(\lambda_1v_1(i)v_1(j))^3 \\
        &\le |U|^2 \cdot (8Cp)^3 - \frac{(p/2)|U|(|U|-1)+|U|}{8}
        \le 216C^3p^3|U|^2-\frac{p}{16}|U|^2
        = -\Omega\big(p|U|^2\big)
      \end{aligned}
    \end{equation}
    In addition, $|U| \ge (1-\frac{1}{4C})n =\Omega(n)$, so $\mathds{1}^T_U B^{\circ 3}\mathds{1}_U= -\Omega(pn^2)=-\Omega(n\bar{d})$.
    
    We complete the proof by the following claim bounding the rest of the terms of (\ref{eq: expansion of cut--hadamard}).
    \begin{claim}
    \begin{itemize}
        \item[(1)] $\mathds{1}_U^T\big(B\circ B\circ E\big)\mathds{1}_U=o(n\bar{d}),$
        \item[(2)] $\mathds{1}_U^T\big(B\circ E\circ E\big)\mathds{1}_U=o(n\bar{d}),$
        \item[(3)] $\mathds{1}_U^T\,E^{\circ 3}\,\mathds{1}_U=o(n\bar{d}).$
        \end{itemize}
    \end{claim}
    
    \begin{proof}
    
        \begin{itemize}

        \item[(1)] Write $X$ for the set of pairs $(i,j) \in U^2$ where $ij \in E(G)$ and $Y=(U\times U)\setminus X$. As discussed above, $|B_{i,j}| = O(Cp)$ for all $(i,j) \in X$.
        So, by the Cauchy--Schwartz inequality, 
        \begin{equation} \nonumber
          \begin{aligned}
            \sum_{(i,j) \in X} B_{i,j}^2E_{i,j}
            = O(C^2p^2) \sum_{(i,j) \in X}|E_{i,j}|
            = O(C^2p^2) \cdot n \Big(\sum_{(i,j) \in X}|E_{i,j}|^2\Big)^{1/2} 
            = O\big(C^2p^2n \|E\|_F \big).
          \end{aligned}
        \end{equation}
        Using that  $C^2p^2\le (C^2p)^{3/2}p^{1/2}\le p^{1/2}$, we see that the above is at most $O\big(p^{1/2}n \|E\|_F \big)$.
        Moreover, $|Y| \le n^2p+n \le 2n^2p$ and $|B_{i,j}| \le 2$ for all $(i,j) \in Y$.
        Again, by the Cauchy--Schwarz inequality, 
        \begin{equation} \nonumber
          \begin{aligned}
            \sum_{(i,j) \in Y} B_{i,j}^2E_{i,j}
            \le 4 \sum_{(i,j) \in Y}|E_{i,j}|
            \le 4 \cdot \Big(|Y|\sum_{(i,j) \in Y}|E_{i,j}|^2\Big)^{1/2} 
            = O\big(p^{1/2}n \|E\|_F\big).
          \end{aligned}
        \end{equation}
        Altogether, we have 
        \begin{equation} \nonumber
          \begin{aligned}
            \mathds{1}_U^T\big(B\circ B\circ E\big)\mathds{1}_U
            = \sum_{(i,j) \in X} B_{i,j}^2E_{i,j} + \sum_{(i,j) \in Y} B_{i,j}^2E_{i,j}
            = O\big(p^{1/2}n\|E\|_F\big)
            = O\big((n\bar{d})^{1/2} \|E\|_F\big).
          \end{aligned}
        \end{equation}
        By (\ref{eq: maxcut F-norm upper bound}), $\|E\|_F^2 = o(n\bar{d}^{\,3/4}C^{-1/4})=o(n\bar{d}^{\,3/4})$, so $$
            \mathds{1}_U^T\big(B\circ B\circ E\big)\mathds{1}_U
            = o\big( (n\bar{d})^{1/2} \cdot n^{1/2}\bar{d}^{\,3/8})\big)
            = o(n\bar{d}).
        $$
        
        \item[(2)] As $|B_{i,j}| \le 2$ holds for all $i,j \in U$ and $\|E\|_F^2 = o\big(n\bar{d}^{\,3/4}\big)$, we have
        \begin{equation} \nonumber
            \mathds{1}_U^T\big(B\circ E\circ E\big)\mathds{1}_U 
            = \sum_{i,j \in U} B_{i,j} E_{i,j}^2
            \le 2 \|E\|_F^2 
            = o\big(n\bar{d}^{\,3/4} \big).
        \end{equation}
        
        \item[(3)]  Recall that $E_{i,i} \le 4CS/n=o\big((C\bar{d})^{1/4}\big)$ for all $i \in U$. As $E$ is positive semidefinite, this implies that $|E_{i,j}|\leq 4CS/n =o\big((C\bar{d})^{1/4}\big)$ for all $i,j \in U$.
        This, together with (\ref{eq: maxcut F-norm upper bound}), completes the proof as 
        \[
            \mathds{1}_U^T\,E^{\circ 3}\,\mathds{1}_U
            = \sum_{i,j \in U} E_{i,j}^3
            = \max_{i,j\in U} |E_{i,j}|\cdot \sum_{i,j \in U} E_{i,j}^2
            = o\big((C\bar{d})^{1/4}\big)\cdot \|E\|_F^2
            = o(n\bar{d}).
            \hfill\qedhere
        \]
        \end{itemize}
    \end{proof}

     With these estimates, we can rewrite  (\ref{eq: expansion of cut--hadamard}) as $0 \le \mathds{1}^T_U\,D\,\mathds{1}_U \le -\Omega(n\bar{d})+o(n\bar{d})$; a contradiction.
\end{proof}

\subsection{Further bounds on balanced graphs}\label{subsec: unified bounds via eigenvalues}

In this subsection, we prove that if $G$ is a balanced graph of sufficiently large density $1-p$, then  $|\lambda_n|\geq \frac{\surp^*(G)}{n}=\Omega(\min\{p^{-1},(pn)^{1/2}\})$, where the first inequality holds by \Cref{claim: surp* lambdan}. Note that this beats the lower bound for $|\lambda_n|$ from the previous section when $p\ll n^{-1/4}$, and beats the lower bound for the surplus when $p\ll n^{-1/5}$.  

\begin{lemma}\label{lemma:dens4_main_balanced}
    Let $G$ be an $n$-vertex graph with edge density $1-p$, whose complement $\overline{G}$ is $C$-balanced. 
    If $C^2p \le 1/100$, then 
    $$\surp^*(G)= \Omega\Big(C^{-3}\cdot \min\{np^{-1},n(pn)^{1/2}\}\Big).$$
    In particular,
    $$|\lambda_n|=\Omega\Big(C^{-3}\cdot\min\{p^{-1},(pn)^{1/2}\}\Big).$$
\end{lemma}

To establish these lower bounds, it is more convenient to work with the complement of $G$. 
Unfortunately, as $G$ is not necessarily regular, there is no simple formula to express the eigenvalues of $\overline{G}$ in terms of those of $G$. 
However, we can use Weyl's inequality to establish the following inequality.

\begin{lemma}\label{lemma:dens4_weyl}
    Let $G$ be an $n$ vertex graph with eigenvalues $\lambda_1\geq \dots\geq \lambda_n$, and let $\mu_1\geq \dots\geq \mu_n$ be the eigenvalues of the complement of $G$. 
    For each $i=1,2\dots,n-1$, we have $$1+\mu_{i+1}\leq -\lambda_{n+1-i}.$$
\end{lemma}
\begin{proof}
    Weyl's inequality states that if $X$ and $Y$ are $n\times n$ symmetric matrices, and $1\leq i,j\leq n$ and $i+j\leq n+1$, then $$\lambda_{i+j-1}(X+Y)\leq \lambda_i(X)+\lambda_j(Y),$$
    where $\lambda_1(X)\geq \dots\geq \lambda_n(X)$ denote the eigenvalues of a matrix $X$. 
    Let $A$ be the adjacency matrix of $G$, and then $-A+J-I$ is the adjacency matrix of $\overline{G}$.
    Set $X=-A$ and $Y=J-I$.
    We have $\lambda_i(X)=-\lambda_{n+1-i}$, $\lambda_1(Y)=n-1$, $\lambda_i(Y)=-1$ for $i=2,\dots,n$, and $\lambda_i(X+Y)=\mu_i$. 
    Hence, applying the above inequality with $j=2$, we get for $i=1,2,\dots,n-1$, 
    \[\mu_{i+1}=\lambda_{i+1}(X+Y) \le \lambda_i(X) + \lambda_2(Y) = -\lambda_{n+1-i}-1. \hfill\qedhere\]
\end{proof}

\begin{proof}[Proof of \Cref{lemma:dens4_main_balanced}]
    We focus on the lower bounds for $\surp^*(G)$ since that of $\lambda_n$ follows from \Cref{claim: surp* lambdan}, i.e. $|\lambda_n|n \ge \surp^*(G)$.
    
    Let $A$ be the adjacency matrix of $G$ with eigenvalues $\lambda_1\geq \dots\geq \lambda_n$, and let $B$ be the adjacency matrix of $\overline{G}$ with eigenvalues $\mu_1\geq\dots\geq \mu_n$.
    Write $\Bar{\Delta}$ for the maximum degree of $\overline{G}$, so $\mu_1\leq \Bar{\Delta}\leq Cpn$. We may assume that $p>0$ and thus $\Bar{\Delta}\geq 1$, as otherwise the statement is trivial. For $k=1,2,3$, set $$P_k=\sum_{i\neq 1, \mu_i>0}\mu_i^k\mbox{\ \ \ and\ \ \ } N_k=\sum_{\mu_i<0}|\mu_i|^k.$$ 
    By \Cref{lemma:dens4_weyl}, whenever $\mu_{i+1}\geq 0$, we also have $\lambda_{n+1-i}\leq -\mu_{i+1}-1<0$. 
    Combined with \Cref{lemma:surp_star}, this shows that
    \begin{align} 
        \surp^*(G) 
            &\geq \sum_{\lambda_i<0}|\lambda_i| 
            \geq \sum_{i\neq 1, \mu_i>0}\mu_i = P_1, \label{eq: first moment complement}\\
        \surp^*(G) 
            &= \Omega\bigg(\Bar{\Delta}^{-1/2}\sum_{\lambda_i<0}|\lambda_i|^2\bigg)
            =\Omega\bigg({\Bar{\Delta}^{-1/2}}\sum_{i\neq 1, \mu_i>0}\mu_i^2\bigg) 
            = \Omega\Big(\Bar{\Delta}^{-1/2}P_2\Big), \label{eq: second moment complement}\\
        \surp^*(G) &
            = \Omega\bigg(\Bar{\Delta}^{-1}\sum_{\lambda_i<0}|\lambda_i|^3\bigg)
            = \Omega\bigg(\Bar{\Delta}^{-1}\sum_{i \neq 1, \mu_i>0}\mu_i^3\bigg)
            = \Omega\Big(\Bar{\Delta}^{-1}P_3\Big). \label{eq: third moment complement}
    \end{align}
    We show that these three inequalities together with some simple identities suffice to prove the lemma. 
    
    First, assume that $N_2\leq \frac{1}{8}pn^2$. 
    Note that $\mu_1^2+P_2+N_2=\|B\|_F^2$ is twice the number of edges of $\overline{G}$, so $\mu_1^2+P_2+N_2=2p\binom{n}{2}\ge pn^2/2$.
    Hence, using that $C^2p\le 1/100$, we get $$P_2\geq pn^2/2-\mu_1^2-N_2\geq pn^2/2-C^2p^2n^2-pn^2/8\geq pn^2/4.$$
    But then, (\ref{eq: second moment complement}) implies $\surp^*(G)=\Omega(C^{-1/2}p^{1/2}n^{3/2})$, and we are done.
    
    In the rest of the proof, we may assume $N_2\geq \frac{1}{8}pn^2$. 
    By the inequality between the quadratic and cubic mean, we have
    $$\left(\frac{N_2}{n}\right)^{1/2}\leq \left(\frac{N_3}{n}\right)^{1/3}$$
    which gives $N_3\geq N_2^{3/2}n^{-1/2}\geq p^{3/2}n^{5/2}/64.$
    
    Next, consider the quantity $T=N_3-P_3$. 
    Observe that $\mu_1^3-T=\sum_{i=1}^n\mu_i^3$ is six-times the number of triangles of $\overline{G}$. 
    In particular, $\mu_1^3-T$ it is nonnegative, showing that $T\leq \mu_1^3\leq \Bar{\Delta}^3$.
    If $N_3\geq 2T$, or equivalently, $P_3\geq N_3/2$. 
    We are done as (\ref{eq: third moment complement}) implies $$\surp^*(G)
        = \Omega\Big(\bar{\Delta}^{-1}{P_3}\Big)
        = \Omega\Big(\bar{\Delta}^{-1}{N_3}\Big)
        = \Omega\Big(C^{-1}p^{1/2}n^{3/2}\Big).$$
    
    Finally, if $N_3\leq 2T$, then $\Bar{\Delta}^3\geq T\geq N_3/2$.
    By the Cauchy--Schwartz inequality applied to the sequences $(|\mu_i|^3)_{\mu_i<0}$ and $(|\mu_i|)_{\mu_i<0}$, we have $N_1N_3\geq N_2^2$, which gives 
    $$N_1\geq \frac{N_2^2}{N_3}\geq \frac{(pn^2/8)^2}{2\Bar{\Delta}^3}\geq \frac{n}{128C^3p}.$$
    But $0=\trace(B)=\mu_1+P_1-N_1$, from which $$P_1=N_1-\mu_1\geq \frac{n}{128C^3p}-\Bar{\Delta}\ge \frac{n}{128C^3p} - Cpn \ge \frac{n}{256C^3p}.$$
    Here, we also used that $pC^2 \le 1/100$.
    This completes the proof by (\ref{eq: first moment complement}), i.e. $\surp^*(G)\geq P_1$.
\end{proof}

\subsection{Completing the proof}\label{subsec:concluding dens3}

We can summarize the results of \Cref{subsec: hadamard bounds,subsec: unified bounds via eigenvalues} in the following theorem.

\begin{theorem} \label{theorem: both bounds for balanced graphs}
    Let $G$ be an $n$-vertex graph with edge density $1-p$, whose complement $\overline{G}$ is $C$-balanced of average degree $d=d(\overline{G})$. If $C^2p \le 1/100$, then 
    \begin{enumerate}
        \item[(a)] the smallest eigenvalue $\lambda_n$ satisfies $|\lambda_n| = \Omega\Big(C^{-3}\cdot \max\big\{\min\{n/d, d^{1/2}\}, d^{1/3}\big\}\Big),$
        \item[(b)] and the surplus satisfies $\surp^*(G) = \Omega\Big(C^{-3}n\cdot \max\big\{\min\{n/d, d^{1/2}\}, d^{1/4} \big\}\Big).$
    \end{enumerate}
\end{theorem}

Combining \Cref{lemma:dens4_finding_balanced} and \Cref{theorem: both bounds for balanced graphs}, we now prove the main theorem of this section. 

\begin{proof}[Proof of Theorem \ref{theorem:dens3_main}]
(a) 
    Applying \Cref{lemma:dens4_finding_balanced} to $\overline{G}$, we get that $G$ contains an induced subgraph $G_1$ on $n_1 = \Omega(n/\log n)$ vertices of density $1-p_1\geq 1-10^{-6}$, and whose complement $\overline{G_1}$ is $C_1$-balanced for $C_1=4\log_2 1/p_1$.
    We are done if $G_1$ is a clique.
    Otherwise, $p_1\geq 1/\binom{n_1}{2}$ implying that $C_1 = O(\log n_1)$. 
    We claim that $G_1$ is the desired subgraph, i.e. $p_1 = \widetilde{O}\big(|\lambda_n|^2/n\big)$, by applying \Cref{theorem: both bounds for balanced graphs} (note that this is valid as as $C_1^2p_1\le 16p_1(\log_2 1/p_1)^2 \le 1/100$ for all $p_1\in (0, 10^{-6})$).

    Let $\lambda_{\min}$ be the smallest eigenvalue of $G_1$. 
    Cauchy's interlacing theorem guarantees $|\lambda_{\min}| \le |\lambda_n|\leq n^{1/4}/(\log n)^4$. 
    Let $d_1 = d(\overline{G_1})$.
    Suppose $d_1 \ge n_1^{2/3}$.
    It is easy to check that $\max\big\{\min\{n_1/d_1, d_1^{1/2}\}, d_1^{1/3}\big\} \ge n_1^{1/4}$.
    Indeed, $d_1^{1/3} \ge n_1^{1/4}$ if $d_1 \ge n_1^{3/4}$ and $\min\{n_1/d_1, d_1^{1/2}\} \ge n_1^{1/4}$ if $n_1^{2/3} \le d_1 \le n_1^{3/4}$.
    So \Cref{theorem: both bounds for balanced graphs} implies $|\lambda_{\min}| = \Omega\big(C_1^{-3}\cdot n_1^{1/4}\,\big)$.
    Recalling that $n_1 = \Omega(n/\log n), C_1=O(\log n_1)$ and $n$ is sufficiently large, this bound gives $|\lambda_{\min}|=\Omega(n^{1/4}/(\log n)^{13/4}\,\big) > n^{1/4}/(\log n)^4$; this is impossible.

    Now we know that $d_1 \le n_1^{2/3}$.
    In this case, $d_1^{1/3} \le d_1^{1/2} \le n_1/d_1$, so \Cref{theorem: both bounds for balanced graphs} implies $|\lambda_{\min}| = \Omega\big(C_1^{-3}\cdot d_1^{1/2}\,\big)$.
    Recall that $|\lambda_{\min}| \le |\lambda_n|$.
    It must be that $C_1^{-3}\cdot d_1^{1/2} = O\big(|\lambda_n|\big)$, i.e. $d_1 = O\big(|\lambda_n|^2 C_1^{6}\,\big) = O\big(|\lambda_n|^2 (\log n)^6\big)$.
    This completes the proof as $p_1 = d_1/(n_1-1) = O\big(|\lambda_n|^2 (\log n)^7/n\big)$.

\medskip

(b) 
    By \Cref{lemma:dens4_finding_balanced} on $\overline{G}$, $G$ has an induced subgraph $G_1\subset G$ on $n_1=\Omega(n/\log n)$ vertices of density $1-p_1 \ge 1-10^{-6}$ and with a $C_1$-balanced complement, where $C_1=O(\log n)$. 
    By \Cref{claim:groth}, $\surp^*(G_1) \le \surp^*(G) = O\big(\surp(G)\cdot \log n\big)=O(n^{6/5}/(\log n)^5)$.
    Write $d_1 = d(\overline{G_1})$.
    If $d_1 \ge n_1^{2/3}$, one can check that $\max\big\{\min\{n_1/d_1, d_1^{1/2}\}, d_1^{1/4}\big\} \ge n_1^{1/5}$, 
    so \Cref{theorem: both bounds for balanced graphs} yields $\surp^*(G_1)= \Omega\big(C_1^{-3}n_1\cdot n_1^{1/5}\big) = \Omega\big(n_1^{6/5}/C_1^3\big)=\Omega\big(n^{6/5}/(\log n)^{6/5+3}\big)>n^{6/5}/(\log n)^5$; this is impossible.
    So $d_1 \le n_1^{2/3}$ must hold.
    But then, \Cref{theorem: both bounds for balanced graphs} implies $\surp^*(G_1)= \Omega\big(C_1^{-3}n_1\cdot d_1^{1/2}\big)$.
    As $\surp^*(G_1) \le \surp^*(G)=O(\surp(G)\cdot \log n)$, we get $d_1 = O\big(C_1^6\surp(G)^2(\log n)^2/n_1^2\big)$.
    So, $p_1=d_1/(n_1-1)=\widetilde{O}(\surp(G)^2/n^3)$, as desired.
    % Since $\surp^*(G_1) \le \surp^*(G)$, Theorem~\ref{theorem: both bounds for balanced graphs} (b) yields $\surp^*(G) \geq \Omega(C_1^{-3}n_1d_1^{1/4})$, which implies that $$d_1\leq O\big((\surp^*(G)\cdot (\log n)^4/n_1)^4\big)\leq n^{4/5}.$$ By a similar calculation as before, we conclude that $\surp^*(G)\ll C_1^{-3}n^2/d$, implying that the minimum in Theorem~\ref{theorem: both bounds for balanced graphs}(b) must be attained by the term $C_1^{-3} n_1d_1^{1/2}$, giving $\surp^*(G)\geq C_1^{-3} n_1d_1^{1/2}$, from where it follows that $d_1\leq \widetilde{O}(\surp^*(G)^2/n^2)$, completing the proof. 
\end{proof}

\section{Large cliques from small eigenvalues or surplus}\label{sect:cliques}

In this section, we combine our densification steps to prove that graphs with large smallest eigenvalue (or with small surplus) contain large cliques. 
This section contains two main results: one result covers graphs whose density is polynomially close to $n$, while the other covers graphs whose average degree could be much smaller than $n$.

\begin{theorem}\label{thm:clique1}
Let $\gamma,\eps,\rho > 0$ such that $\rho< 1/2$, $\eps+6\gamma<1$ and $\rho/\eps+ 2\gamma/(1-\eps-4\gamma) < 1$. Then for sufficiently large $n$, any $n$-vertex graph $G$ of edge density at least $n^{-\rho}$ and smallest eigenvalue $|\lambda_n|\leq n^{\gamma}$ contains a clique of size at least $n^{1-\eps - 2\gamma}$.
\end{theorem}
\begin{proof}
    % In this proof, we put together the three phases of densification that we have presented thus far.
    % Before we do this, we choose 
    Choose a parameter $\eps'<\eps$ for which the inequalities $$\rho< 1/2,\quad \eps'+6\gamma<1 \text{\ \ and\ \ } \rho/\eps'+ 2\gamma/(1-\eps'-4\gamma) < 1$$ are still satisfied. 
    This can be done since all the inequalities are strict.
    
    As the first step, we apply \Cref{lemma:dens1_main}, which states that if the above inequalities on the parameters $\gamma, \eps'$ and $\rho$ are satisfied and $n$ is sufficiently large, there exists an induced subgraph $G_1\subseteq G$ on at least $n_1\geq n^{1-\eps'}$ vertices with density $\Omega(1)$ (where the hidden constant may depend on $\gamma, \rho$ and $\eps'$). 
    Denote the smallest eigenvalue of $G_1$ by $\lambda^{(1)}$; then Cauchy's interlacing theorem implies $|\lambda^{(1)}|\leq |\lambda_n|\leq n^{\gamma}$.
    
    Next, applying the second phase of densification to $G_1$, that is, \Cref{cor:dens2_main}, we get that $G_1$ contains an induced subgraph $G_2$ on  $n_2= \Omega(n_1)= \Omega(n^{1-\eps'})$ vertices with edge density at least $1-10^{-6}$. 
    The smallest eigenvalue $\lambda^{(2)}$ of $G_2$ still satisfies $|\lambda^{(2)}|\leq |\lambda^{(1)}|\leq n^{\gamma}$.
    
    Finally, we apply the third phase of densification to $G_2$, that is, \Cref{theorem:dens3_main}.
    This theorem states that if $|\lambda^{(2)}|\leq n_2^{1/4}/(\log n)^4$, then $G_2$ contains an induced subgraph $G_3$ on $n_3=\Omega(n_2/\log n_2)$ vertices with edge density $1-p_3$, where $p_3\le |\lambda^{(2)}|^2/n_2^{1-o(1)}$. 
    The condition on $\lambda^{(2)}$ is easy to verify since $|\lambda^{(2)}|\leq n^{\gamma}= o\big(n^{(1-\eps')/4}/(\log n)^4\big) <  n_2^{1/4}/(\log n_2)^4$. 
    Here, we used $n_2=\Omega(n^{1-\eps'})$ and $4\gamma < 1-\eps'$, which follows directly from $\eps'+6\gamma<1$. 
    
    To conclude the proof, we apply Tur\'an's theorem to $G_3$. 
    This guarantees $G_3$ has a clique of size $$\min\Big(n_3,\Omega\big(p_3^{-1}\big)\Big)
        \ge \frac{n_2^{1-o(1)}}{|\lambda^{(2)}|^2}
        \geq \frac{n^{1-\eps'-o(1)}}{n^{2\gamma}}
        \geq n^{1-\eps'-2\gamma-o(1)}.$$
    Since $\eps'<\eps$, we conclude that $G$ contains a clique of size at least $n^{1-\eps-2\gamma}$, if $n$ is sufficiently large.
\end{proof}

Here, the bound $n^{1-2\gamma-\eps}$ is close to optimal (when $\rho,\eps$ are both small). Indeed, the Erd\H{o}s-R\'enyi random graph $G(n,p)$ with edge probability $p=1-n^{2\gamma-1}$ has no clique of size larger than $n^{1-2\gamma+o(1)}$, and its smallest eigenvalue satisfies $|\lambda_n|=\Theta(n^\gamma)$ with high probability. We now shift our focus to sparse graphs, with the goal of proving Theorem~\ref{thm:MAIN_SE3}, which we restate here. 

\begin{theorem}
Let $\gamma\in (0,1/10)$, and let $d$ be sufficiently large. Then, every graph $G$ of average degree $d$ whose smallest eigenvalue satisfies $|\lambda_n|\leq d^\gamma$ contains a clique of size at least $d^{1-4\gamma}$.
\end{theorem}
\begin{proof}
    We will reduce to \Cref{thm:clique1} by finding an induced subgraph $G_0\subseteq G$ whose average degree is polynomially close to the number of its vertices. 
    To do this, let $x$ be a vertex of maximum degree in $G$, and let $S$ be a set of $d$ neighbours of $x$. 
    We claim that $e(G[S])\geq d^2/4|\lambda_n|$, and then will take $G_0=G[S]$.
    
    To prove the claim, define a vector $v\in \mathbb{R}^{n}$ such that $v(x)=1$, $v(y)=\frac{\lambda_n}{d}$ for $y\in S$, and $v(z)=0$ if $z\not\in S\cup\{x\}$. Since $\lambda_n$ is the smallest eigenvalue of the adjacency matrix $A$, we must have $v^TAv\geq \lambda_n\|v\|_2^2$. 
    For this $v$, we also have 
    $v^T A v=2\lambda_n+2\frac{\lambda_n^2}{d^2}e(G[S])$ and $\|v\|_2^2=1+d\frac{\lambda_n^2}{d^2}\leq 3/2$.
    Combining these two observations, we have
    \[2\lambda_n+2\frac{\lambda_n^2}{d^2}e(G[S])=v^T A v\geq \lambda_n \|v\|_2^2\geq \frac{3}{2}\lambda_n.\]
    Therefore, we conclude that $e(G[S])\geq \frac{d^2}{4|\lambda_n|}$, as needed.
    
    Now, define $G_0=G[S]$; it has $d$ vertices and density at least $d^{-\gamma}/2$. 
    Hence, \Cref{thm:clique1} applies to $G_0$ with parameters $\rho=\gamma+\kappa$ and $\eps=2\gamma$, where $\kappa>0$ is an arbitrary small constant, assuming $d$ is sufficiently large. 
    These parameters clearly satisfy $\rho<1/2$ and $\eps+6\gamma=8\gamma<1$ (using $\gamma<1/10$) and we also have $\rho/\eps+2\gamma/(1-\eps-4\gamma)=1/2+2\gamma/(1-6\gamma)+\kappa/\eps<1$ for $\kappa$ sufficiently small with respect to~$\gamma$. 
    Hence, we conclude that $G_0$ contains a clique of size $d^{1-\eps-2\gamma}=d^{1-4\gamma}$. This completes the proof.
\end{proof}

Now we prove the MaxCut version of \Cref{thm:clique1}, and use it to deduce \Cref{thm:MAIN_MC1}.

\begin{theorem}\label{thm:clique_MaxCut}
  Let $\gamma\in (0,1/60)$ and $\delta>0$, then there exist $\rho>0$ such that the following holds for every sufficiently large $n$. Let $G$ be an $n$-vertex graph of edge density at least $n^{-\rho}$ such that $\surp(G)\leq n^{1+\gamma}$. Then $G$ contains a clique of size at least $n^{1-2\gamma-\delta}$.
\end{theorem}

\begin{proof}
The proof of this is essentially identical to the proof of \Cref{thm:clique1}. The only difference is that we cite the MaxCut versions of our main densification results:  \Cref{lemma:dens1_main_surplus}, \Cref{cor:dens2_main_MaxCut}, and \textit{(b)} of \Cref{theorem:dens3_main}. 
\end{proof}

\begin{proof}[Proof of \Cref{thm:MAIN_MC1}]
    Let $\gamma=\delta/2$, and let $\rho$ be the constant in \Cref{thm:clique_MaxCut}. 
    We show that $\eps=\min\{\rho/4,\gamma/2\}$ suffices. 
    Suppose for contradiction that $\surp(G)< m^{0.5+\eps}$.
    Let $n$ be the number of vertices in $G$. 
    We may assume that $G$ contains no isolated vertices. 
    Then, a result of Erd\H{o}s, Gy\'arf\'as, and Kohayakawa \cite{EGyK} shows that $\surp(G)\ge\frac{n}{6}$, so $m>(n/6)^{1/(0.5+\eps)}>n^{2-4\eps}$.
    Then edge density of $G$ is at least $n^{-4\eps}\ge n^{-\rho}$. 
    As $m \le n^2$, $\surp(G)\leq m^{0.5+\eps}\leq n^{1+2\eps}\leq n^{1+\gamma}$.
    We get the contradiction as \Cref{thm:clique_MaxCut} guarantees that $G$ contains a clique of size $n^{1-2\gamma-\delta}=n^{1-2\delta}\geq m^{1/2-\delta}$.
\end{proof}

\section{Edit distance from the union of cliques}\label{sect:stab}

\Cref{thm:MAIN_SE1,thm:dens2_main_MaxCut} become meaningless once the graph $G$ has density at most $n^{-c}$ for any small constant $c>0$, as then $G$ is already $o(1)$-close to the empty graph. 
However, in this section we prove \Cref{thm:MAIN_SE2}, which deals with these graphs, under a somewhat stronger condition on the smallest eigenvalue.

The proof works as follows. 
Using \Cref{thm:clique1}, we can repeatedly pull out ``large'' cliques in $G$ as long as the rest graph has ``many'' edges.
Then, we show that the union of these cliques induce almost all edges in $G$, so it is sufficient to consider this disjoint union of cliques.
Moreover, we show that between any two cliques, it is either very sparse or very dense.
We use this to derive that $G$ itself must resemble a disjoint union of cliques.

We start with the following simple lemma, which will be used to argue that a dense graph with large smallest eigenvalue cannot induce sparse subgraphs ($G[Y]$ in the following statement).
In fact, we can prove a stronger statement in terms of the surplus (recall that $\surp(G) \le n|\lambda_n|/4$ from \Cref{claim:surplus_and_lambdan}).

\begin{lemma}\label{lemma:inbalanced_partition}
Let $G$ be a graph on $n$ vertices. Let $X\cup Y$ be a partition of $V(G)$, and let $b=e(G[X,Y])$ and $c=e(G[Y])$. Then $\surp(G)\geq \frac{b^2}{4n^2}-c$.
\end{lemma}

\begin{proof}
If $a=e(G[X])$ satisfies $a\leq b/2$, then $\surp(G)$ is at least $$e(G[X,Y])-\frac{e(G)}{2}=b-\frac{a+b+c}{2}= \frac{b-a-c}{2} \ge \frac{b}{4}-\frac{c}{2}\ge \frac{b^2}{4n^2}-c,$$ as desired.

Otherwise, we have $b < 2a$ and we can take $p=b/(4a) \in [0,1/2)$. 
Let $U$ be a random subset of $X$, where each vertex is included independently with probability $1/2+p$, and consider the cut $(U, (X\setminus U)\cup Y)$. Each edge in $G[X]$ has probability $1/2-2p^2$ of being cut, and each edge between $X$ and $Y$ is cut with probability $1/2+p$. Therefore, the expected size of this cut is $a(1/2-2p^2)+b(1/2+p)$, showing that the expected surplus is 
\[a\Big(\frac{1}{2}-2p^2\Big)+b\Big(\frac{1}{2}+p\Big)-\frac{a+b+c}{2}=bp-2ap^2-\frac{c}{2}=\frac{b^2}{8a}-\frac{c}{2} \ge \frac{b^2}{4n^2}-c,\]
where we have used that $a=e(G[X])\leq n^2/2$ in the last step.
\end{proof}

Next, we show that a graph with large smallest eigenvalue contains a collection of large cliques such that almost all edges are contained in the subgraph induced by the union of these cliques.

\begin{lemma}\label{lemma:many_cliques}
Let $\gamma\in (0,1/6)$, then there exists $\alpha>0$ such that the following holds for sufficiently large $n$. Let $G$ be a graph on $n$ vertices such that $|\lambda_n|\leq n^{\gamma}$. Then there exists $X\subset V(G)$ such that the number of edges not in $G[X]$ is at most $n^{2-\alpha}$, and $G[X]$ can be partitioned into cliques of size $\sqrt{n}$. 
\end{lemma}

\begin{proof}
Choose any constants $\gamma_0,\delta_0>0$ such that $\gamma_0>\gamma$, $2\gamma_0+\delta_0<3/8$ and $6\gamma_0+\delta_0 < 1$. 
Then, there exists a constant $\rho \in (0,1/2)$ such that $\rho/\delta_0+2\gamma_0/(1-\delta_0-4\gamma_0) < 1$.
By Theorem \ref{thm:clique1}, every $n_0$-vertex graph of edge density at least $n_0^{-\rho}$ and smallest eigenvalue at least $-n_0^{\gamma_0}$ contains a clique of size $n_0^{1-2\gamma_0-\delta_0}\ge n_0^{5/8}$ as long as $n_0$ is sufficiently large.

We show that $\alpha=\min\{1/16,\rho/5,(1-\gamma/\gamma_0)/2\}$ suffices. 
Repeatedly delete vertices of degree less than $d=n^{1-2\alpha}$, and let $G_0$ be the resulting graph.
Note that we removed at most $dn$ edges. 
We are done if $G_0$ is empty, so we may assume $G_0$ has minimum degree at least $d$.
Let $X$ be the maximal subset of $V(G)$ which can be partitioned into disjoint cliques of size $\sqrt{n}$. 

The goal is now to show that the number of edges not in $X$ is at most $n^{2-\alpha}$. To do that, we let $Y=V(G_0)\backslash X$ be the set of remaining vertices, and we show that $|Y| \le n^{1-2\alpha}$. This would be sufficient to complete the proof since the number of edges of $G$ not in $G[X]$ is at most 
$$dn+e(G[X,Y])+e(G[Y])\leq dn+|Y|n\leq n^{1-\alpha}.$$

Suppose, for the sake of contradiction, that $|Y|> n^{1-2\alpha}$. 
By maximality of $X$, the set $Y$ contains no clique of size $\sqrt{n}$, and we will ultimately derive the contradiction to this.

\begin{claim}
The edge density of $G[Y]$ is at least $|Y|^{-\rho}$.
\end{claim}
\begin{proof}
    Following the notation in the \Cref{lemma:inbalanced_partition}, let $b=e(G_0[X,Y])$ and $c=e(G_0[Y])$.
    As $G_0$ has minimum degree $d$, we have $b+2c\geq d|Y|$. 
    If $c \ge \frac{1}{4}d|Y|$, then the density of $G[Y]$ is at least $\frac{1}{4}d|Y|/\binom{|Y|}{2} \ge \frac{d}{2|Y|}\ge\frac{n^{1-2\alpha}}{2|Y|}>|Y|^{-5\alpha}\ge |Y|^{-\rho}$.
    Now, we assume that $c < \frac{1}{4}d|Y|$, so $b > \frac{1}{2}d|Y|$.
    By \Cref{lemma:inbalanced_partition}, we have
    $$\surp(G)\geq \surp(G_0)\geq \frac{b^2}{4n^2}-c.$$
    On the other hand, $\surp(G)\ge \frac{1}{4}|\lambda_n|n \leq n^{1+\gamma}$ by \Cref{claim:surplus_and_lambdan}, so     
    $$c\geq \frac{b^2}{4n^2}-n^{1+\gamma}> \frac{d^2|Y|^2}{16n^2}-n^{1+\gamma}\geq \frac{d^2|Y|^2}{20n^2}.$$
    In the last inequality we used that $d, |Y|\geq n^{1-2\alpha}$.
    Hence, $G[Y]$ has at least $e(G[Y])\geq d^2|Y|^2/(20n^2)\ge |Y|^2n^{-4\alpha}/20\geq |Y|^{2-5\alpha} \ge |Y|^{2-\rho}$ edges, which shows that $G_I$ has edge density at least $|Y|^{-\rho}$. 
\end{proof}

By Cauchy's interlacing theorem, the smallest eigenvalue of $G[Y]$ is at least that of $G$, i.e. we have $|\lambda_n(G[Y])|\leq |\lambda_n(G)|\leq n^\gamma$. 
Since $|Y|> n^{1-2\alpha}$, this shows $|\lambda_n(G[Y])|\leq |Y|^{\gamma/(1-2\alpha)}\le |Y|^{\gamma_0}$.

But then, as discussed in the beginning, \Cref{thm:clique1} guarantees that $Y$ contains a clique of size $|Y|^{5/8}>n^{(5/8)(1-2\alpha)}\geq \sqrt{n}$, contradicting the assumption that $Y$ contains no clique of size $\sqrt{n}$.
Therefore, we must have $|Y|\leq n^{1-2\alpha}$, and as discussed above, this finishes the proof.
\end{proof}

Next, we show that the graph between two cliques must be either very dense or very sparse. 

\begin{lemma}\label{lemma:bipartite_eigenvalue}
Let $G$ be an $n$-vertex graph with the smallest eigenvalue $\lambda_n$ and let $X,Y\subset V(G)$ be disjoint cliques of the same size. Then $G[X,Y]$ has either at most $O(|\lambda_n|^2 |X|)$ edges, or at least $|X|^2-O(|\lambda_n|^2 |X|)$ edges.
\end{lemma}
\begin{proof}
Recall that in Section~\ref{sect:chowla}, we identified a graph $H_k$, consisting of a clique of size $2k$ and an additional vertex with $k$ neighbours in the clique, with the property that no graph $G$ with at $\lambda_n\geq -\sqrt{k/2}$ contains $H_k$ as an induced subgraph (see Claim~\ref{claim:H_k} and the subsequent discussion). Thus, if we set $k=2|\lambda_n|^2$, the graph $G$ does not contain $H_k$ as an induced subgraph. 
We may assume that $|X|=|Y|\geq 4k$, otherwise the statement is trivial.
Then each vertex in $X$ has either at most $k$ neighbours or at most $k$ non-neighbours in $Y$. 
Moreover, each vertex in $Y$ has at most $k$ neighbours or at most $k$ non-neighbours in $X$.  
Let $X_0\subset X$ be the set of vertices with at most $k$ neighbours, and let $X_1=X\setminus X_0$, and define $Y_0,Y_1\subset Y$ analogously. 
Suppose that $X_0$ and $Y_1$ both have size at least $2k$. 
If the number of edges between $X_0$ and $Y_1$ is at least $|X_0||Y_1|/2$, then there is a vertex in $X_0$ with at least $|Y_1|/2\geq k$ neighbours in $Y_1$, contradiction. 
On the other hand, if the number of edges between $X_0$ and $Y_1$ is at most $|X_0||Y_1|/2$, then there is a vertex in $Y_0$ with at least $|X_0|/2\geq k$ non-neighbours, contradiction. 
Therefore, it must hold that at least one of $X_0$ or $Y_1$ has size at most $2k$.
If $|X_0|\leq 2k$, then $G[X,Y]$ has at least $|X_1|(|X|-k)\geq |X|^2-3k|X|=|X|^2-O(|\lambda_n|^2|X|)$ edges. 
Otherwise, if $|Y_1|\leq 2k$, then $G[X,Y]$ has at most $|Y_0|k+|Y_1||X|\leq 3k |X|=O(|\lambda_n|^2|X|)$ edges. 
\end{proof}

We are ready to prove Theorem \ref{thm:MAIN_SE2}, which we restate here for convenience.

\begin{theorem}
For every $\gamma\in (0, 1/6)$, there exists $\alpha>0$ such that for every sufficiently large $n$ we have the following. If $G$ is an $n$-vertex graph with $|\lambda_n|\leq n^{\gamma}$, then $G$ is $n^{-\alpha}$-close to the vertex-disjoint union of cliques.
\end{theorem}

\begin{proof}
    Let $\alpha_0=\alpha_0(\gamma)>0$ be the constant guaranteed by \Cref{lemma:many_cliques}. 
    We show that $\alpha=\min\{1/7,\alpha_0/2\}$ works.
    
    By \Cref{lemma:many_cliques}, there exists a set $X\subset V(G)$ such that $X$ can be partitioned into the union of cliques of size $\sqrt{n}$, and $G$ has at most $n^{2-\alpha_0}$ edges not in $G[X]$.
    Let $C_1,\dots,C_I$ be the cliques of size $\sqrt{n}$ partitioning $X$; then $I=|X|/\sqrt{n}$.  
    \Cref{lemma:bipartite_eigenvalue} implies that the bipartite graph between $C_i$ and $C_j$ has either at most $O(\sqrt{n}|\lambda_n|^2)$, or at least $n-O(\sqrt{n}|\lambda_n|^2)$ edges. 
    If $n$ is sufficiently large, $O(|\lambda_n|^2/\sqrt{n})<n^{-1/6}$. Define the auxiliary graph $\Gamma$ on vertex set $\{1,\dots,I\}$, where we connect $i$ and $j$ if $G[C_i,C_j]$ has density at least $1-n^{-1/6}$.
    
    \begin{claim}
    $\Gamma$ contains no cherry, i.e. if $ij, jk\in E(\Gamma)$, then $ik\in E(\Gamma)$ as well.
    \end{claim}
    
    \begin{proof}
        If there is a triple $(C_i,C_j,C_k)$ such that $G[C_i,C_j], G[C_j,C_k]$ have density at least $1-n^{-1/6}$, but $G[C_i,C_k]$ has density at most $1-n^{-1/6}$, then we can apply \Cref{lemma:dens2_3parts}. 
        This lemma shows that $\surp(G[C_i\cup C_j\cup C_k])\geq (1/4-3n^{-1/6})|C_i|^2\geq n/8$. 
        Then, \Cref{claim:surplus_and_lambdan} implies the smallest eigenvalue of $G[C_i\cup C_j\cup C_k]$ is at most $-n^{1/2}/6$, which is a contradiction since $|\lambda_n(G[C_i\cup C_j\cup C_k])|\leq|\lambda_n(G)|=n^\gamma$. 
    \end{proof}
    
    Recall that graphs containing no cherry are the disjoint union of cliques. 
    Therefore, we can partition $V(\Gamma)$ into sets $I_1,\dots,I_\ell$ such that $\Gamma[I_a]$ is a clique and there are no edges between $I_a$ and $I_b$ in $\Gamma$. 
    But this gives a partition of $X$ into sets $Y_1,\dots,Y_\ell$ by setting $Y_a=\bigcup_{i\in I_a}C_i$. 
    Define $\widetilde{G}$ to be the graph on vertex set $V(G)$, where $Y_1,\dots,Y_\ell$ are cliques, and all edges of $\widetilde{G}$ are contained in one of these cliques.
    
    We prove that $\widetilde{G}$ is $n^{-\alpha}$-close to $G$. For $1\leq i<j\leq I$, $G[C_i,C_j]$ and $\widetilde{G}[C_i,C_j]$ differ by at most $n\cdot n^{-1/6}=n^{5/6}$ edges.
    Therefore, $G[X]$ and $\widetilde{G}[X]$ differ by at most $$\binom{|X|/\sqrt{n}}{2}\cdot n^{5/6}\leq n^{11/6}$$
    edges. 
    Furthermore, there are at most $n^{2-\alpha_0}$ edges of $G$ not in $G[X]$, so $G$ and $\widetilde{G}$ differ by at most $n^{2-\alpha_0}+n^{11/6}\leq n^{2-\alpha}$ edges. 
    This finishes the proof.
\end{proof}

\section{Further directions}

We conclude our paper by discussing some open problems. One of our main contributions is \Cref{thm:MAIN_SE3}, which shows that every graph of average degree $d$ with $|\lambda_n|\leq d^\gamma$, where $\gamma\in (0, 1/10)$, contains a clique of size $d^{1-4\gamma}$.
The constants $1/10$ and $4$ are likely not optimal, and it would be interesting to understand the limits of this theorem.  
Yang and Koolen~\cite{YK} conjecture that if $d$ is exponentially larger than $\lambda_n$, then $G$ contains a clique of size $\Omega(d/|\lambda_n|^2)$. We believe this can be already true if  $|\lambda_n|=O(d^\gamma)$ for some small $\gamma$.

\begin{question}
    Let $G$ be an $n$-vertex graph of average degree $d$ and $|\lambda_n| = \Theta(d^\gamma)$ for some $\gamma\in (0, 1/2)$.
    How small can its clique number $\omega(G)$ be? 
\end{question}

Moreover, despite the substantial progress presented in this paper, finding precise exponents for the MaxCut problem in $H$-free graphs remains open. In particular, it is still an intriguing problem to determine the largest exponent $\alpha_r$ such that every $K_r$-free graph $G$ with $m$ edges satisfies $\surp(G)\geq \Omega(m^{\alpha_r})$. The work of Alon \cite{AlonMaxCut} shows that $\alpha_3=4/5$, but the value of $\alpha_4$ is already unknown. Following our calculations, we established that $\alpha_r>0.51$ for all fixed $r \ge 3$, improving the barrier $1/2$. However, the celebrated conjecture of Alon, Bollob\'as, Krivelevich and Sudakov \cite{ABKS} asserts that $\alpha_r>3/4$.

\section*{Acknowledgements}
We are grateful to Ilya Shkredov for bringing Chowla's problem to our attention, and pointing out connections to our earlier results. Moreover, we thank Benjamin Bedert for sharing his manuscript on Chowla's problem and for an insightful conversation about it. Finally, we would like to thank Igor Balla, Clive Elphick, Victor Falgas-Ravry, Jacob Fox,
Lianna Hambardzumyan, Jack Koolen, Anqi Li, Nitya Mani, Benny Sudakov, and  Quanyu Tang  for many valuable discussions on various parts of this project.

\appendix

\section{Chowla's cosine problem in finite groups} \label{sec: chowla for groups}

In Section~\ref{subsec:intro_chowla} we discussed Chowla's cosine problem and its relation to Littlewood's $L_1$-problem over the integers. Variants of these two problems have also been studied in the setting of discrete abelian groups. Namely, in 2009, Green and Konyagin \cite{GK09} studies how small the $L_1$-norm of a dense set $A\subseteq \mathbb{Z}/p\mathbb{Z}$ can be. If $\widehat{\mathds{1}_A}$ denotes the Fourier transform of the indicator function of $A$ over $\mathbb{Z}/p\mathbb{Z}$, they showed that $\sum_{r}|\widehat{\mathds{1}_A}(r)|\geq (\log p)^{1/3-o(1)}$, which was later improved by Sanders \cite{S10} to $(\log p)^{1/2-o(1)}$. For sparser sets $A\subseteq \mathbb{Z}/p\mathbb{Z}$, a similar question has been studied by Schoen \cite{S17} and Konyagin and Shkredov \cite{KS16}. 

Paralleling the extensions of the Littlewood $L_1$-problem, Sanders \cite{Sanders} extended Chowla's problem to discrete abelian groups $\Gamma$ as follows. For a symmetric subset $A\subset \Gamma$, one can define 
$$M_{\Gamma}(A)=\sup_{y\in \widehat{\Gamma}} -\widehat{\mathds{1}_A}(y),$$
where $\widehat{\Gamma}$ is the dual group of $\Gamma$.
%where $\widehat{f}:\Gamma\rightarrow\mathbb{C}$ is the discrete Fourier transform of a function $f:\Gamma\rightarrow\mathbb{C}$, and $\mathds{1}_A$ is the indicator function of $A$.\am{we already said what $\widehat{\mathds{1}_A}(y)$ is in the previous paragraph} 
Note that since $A$ is a symmetric set, $\widehat{\mathds{1}_A}$ is a real function. 

In this language, Chowla's cosine problem asks to show that for every $A\subseteq \bZ$ we have $M_{\bZ}(A)\geq \Omega(\sqrt{|A|})$, but it is also natural to ask how small $M_{\Gamma}(A)$ can be in a general group $\Gamma$.
One can quickly observe that $M_{\Gamma}(A)$ need not go to infinity with the size of $A$. 
Indeed, if $A$ is a subgroup of $\Gamma$, then $M_{\Gamma}(A)=0$. 
On the other hand, Sanders (\cite[Theorem 1.3]{Sanders}) proved that if $A$ is far from a subgroup of $\Gamma$, then $M_{\Gamma}(A)$ is necessarily large. 
Formally, he proved that for every $\delta>0$ there exists $c(\delta)>0$ such that if $M_\Gamma(A)\leq |\Gamma|^{c(\delta)}$, then there is some subgroup $H<\Gamma$ satisfying  $|H\triangle A|\leq \delta |\Gamma|$. 
Noting that the Fourier coefficients of $\mathds{1}_A$ correspond to the eigenvalues of the Cayley graph $\Cay(\Gamma,A)$ generated by $A$, we can use our main results to improve these bounds when $\Gamma$ is finite, and to extend them to non-abelian groups as well.

Recall, the Cayley graph $\Cay(\Gamma,A)$, where $\Gamma$ is a finite group and $A\subseteq \Gamma$ is symmetric, has vertex set $\Gamma$ and two vertices $x,y\in \Gamma$ are adjacent if and only if $xy^{-1} \in A$.
We define $M_\Gamma(A) := \max-\lambda$, where the maximum is taken among all eigenvalues of $\Cay(\Gamma,A)$, which then coincides with the earlier definition for finite abelian groups. 
As discussed above, $M_\Gamma(A)=0$ if $A$ is a subgroup of $\Gamma$, so $M_\Gamma(A)$ is small if $A$ is close to a subgroup of $\Gamma$.
Here, we show the opposite: if $M_\Gamma(A)$ is small, then $A$ must be close to a subgroup of $\Gamma$.

\begin{theorem}\label{thm:chowla_group}
Let $\delta,\gamma>0$, then the following holds for every sufficiently large finite group  $\Gamma$. 
Let $A\subset \Gamma$ such that $A=A^{-1}$. 
If $M_\Gamma(A)\leq |\Gamma|^{\gamma}$ and $\gamma\in (0, 1/4)$ then there exists a subgroup $H<\Gamma$ such that $$|H\triangle A|\leq \delta |\Gamma|.$$
Moreover, if $\alpha>0$ is sufficiently small as a function of $\gamma$ and $M_{\Gamma}(A)< |\Gamma|^{\gamma}$ for $\gamma\in (0, 1/6)$, then there exists a subgroup $H<\Gamma$ such that
$$|H\triangle A|\leq |\Gamma|^{1-\alpha}.$$
\end{theorem}

The main idea of the proof is to show that $\Cay(\Gamma,A)$ is close to the disjoint union of cliques if and only if $A$ is close to a subgroup of $\Gamma$. This is proved in the following lemma.

\begin{lemma}\label{lemma:chowla_stability}
Let $\Gamma$ be a group and let $A\subset \Gamma$, $A=A^{-1}$, such that the number of pairs $(x,y)\in A\times A$ such that $xy\not\in A$ is at most $\eps |A|^2$. Then there exists a subgroup $H<\Gamma$ such that $|H\triangle A|\leq O(\eps^{1/2} |A|)$.
\end{lemma}

In the proof, we use an old theorem of Freiman \cite{freiman} on sets of very small doubling, sometimes referred to as Freiman's $3/2$-theorem. See also the blog of Tao \cite{tao} for a short proof. Given subsets $A,B$ of a group $\Gamma$, we write $A\cdot B=AB=\{xy: x\in A, y\in B\}$.

\begin{theorem}[Freiman's $3/2$-theorem]\label{thm:chowla_3/2}
Let $\Gamma$ be a group and let $A\subset \Gamma$ such that $|AA^{-1}|< \frac{3}{2}|A|$. Then $A A^{-1}$ and $A^{-1} A$ are both subgroups of $\Gamma$.
\end{theorem}

\begin{proof}[Proof of Lemma \ref{lemma:chowla_stability}]
We may assume that $1/|A| \le \eps<1/1000$, otherwise the statement is trivial. Also, we can assume that the identity $1_\Gamma\in A$, as adding $1_\Gamma$ does not increase the number of pairs $(x,y)\in A\times A$ with $xy^{-1}\not\in A$, and it only changes the size of $A$ by $1$.

Let $N$ be the number of pairs $(x,y)\in A\times A$ such that $xy\not\in A$. 
Also, for every $x\in A$, let $$N(x)=|(xA)\triangle A|=|(xA)\backslash A|+|A\backslash (xA)|=|(xA)\backslash A|+|(x^{-1}A)\backslash A|.$$
Hence, we have $$
    N= \sum_{x \in A} |(xA)\setminus A|
    = \frac{1}{2}\sum_{x\in A}N(x).
$$
Therefore, $\frac{1}{|A|}\sum_{x\in A}N(x)\leq 2\eps |A|$. 
Let $\delta=(2\eps)^{1/2}$, and define $$B=\{x\in A: N(x)\leq \delta |A|\}.$$
Then by simple averaging, we have $|B|\geq (1-2\eps/\delta)|A|=(1-\delta)|A|$. 
We also note that $B=B^{-1}$ as $N(x)=N(x^{-1})$, and $1_\Gamma\in B$. 
Our goal is to show to apply Freiman's 3/2-theorem to the set $B$, and so we now show that $|BB^{-1}|=|B\cdot B|\leq \frac{3}{2}|B|$.

Observe that for every $x_1,x_2\in B$, we can use the triangle inequality to write
$$|(x_1x_2A)\triangle A|\leq |(x_1A) \triangle A|+|(x_1x_2A)\triangle (x_1A)|\leq 2\delta |A|.$$
In particular, for every $x\in B\cdot B$, we have $|(xA)\triangle A|\leq 2\delta |A|$. 
Therefore, $$\sum_{x\in B\cdot B} |(xA)\triangle A|\leq 2\delta |A||B\cdot B|.$$
On the other hand, $\sum_{x\in B\cdot B} |(xA)\triangle A|$ counts the number of pairs $(x,y)\in (B\cdot B)\times A$ such that $xy\not\in A$ or $y\not\in xA$. 
For every fixed $y$, the number of such pairs is clearly lower bounded by $|B\cdot B|-|A|$. Therefore, we can also write
$$\sum_{x\in B\cdot B} |(xA)\triangle A|\geq |A|(|B\cdot B|-|A|).$$
Comparing the lower and upper bounds on $\sum_{x\in B\cdot B} |(xA)\triangle A|$, we get
$$2\delta |A||B\cdot B|\geq |A|(|B\cdot B|-|A|),$$
from which
$$|B\cdot B|\leq \frac{1}{1-2\delta} |A|<(1+4\delta)|A|.$$
Since $|B|\geq (1-\delta)|A|$ and $\delta$ is sufficently small, we conclude that $|B\cdot B|< 3|B|/2$. 
By \Cref{thm:chowla_3/2}, $B\cdot B^{-1}=B\cdot B$ is a subgroup of $\Gamma$. Since $1_\Gamma\in B$, we have $B\subset B\cdot B$, so $|A\cap (B\cdot B)|\geq |B|$. In conclusion
$$|A\triangle (B\cdot B)|\leq |A|+|B\cdot B|-2|B|\leq  6\delta |A|,$$
showing that $H=B\cdot B$ suffices.
\end{proof}

With \Cref{lemma:chowla_stability} in our hands, \Cref{thm:chowla_group} follows almost immediately from \Cref{thm:MAIN_SE1,thm:MAIN_SE2}.

\begin{proof}[Proof of \Cref{thm:chowla_group}]
    We start with the first part of the theorem. 
    We may assume that $|A|\geq \delta |\Gamma|$, otherwise the statement is trivial by choosing $H=\{1_\Gamma\}$. Also, fix a parameter $\beta\ll \delta^4$.
    
    Let $G=\Cay(\Gamma,A)$, and let $\lambda_n=-M_\Gamma(A)$ be the smallest eigenvalue of $G$, $n=|\Gamma|$. 
    We may assume that $1_\Gamma\not\in A$, by noting that removing $1_\Gamma$ shifts the eigenvalues by $-1$. Therefore, $G$ is a simple graph with no loops. 
    If $n$ is sufficiently large as a function of $\delta$, \Cref{thm:MAIN_SE1} shows that the inequality $|\lambda_n|\leq n^{\gamma}$ for $\gamma\in (0, 1/4)$ implies that $G$ is $\beta$-close to a disjoint union of cliques. 
    But then $G$ contains at most $3\beta n^3$ induced cherries by \Cref{lemma:dens2_removal}. Recall, a cherry is a triple of vertices $u, v, w\in G$ such that $uv, vw\in E(G)$, but $uw\notin E(G)$.
    Since $G=\Cay(\Gamma, A)$, every cherry corresponds to a pair $x=uv^{-1}\in A, y=vw^{-1}\in A$ for which we have $xy=uw^{-1}\notin A$, and each such pair $(x, y)\in A$ corresponds to exactly $n$ cherries. Hence, there are at most $6\beta n^2\leq 6(\beta/\delta^2)|A|^2$ pairs $(x,y)\in A\times A$ such that $xy\not\in A$. 
    By \Cref{lemma:chowla_stability}, then $|A\triangle H|\leq O(\beta^{1/2}/\delta |A|)\leq \delta n$ for some subgroup $H<\Gamma$, since $\beta\ll\delta^4$.
    
    The second part of the theorem follows essentially in the same manner, but we cite \Cref{thm:MAIN_SE2} instead of \Cref{thm:MAIN_SE1}. We omit the details.
\end{proof}

\section{Stability of graphs with small MaxCut}\label{sect:MaxCut_stability}

One of the main results of our paper, Theorem~\ref{thm:MAIN_SE2} shows that every graph with $|\lambda_n|\leq n^\gamma$, for $\gamma\in (0, 1/6)$, is $n^{-\alpha}$-close to a disjoint union of cliques. Here, we present a variant of this result concerning graphs with no large MaxCut, that is, we show that graphs with small maximum cut are $n^{-\alpha}$-close to a disjoint union of cliques.

\begin{theorem}\label{thm:MAIN_MC2}
    There exist absolute constants $\eps, \alpha>0$ such that the following holds for every sufficiently large $n$. 
    If $G$ is an $n$-vertex and $m$-edge graph with no cut of size larger than $\frac{m}{2}+m^{\frac{1}{2}+\eps}$, then $G$ is $n^{-\alpha}$-close to a disjoint union of cliques.
\end{theorem}

The proof of this theorem mostly follows the same strategy as the proof of \Cref{thm:MAIN_SE2}. We therefore recommend the readers to familiarize themselves with the arguments of \Cref{sect:stab} before reading this section. 
In short, the proof proceeds in two steps --- first, we show that the graph can be partitioned into large vertex-disjoint cliques, as the following lemma shows.

\begin{lemma}\label{lemma:many_cliques_MaxCut}
    There exist absolute constants $\alpha, \gamma>0$ such that the following holds for every sufficiently large $n$. 
    Let $G$ be a graph on $n$ vertices such that $\surp(G)\leq n^{1+\gamma}$. Then there exists $X\subset V(G)$ such that the number of edges not in $G[X]$ is at most $n^{2-\alpha}$, and $G[X]$ can be partitioned into cliques of size $n^{1-3\gamma}$. 
\end{lemma}

\begin{proof}
    The proof this is almost identical to the proof of Lemma \ref{lemma:many_cliques}. 
    The only difference is that we use \Cref{thm:clique_MaxCut} to pull out cliques of size $n^{1-3\gamma}$ instead of using \Cref{thm:clique1} to pull out cliques of size $\sqrt{n}$. 
    We omit further details.
\end{proof}

In the second step of the proof, we analyse the edges between pairs of cliques coming from \Cref{lemma:many_cliques_MaxCut}, with the goal of showing that any two cliques induce an almost empty or almost complete bipartite graph. 
This step is analogous to \Cref{lemma:bipartite_eigenvalue}, whose proof relied on finding a simple forbidden subgraph $H_k$. 

Unfortunately, the graphs of small surplus no longer avoid such a simple forbidden structure. Instead, we show that any two large cliques in $G$ induce a graph which is very close to a complement of a complete bipartite graph (perhaps on a smaller vertex set). For the precise statement, see \Cref{lemma:bipartite_MaxCut}. 
We prepare the proof the following lemma.

A \emph{Boolean matrix} is a matrix with only zero and one entries. First, we show that if a Boolean matrix is approximated by a rank one matrix, then it is also approximated by a rank one Boolean matrix, or equivalently, a combinatorial rectangle.

\begin{lemma}\label{lemma:rank1_approximation}
Let $A$ be an $n\times n$ Boolean matrix, and let $\delta\geq 0$. If there exist $u,v\in \mathbb{R}^n$ such that $\|A-uv^T\|_F^2\leq \delta n^2$, then there exist $x,y\in \{0,1\}^n$  such that $\|A-xy^T\|_F^2\leq O(\delta^{1/3}n^2)$.
\end{lemma}

\begin{proof}
    Without loss of generality, we may assume that $\delta\leq 1$. 
    Furthermore, we may assume that $u$ and $v$ has nonnegative entries, as replacing every entry with the absolute value does not increase $\|A-uv^T\|_F^2$. 
    Observe that $\|u\|_2^2\|v\|_2^2=\|uv^T\|_F^2$, which shows that $$\|u\|_2\|v\|_2\leq \|A\|_F+\sqrt{\delta}n\leq 2n.$$
    We may rescale $u$ and $v$ such that $\|u\|_2=\|v\|_2\leq \sqrt{2n}$. 
    Let $\eta=\delta^{1/6}$, and define $x,y\in \{0,1\}^n$ such that 
    $$
        x_i=\mathds{1}_{u_i \ge \eta}
        \;\;\text{ and }\;\;
        y_i=\mathds{1}_{v_i \ge \eta}
        \quad \text{ for all } 1 \le i \le n.
    $$
 
     We show that $xy^T$ is a good approximation of $A$. 
     Note that $\|A-xy^T\|_F^2$ is the number of pairs $(i,j)$ such that $A_{i,j}\neq x_iy_j$. 
     We count these pairs in three cases, and upper bound each case by $O(\delta^{1/3}n^2)$.
 \begin{description}
     \item[Case 1.] $A_{i,j}=1$ and $x_i=0$.

     In this case, we have $u_i< \eta$. If $v_j\leq 1/(2\eta)$, then $(A_{i,j}-u_iv_j)^2>1/4$, so there are at most $4\delta n^2$ such pairs $(i,j)$. On the other hand, the number of $j$ such that $v_j\geq 1/(2\eta)$ is at most $8\eta^2 n$, as $\|v\|_2^2=\sum_{j=1}^nv_j^2\leq 2n$. Therefore, the number of $(i,j)$ such that $A_{i,j}=1$ and $x_i=0$ is at most $$4\delta n^2+8\eta^2n^2=4\delta n^2+8\delta^{1/3}n^2=O(\delta^{1/3}n^2).$$

    \item[Case 2.] $A_{i,j}=1$ and $y_j=0$.
    
    This is symmetric to the previous case, so the number of such pairs is also at most $O(\delta^{1/3}n^2)$.
     
     \item[Case 3.] $A_{i,j}=0$ and $x_i=y_j=1$.

     In this case, $u_i\geq \eta$ and $v_j\geq \eta$, so $(A_{i,j}-u_iv_j)^2\geq \eta^4$. Thus, the total number of pairs $(i,j)$ in this case is at most $\delta n^2/\eta^4=\delta^{1/3}n^2$.\qedhere
 \end{description}
\end{proof}
% Next, we prove a simple technical lemma which shows that the union of two cliques has large surplus as long as the two cliques are not too disjoint, and do not overlap too much. 

% \begin{lemma}\label{lemma:union_of_cliques}
% Let $G$ be a graph such that $V(G)=C_1\cup C_2$ and $E(G)=\binom{C_1}{2}\cup \binom{C_2}{2}$. Let $|C_1\setminus C_2|=a$, $|C_2\setminus C_1|=b$ and $|C_1\cap C_2|=c$. Then $$\surp(G)\geq \frac{1}{4}\min\{a^2,b^2,c^2\}.$$ 
% \end{lemma}

% \begin{proof}
% Let $A=C_1\setminus C_2$, $B=C_2\setminus C_1$, and $C=C_1\cap C_2$. We may assume that $a=b$. Otherwise, if, say $a\leq b$, we remove vertices of $B\setminus C$ until its size is exactly $a$. Then it is enough to show that the resulting graph has surplus at least $\frac{1}{4}\min\{a^2,c^2\}$.

% The number of edges of $G$ is 
% $$e(G) = 2\binom{a+c}{2}-\binom{c}{2}
%     = (a+c)(a+c-1)-\frac{c(c-1)}{2} 
%     = a^2+2ac + c^2 - a - c - \frac{c^2-c}{2}
%     < a^2+2ac+\frac{c^2}{2}.$$ 
% If $c\leq a$, then define the cut $(U,V)$ such that $U$ is some $(a+c)/2$ element subset of $A$ together with some $(a+c)/2$ element subset of $B$. The number of edges in this cut is $(a+c)^2/2$. Hence, the surplus of $G$ is at least $\frac{1}{2}(a+c)^2-\frac{1}{2}e(G)>c^2/4$.

% If $c\geq a$, then define the cut $(U,V)$ such that $U$ is some $(a+c)/2$ element subset of $C$. Then the number of edges in this cut is $\frac{a+c}{2}\cdot \frac{c-a}{2}+2\frac{a+c}{2}\cdot a=\frac{c^2}{4}+\frac{3}{4}a^2+ac$. Therefore, the surplus is at least $a^2/4$.
% \end{proof}

Now we are ready to prove the counterpart of \Cref{lemma:bipartite_eigenvalue}: if the surplus is small, then between two disjoint cliques of the same size, it is either very sparse or very dense.

\begin{lemma}\label{lemma:bipartite_MaxCut}
    Let $\alpha, \gamma>0$ be sufficiently small constants, and let $G$ be an $n$-vertex graphs with $\surp(G)\leq n^{1+\gamma}$, and let $X, Y\subset V(G)$ be disjoint cliques of the size $|X|=|Y|\geq n^{1-3\gamma}$. 
    Then $G[X, Y]$ has either at most $O(|X|^{2-\alpha})$ edges or at least $|X|^2-O(|X|^{2-\alpha})$.
\end{lemma}
\begin{proof}
Let $|X|=|Y|=k$, and let $H=G[X\cup Y]$. For simplicity, we denote the vertices of $X$ by $1, \dots, k$ and the vertices of $Y$ by $k+1, \dots, 2k$. Since $\surp(H)\leq \surp(G)\leq n^{1+\gamma}$ (cf. \Cref{sect:prelim}), we have $\surp(H)\leq k^{\frac{1-\gamma}{1-3\gamma}}\leq k^{1+5\gamma}$.

Let $A$ be the adjacency matrix of $H$ with eigenvalues $\lambda_1\geq \dots\geq \lambda_{2k}$. Furthermore, let $M$ be the adjacency matrix of $\overline{H}$, which is a bipartite graph, and let $\mu_1\geq \dots\geq \mu_{2k}$ be the eigenvalues of $M$. 
Note that $\overline{H}$ is bipartite, so $\mu_i=-\mu_{2k+1-i}$ for $i\in [2k]$. 
    
We can use Lemma~\ref{lemma:surp_star} \textit{(ii)} to obtain a lower bound on the surplus of $H$ based on its eigenvalues, and Lemma~\ref{lemma:dens4_weyl} to relate it to the eigenvalues of $\overline{H}$. Concretely, we have
$$\surp^*(H)
    = \Omega\left(\frac{1}{\sqrt{k}}\sum_{\lambda_i<0}\lambda_i^2\right)
    = \Omega\left(\frac{1}{\sqrt{k}}\sum_{i\neq 1,\mu_i>0}\mu_i^2\right)
    =\Omega\left(\frac{1}{\sqrt{k}}\sum_{i\neq 1,2k}\mu_i^2\right).$$
In addition, recall that $\surp(H) = \Omega(\surp^*(G)/\log k)$ from \Cref{claim:groth}.
Recall that $\surp(H)\leq k^{1+5\gamma}$. 
We acquire $\sum_{i\neq 1, 2k}\mu_i^2 \le k^{3/2+5\gamma+o(1)}$.

On the other hand, we can express $\sum_{i\neq 1, 2k}\mu_i^2$ as follows. The matrix $M$ has the form $M=\begin{pmatrix}0 & B \\ B^T & 0\end{pmatrix}$ with an appropriate $k\times k$ matrix $B$. The principal eigenvector of $M$ can be written as $v_1=(u,v)$, where $u,v\in\mathbb{R}^k$ correspond to the two vertex classes of $H$. Then the eigenvector corresponding to the smallest eigenvalue $\lambda_{2k}=-\lambda_1$ is $v_{2k}=(u,-v)$, and we have
\begin{align*}
    \sum_{i\neq 1,2k}\mu_i^2
    &=\big\|M-\lambda_1v_1v_1^T-\lambda_{2k}v_{2k}v_{2k}^T \big\|_F^2\\
    &=\left\|\begin{pmatrix} 0 & B \\ B^T & 0\end{pmatrix}
        -\lambda_1\begin{pmatrix} uu^T & uv^T \\ vu^T & vv^T\end{pmatrix}
        +\lambda_1\begin{pmatrix} uu^T & -uv^T \\ -vu^T & vv^T\end{pmatrix}\right\|_F^2
    =2\big\|B-2\lambda_1uv^T \big\|_F^2.
\end{align*}
This means $\big\|B-2\lambda_1vu^T \big\|_F^2\le \frac{1}{2}\sum_{i\neq 1,2k}\mu_i^2 \le  k^{3/2+5\gamma+o(1)}$, i.e. $B$ is well-approximated by a rank-1 matrix.
By \Cref{lemma:rank1_approximation}, there exist $x,y\in \{0,1\}^k$ such that $\|B-xy^T\|_F^2 \le k^{11/6+5\gamma/3+o(1)} \leq k^{2-\alpha}/2$ (as $k$ is large enough). 
The matrix $xy^T$ naturally corresponds to a bipartite graph $H'$ with two parts $\{1,\dots,k\}$ and $\{k+1,\dots,2k\}$ whose edges are all that cross $X=\{i\in [k]: x_i=1\}$ and $Y=\{j+k: j\in [k], y_j=1\}$.

Consider the complement $\overline{H'}$. 
Since $\|B-xy^T\|_F^2\leq k^{2-\alpha}/2$, the graphs $\overline{H}$ and $H'$ differ in at most $k^{2-\alpha}/2$ edges. Equivalently, $H$ and $\overline{H'}$ differ in at most $k^{2-\alpha}/2$ edges. 
Hence, we are done if $H'$ has at most $k^{2-\alpha}/2$ edges, or at least $k^2-k^{2-\alpha}/2$ edges: in the former case, $G[X,Y]=H$ has at least $|X|^2-O(|X|^{2-\alpha})$ edges while in the latter case, $G[X,Y]=H$ has at most $O(|X|^{2-\alpha})$ edges.
% Therefore, if $\widetilde{H}$ either has at most $k^{2-\alpha}/2$ edges, or at least $k^2-k^{2-\alpha}/2$ edges, then we are done. 
% Note that $e(\widetilde{H})=k^2-|X||Y|$, and $\|B-xy^T\|_F^2$ is the number of edges $\widetilde{H}$ differs from $H$. Therefore, if $e(\widetilde{H})\leq k^{2-\alpha}/2$, then $e(H)\leq e(\widetilde{H})+\|B-xy^T\|_F\leq k^{2-\alpha}/2 + O(k^{11/6+\gamma/3+o(1)}) \le k^{2-\alpha}$, so we are done. We can proceed similarly if $e(\widetilde{H})\geq k^2-k^{2-\alpha}/2$. 

We are left to show that $k^{2-\alpha}/2\le e(H')\leq k^2-k^{2-\alpha}/2$ is impossible.
In this case, since $e(H') = |X||Y|$, we know $k^{2-\alpha}/2\le |X||Y|\leq k^2-k^{2-\alpha}/2$.
This implies $|X|,|Y| \ge k^{1-\alpha}/2$.
In addition, $$
    k^2 = |X||Y| + (k-|X|)|Y| + |X|(k-|Y|) + (k-|X|)(k-|Y|) \le k^2-k^{2-\alpha}/2 + k(2k-|X|-|Y|),
$$
so $|\{1,\dots,2k\} \setminus (X\cup Y)|=2k-|X|-|Y| \ge k^{1-\alpha}/2$.
Pick $X_0\subseteq X, Y_0\subseteq Y, Z_0 \subseteq V(H) \setminus (X\cup Y)$ be sets of size exactly $k^{1-\alpha}/2$.
% In this case, as $k^2-k^{2-\alpha}/2\geq e(\widetilde{H})=k^2-|X||Y|$, we have $|X|,|Y|\geq k^{1-\alpha}/2$. Let $X_0\subseteq X, Y_0\subseteq Y$ be sets of size exactly $k^{1-\alpha}/2$. Also, as $e(\widetilde{H})\geq k^{2-\alpha}/2$, there are at least $k^{2-\alpha}/2$ missing edges from the complete bipartite graph between $X$ and $Y$. Since the number of missing edges is also at most $k(k-|X|)+k(k-|Y|)=k|V(G)\backslash (X\cup Y)|$, we find that $|V(G)\setminus (X\cup Y)|\geq k^{1-\alpha}/2$. Finally, let $Z_0\subseteq V(G)\backslash (X\cup Y)$ be a set of size $k^{1-\alpha}/2$.
Since the pairs $(X_0, Z_0)$ and $(Y_0, Z_0)$ are complete in $\overline{H'}$ while the pair $(X_0, Y_0)$ is empty in $\overline{H'}$, \Cref{lemma:dens2_3parts} guarantees $\surp(\overline{H'})\geq \frac{1}{4}|X_0|^2=k^{2-2\alpha}/16$. 
As discussed above, $H$ and $\overline{H'}$ differ by at most $\|B-xy^T\|_F^2=O(k^{11/6+5\gamma/3+o(1)})$ edges.
For sufficiently small $\gamma,\alpha>0$, we have 
$$\surp(H)\geq \surp(\overline{H'})-O(k^{11/6+5\gamma/3+o(1)})\geq \Omega(k^{2-2\alpha})>k^{1+5\gamma}.$$
This contradicts our assumption that $\surp(H) \le k^{1+5\gamma}$.
\end{proof}

Now we are ready to prove the main theorem of this section --- \Cref{thm:MAIN_MC2}. 
Recall that this theorem states: if $G$ has no cut of size $m/2+m^{1/2+\eps}$, then $G$ is $n^{-\alpha}$-close to a disjoint union of cliques. 
The proof is essentially the same as that of \Cref{thm:MAIN_SE2}, and therefore we only briefly outline it.

% \begin{theorem}\label{thm:stability_MaxCut}
% Let $\gamma\in (0,1/60)$, then there exists $\eps>0$ such that the following holds. Let $G$ be an $n$-vertex graph such that $\surp^*(G)\leq n^{1+\gamma}$. Then $G$ is $n^{-\alpha}$-close to the disjoint union of cliques.
% \end{theorem}
\begin{proof}[Proof of \Cref{thm:MAIN_MC2}.]
We may assume $m \ge n^{2-\alpha}$ since otherwise the statement is trivial.
Then, it suffices to show that for small enough constants $\alpha, \gamma>0$, whenever $\surp(G)\leq n^{1+\gamma}$, $G$ is $O(n^{-\alpha/2})$-close to a disjoint union of cliques. 

By Lemma~\ref{lemma:many_cliques_MaxCut}, there is a set $X\subset V(G)$ which can be partitioned into the union of cliques of size $n^{1-3\gamma}$, and $G$ has at most $n^{2-\alpha}$ edges not in $G[X]$. By Lemma~\ref{lemma:bipartite_MaxCut}, the bipartite graph between any two of these cliques is $n^{-\alpha/2}$-close to either complete or empty. Thus, one can define an auxiliary graph on these cliques, where two cliques are adjacent if the induced bipartite graph is almost complete. Due to Lemma~\ref{lemma:dens2_3parts}, the auxiliary graph has no induced cherries, meaning that it is a disjoint union of cliques.

Therefore, the graph $G[X]$ is $n^{2-\alpha/2}$-close to a disjoint union of cliques, and since $X$ misses at most $n^{2-\alpha}$ edges of $G$, the whole graph $G$ must also be polynomially close to a disjoint union of cliques.
\end{proof}

\section{Bisection width}

The \emph{bisection width} of a graph is defined as the minimum number of edges crossing a balanced partition of the vertex set, and it is denoted by $\bw(G)$.
As a natural dual to the maximum cut, this parameter is also of central interest in theoretical computer science \cite{HLRW,HRW,Karger}, probabilistic \cite{BNP,DDSW,DMS, DSW} and extremal graph theory \cite{Alon93,RST,RT24}.

It is convenient to measure the bisection width via the \emph{deficit}, which is defined as 
$$\dfc(G)=e(G)\left(\frac{1}{2}+\frac{1}{2n-2}\right)-\bw(G).$$ By the uniform random balanced cut, the deficit is always non-negative, and if $G$ is a regular graph that is neither empty nor complete, then $\dfc(G)=\Omega(n)$, see e.g. \cite{RST}. This is optimal if $G$ is a Tur\'an graph.

A classic result of Alon \cite{Alon93} states that if $G$ is $d$-regular, and $d=O(n^{1/9})$, then  $\dfc(G)=\Omega(\sqrt{d}n)$, which is optimal for random $d$-regular graphs. 
Recently, R\"aty, Sudakov, and Tomon \cite{RST} greatly extended this bound by showing that 
$$\dfc(G)=\begin{cases}\Omega(d^{1/2}n) &\mbox{ if } d\leq n^{2/3},\\
                        \Omega(n^2/d) &\mbox{ if } d\in [n^{2/3},n^{4/5}],\\
                        \widetilde{\Omega}(d^{1/4}n) &\mbox{ if } d\in [n^{4/5},(1/2-\eps)n].
            \end{cases}$$
These results are sharp for $d\in [1,n^{3/4}]$, and there are $d$-regular graphs for $d\approx n/3$ with deficit $O(n^{4/3})$. For $d=n(1-1/r)$, where $r$ is a positive integer, the Tur\'an graph $T_r(n)$ shows that we cannot hope for a bound better than $\Omega(n)$. R\"aty, Sudakov, and Tomon \cite{RST} conjectured that Tur\'an graphs are the only obstruction to large deficit. Using the terminology of \emph{positive discrepancy}, they conjectured that if $\dfc(G)=o(n^{5/4})$, then $G$ is $o(1)$-close to a Tur\'an graph. We prove that this conjecture holds qualitatively, by establishing the bisection width analogue of our MaxCut result (Theorem \ref{thm:MAIN_MC2}).

\begin{theorem}\label{thm:MAIN_BW1}
There exists $\eps>0$ such that the following holds for every sufficiently large $n$. Let $G$ be an $n$-vertex $d$-regular graph.
If the bisection width of $G$ is more than $\frac{dn}{4}-n^{1+\eps}$, then $G$ is $n^{-\eps}$-close to a Tur\'an graph. Thus, if $\dfc(G)\leq n^{1+\eps}$, then
$$\frac{d}{n}\in \left\{1-\frac{1}{r}:r\in \mathbb{Z}^+\right\}+[-n^{-\eps},n^{\eps}].$$ 
\end{theorem}

\begin{proof}
    Define the \emph{positive discrepancy} of a graph $G$ of edge density $p$ as 
    $$\disc^{+}(G)=\max_{U\subset V(G)} e(G[U])-p\binom{|U|}{2},$$
    and define the negative discrepancy as 
    $$\disc^{-}(G)=\max_{U\subset V(G)} p\binom{|U|}{2}-e(G[U]).$$
    It was proved in \cite[Lemma 2.6]{RST} that if $G$ is regular, then $\surp(G)=\Theta(\disc^{-}(G))$ and $\dfc(G)=\Theta(\disc^{+}(G))$. Moreover, $\disc^{+}(G)=\disc^{-}(\overline{G})$. 
    Therefore, the theorem follows from Theorem \ref{thm:MAIN_MC2} after taking complement of $G$, and noting that if a \emph{regular graph} $G$ is close to a complement of a disjoint union of cliques, then $G$ is close to a Tur\'an graph.
\end{proof}

\end{document}